\documentclass[10pt,a4paper,oneside,final]{article}
\usepackage{amsmath}
\usepackage{amssymb}
\usepackage{amsthm}
\usepackage{color}
\usepackage{fullpage}
\usepackage[pdftex]{graphicx}
\usepackage[british]{babel}

\newcommand{\R}{{\mathbb{R}}}

\newcommand{\N}{{\mathbb{N}}}
\newcommand{\de}{{\mathrm{d}}}

\newcommand{\dist}{{\mathrm{dist}}}
\newcommand{\Id}{\mathrm{Id}}
\newcommand{\setchar}[1]{\mathbf{1}_{#1}}
\newcommand{\bigtimes}{\mathop{\scalebox{2}{$\times$}}}
\newcommand{\lebesgue}{\mathcal{L}}
\newcommand{\lebesguecodimone}{\mathcal{L}^{n-1}}
\newcommand{\hd}{\mathcal{H}}
\newcommand{\hdone}{\mathcal{H}^1}
\newcommand{\hdtwo}{\mathcal{H}^2}
\newcommand{\hdndim}{\mathcal{H}^n}
\newcommand{\hdcodimone}{\mathcal{H}^{n-1}}
\newcommand{\rca}{\mathrm{rca}}
\newcommand{\fbm}{{\mathrm{fbm}}}
\newcommand{\pushforward}[2]{{{#1}_{\#}#2}}
\newcommand{\restr}{{\mbox{\LARGE$\llcorner$}}}
\newcommand{\spt}{{\mathrm{spt}}}

\newcommand{\weakstarto}{\stackrel{*}{\rightharpoonup}}
\newcommand{\Wd}[1]{\mathrm{W}_{#1}}
\newcommand{\Wdone}{\Wd{1}}
\newcommand{\AC}{\mathrm{AC}}

\newcommand{\cont}{{\mathcal{C}}}

\newcommand{\smooth}{\mathcal{C}_{c}^{\infty}}

\newcommand{\m}{{m}}

\newcommand{\reSpace}{\Gamma}
\newcommand{\reMeasure}{P_{\reSpace}}
\newcommand{\reMeasureindexed}[1]{P_{#1}}

\newcommand{\J}{{\mathcal{J}}}
\newcommand{\brTptEn}{{\mathcal{M}}}
\newcommand{\urbPlEn}{{\mathcal{E}}}

\newcommand{\MMSEn}{M_{\mathrm{P}}}

\newcommand{\urbPlMMS}{E_{\mathrm{P}}}
\newcommand{\urbPlBB}{{\mathcal E}_{\mathrm{W}}}
\newcommand{\repsilonachi}{r_{\varepsilon,a}^\chi}
\newcommand{\Bcal}{\mathcal{B}}

\newcommand{\Mcal}{\mathcal{M}}
\newcommand{\Ncal}{\mathcal{N}}

\newcommand{\numLev}{K}
\newcommand{\SigmaOpt}{{\Sigma^*}}

\newcommand{\chiOpt}{\chi^*}
\newcommand{\xiOpt}{\xi^*}

\newcommand{\keywords}[1]{\noindent\textbf{Keywords:}\enspace#1}
\newcommand{\subjclass}[1]{\bigskip\noindent\emph{2010 MSC:}\enspace#1}

\numberwithin{equation}{subsection}

\theoremstyle{plain}
\newtheorem{theorem}{Theorem}[subsection]
\newtheorem{lemma}[theorem]{Lemma}
\newtheorem{proposition}[theorem]{Proposition}
\newtheorem{corollary}[theorem]{Corollary}

\theoremstyle{definition}
\newtheorem{definition}[theorem]{Definition}

\theoremstyle{remark}
\newtheorem{remark}[theorem]{Remark}

\begin{document}

\title{Optimal micropatterns in transport networks}
\author{Alessio Brancolini\footnote{Institute for Numerical and Applied Mathematics, University of M\"unster, Einsteinstra\ss{}e 62, D-48149 M\"unster, Germany}\ \footnote{Email address: \texttt{alessio.brancolini@uni-muenster.de}} \and Benedikt Wirth\footnotemark[1]\ \footnote{Email address: \texttt{benedikt.wirth@uni-muenster.de}}}
\maketitle

\begin{abstract}
We consider two variational models for transport networks, an urban planning and a branched transport model,
in which the degree of network complexity and ramification is governed by a small parameter $\varepsilon>0$.
Smaller $\varepsilon$ leads to finer ramification patterns,
and we analyse how optimal network patterns in a particular geometry behave as $\varepsilon\to0$ by proving an energy scaling law.
This entails providing constructions of near-optimal networks
as well as proving that no other construction can do better.

The motivation of this analysis is twofold.
On the one hand, it provides a better understanding of the transport network models;
for instance, it reveals qualitative differences in the ramification patterns of urban planning and branched transport.
On the other hand, several examples of variational pattern analysis in the literature use an elegant technique based on relaxation and convex duality.
Transport networks provide a relatively simple setting to explore variations and refinements of this technique,
thereby increasing the scope of its applicability.

\bigskip\keywords{micropatterns, energy scaling laws, optimal transport, optimal networks, branched transport, irrigation, urban planning, Wasserstein distance}

\subjclass{49Q20, 49Q10, 90B10}
\end{abstract}

\section{Introduction}\label{sec:intro}
Pattern formation in physical experiments or in materials can sometimes be understood by means of so-called energy scaling laws.
Typically, such an experiment can be described by a physical energy which is minimised subject to some boundary conditions.
If the boundary conditions (or another global constraint) are incompatible with a homogeneous state,
a pattern of rapid spatial oscillations between energetically favourable states can often be observed,
which tries to recover compatibility with the boundary conditions.
While an infinitely fine oscillation could achieve full compatibility,
the physical energy usually prevents too fine oscillations via some regularising energy component
so that the resulting pattern is a balance between the objectives
of satisfying compatibility constraints and of keeping the regularising energy contribution small.

It is commonly impossible to find the truly optimal pattern, however,
if the regularising energy term has a small weight parameter $\varepsilon>0$,
one can instead try to prove how the minimum energy scales in $\varepsilon$, a so-called energy scaling law.
This involves proving a lower and an upper bound for the energy as a function of $\varepsilon$,
where both bounds differ at most by a multiplicative constant.
The benefit is that any pattern satisfying the upper bound must have optimal energy up to a constant factor
(and thus indicates how near-optimal patterns look like),
since the truly optimal energy still lies above the lower bound.
The proof of the upper bound is by constructing an appropriate pattern with the desired energy scaling,
while the proof of the lower bound, which shows that no construction can do better, is ansatz-free and thus more complicated
(however, in some cases a particular simple technique can be employed, variations of which are explored in this article).

Energy scaling laws have been derived for various systems such as martensite--austenite transformations \cite{KM92,KoMu94,KnKoOt13,ChCo14,BeGo14},
micromagnetics \cite{ChKoOt99}, intermediate states in type-I superconductors \cite{ChKoOt04,ChCoKo08},
membrane folding and blistering \cite{BeKo14,BeKo15}, or epitaxial growth \cite{GoZw14}.
Energy scaling laws may not only be used to explain observed physical patterns,
but they can also be applied to design problems \cite{KoWi14,KoWi14b},
for instance, to find efficient engineering structures
or to better understand why biological evolution has led to particular structures.
This article represents a further contribution to the application of these techniques to design problems.

The setting we shall look at is a type of network optimisation to be explained in more detail in the next two sections.
We consider two uniform measures $\mu_0,\mu_1$ in $\R^n$, of same mass and supported on hyper-squares of codimension one,
and we ask how the optimal transport network looks like to transport the mass in $\mu_0$ to $\mu_1$.
In more detail, let $A\subset\R^{n-1}$ be a hyper-square and consider
\begin{equation*}
\mu_0=\m\hdcodimone\restr(A\times\{0\})\,,\qquad
\mu_1=\m\hdcodimone\restr(A\times\{L\})
\end{equation*}
with $\m>0$ and $\hdcodimone$ the $(n-1)$-dimensional Hausdorff measure (see Figure\,\ref{fig:setting}).
Now $\mu_0$ represents an initial and $\mu_1$ the final distribution of particles,
and we seek the optimal, locally one-dimensional transport network to move all particles.
Optimality is with respect to a certain objective energy (in this article we consider a branched transport and an urban planning energy)
which effectively encodes that a lumped transportation of many particles together is cheaper than transporting each particle individually.
As a result, branched network structures as in Figure\,\ref{fig:setting} are favoured,
which balance the preference for lumping many particles together versus the restriction to evenly collect or distribute the particles on $\mu_0$ and $\mu_1$, respectively.
The degree of cost savings through lumped transportation will be described by a parameter $\varepsilon>0$,
and for $\varepsilon>0$ small, very fine network patterns will be optimal.
We will analyse the (near-)optimal patterns by proving corresponding energy scaling laws in $\varepsilon$.

\begin{figure}
\setlength{\unitlength}{\linewidth}
\begin{picture}(1,.49)
\put(0,.26){\includegraphics[height=.23\unitlength,trim=0 0 190 0,clip]{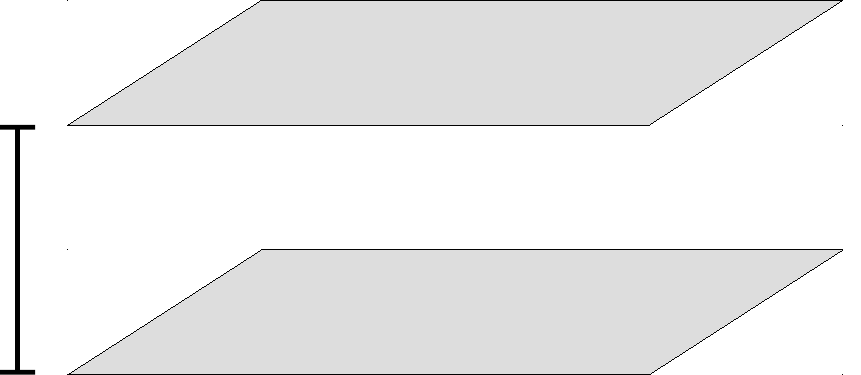}}
\put(.1,0){\includegraphics[height=.23\unitlength,trim=270 0 0 0,clip]{geometrySketch}}
\put(.015,.331){$L$}
\put(.23,.29){$\mu_0$}
\put(.23,.44){$\mu_1$}
\put(.26,.01){$\mu_0$}
\put(.26,.19){$\mu_1$}
\put(.7,.1){\includegraphics[height=.3\unitlength,angle=-90,origin=c]{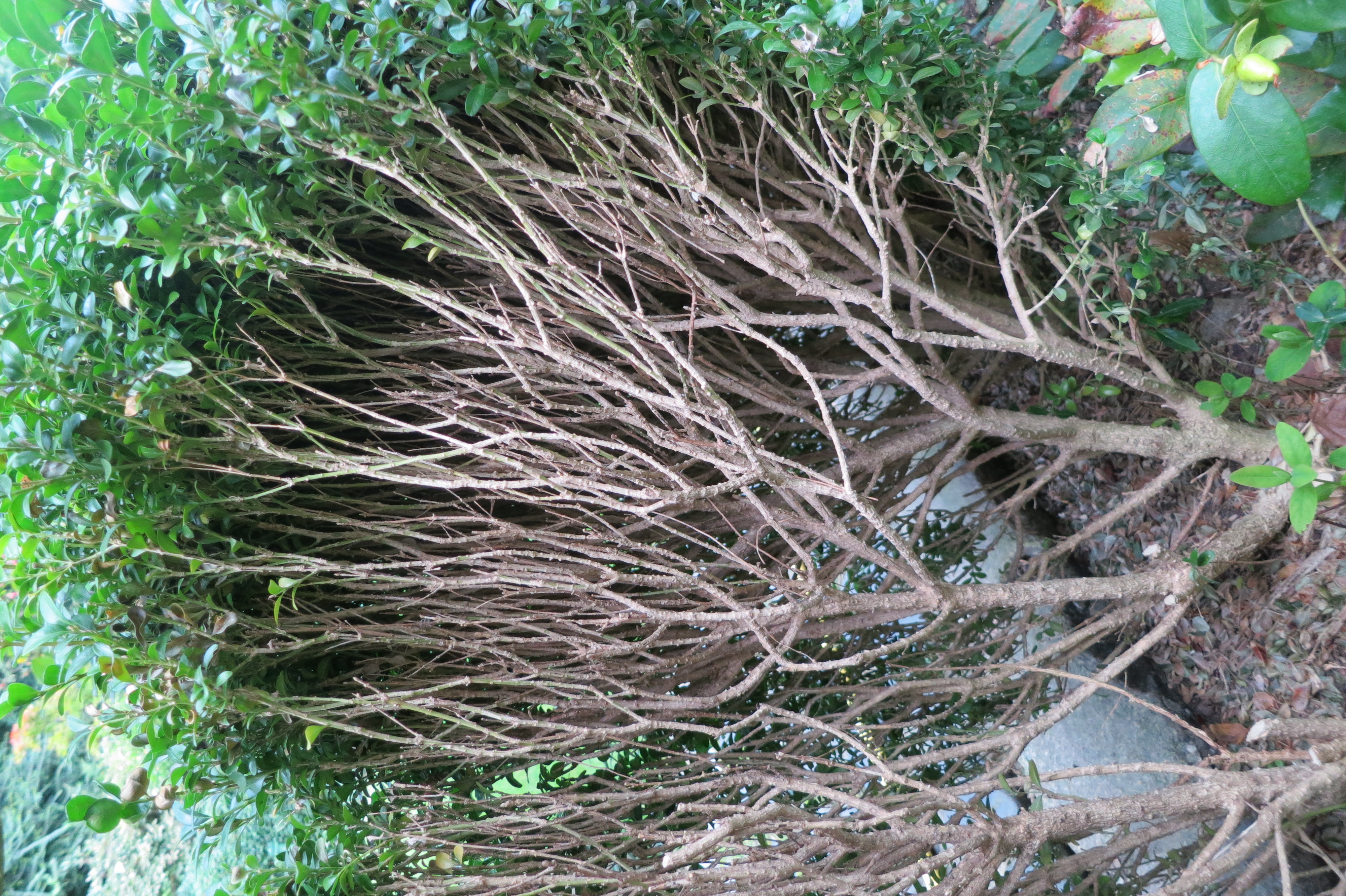}}
\end{picture}
\caption{Top left: Sketch of the considered geometry, two measures $\mu_0$ and $\mu_1$ supported on codimension-one hypersquares at distance $L$.
Bottom left: Exemplary transport network in two dimensions, transporting mass from $\mu_0$ to $\mu_1$.
Right: Photograph of a box tree exhibiting a branched network structure and a comparatively thin leaf canopy that may almost be thought of as a lower-dimensional manifold.}
\label{fig:setting}
\end{figure}

The motivation of this article and the goals pursued are multiple ones:
\begin{itemize}
\item The network models we examine and their generalisations have been used to model---among others---networks
in public transport, pipeline systems, river geometry, and biology such as blood vessel systems or water transport paths in plants
(see the extensive list of references in the monographs \cite{BeCaMo09,BuPrSoSt09}).
Due to the strong interest in these models it is worthwhile to better understand their behaviour.
Even though we consider a very special geometry for $\mu_0$ and $\mu_1$,
we will later argue that the same phenomena will be observed
for geometry variations such as varying the shape of the set $A$ or replacing $\mu_0$ by a delta distribution.
The measure support on a set of codimension one may be seen as an approximation of cases
where the extent of the initial or final mass distribution in the dominant transport direction
is considerably smaller than its transversal extent as well as the typical transport length.
Thinking of water transport to plant leaves, for instance,
a single leaf or the leaf canopy may often be abstracted as a two-dimensional manifold embedded in three-dimensional space (Figure\,\ref{fig:setting}).

\item Different network models have different mechanisms to give preference to branched network structures.
The question arises whether the models exhibit more or less the same behaviour and thus are phenomenologically equivalent
or whether one model can be told from the other just based on an observation of the network structure.
We will see that the urban planning and the branched transport model, which we consider,
share very similar features in three dimensions, but differ qualitatively in higher and lower dimensions.
Indeed, while the energy distribution in branched transport is always rather evenly spread across all branching structures except the ones at the boundary,
the major energy in urban planning is contributed by the coarsest structures in two dimensions and by the finest structures in dimensions four and higher.

\item The (locally) one-dimensional networks that we consider may be viewed as approximation of transport channels that are very thin yet have a positive width.
In that sense the models treated here can also be used to analyse structures
for the transport of a heat or a magnetic flux as examined in \cite{ChCoKo08,KoWi14}.
In mathematical terms, it is expected (and currently investigated \cite{Co14}) that for smaller and smaller fluxes,
the physical objective energy for conducting a heat or magnetic flux
$\Gamma$-converges in the right scaling against the optimal network energy.

\item The derivation of energy scaling laws is usually performed on a case to case basis,
and no generally applicable method or theory has been developed so far.
Instead, the field is still in the state of building and refining technique.
In this article we explore how far one can advance based on a very simple and transparent technique originally introduced by Kohn and M\"uller \cite{KoMu94}.
The method mainly relies on convex duality, and appropriate variations of it indeed suffice to almost comprehensively understand energy scaling in optimal transport networks.
\end{itemize}

The remainder of the introduction provides some general notation and a reference list of symbols to be used in the network models.
Section\,\ref{sec:transportModels} then introduces the network models considered in this article, including references to the relevant literature,
as well as the statement of the corresponding energy scaling laws Theorems\,\ref{thm:scalingUrbPlan} and \ref{thm:scalingBrnchTrpt}, the main results of this article.
Sections\,\ref{sec:upperBound} and \ref{sec:lowerBound} contain the proofs of the upper bounds and the lower bounds, respectively,
where both network models are treated separately.
We close with a discussion in Section\,\ref{sec:discussion}.

\subsection{Preliminaries: notation and useful notions}
Here we fix some frequently used basic notation.

\begin{itemize}
 \item \textbf{Lebesgue measure.} $\lebesgue^n$ denotes the $n$-dimensional \emph{Lebesgue measure}.

 \item \textbf{Hausdorff measure.} $\hd^r$ denotes the $r$-dimensional \emph{Hausdorff measure}.

 \item \textbf{Non-negative finite Borel measures.} $\fbm(X)$ denotes the set of \emph{non-negative finite Borel measures} on a Borel set $X\subset\R^n$. Notice that these measures are countably additive and also regular by \cite[Thm.\,2.18]{Ru87}. The corresponding total variation norm is denoted by $\|\cdot\|_\fbm$.

 \item \textbf{(Signed or vector-valued) regular countably additive measures.} $\rca(X)$ denotes the set of \emph{(signed or vector-valued) regular countably additive measures} on a Borel set $X\subset\R^n$. The corresponding total variation norm is denoted by $\|\cdot\|_\rca$.

 \item \textbf{Weak-$*$ convergence.} The weak-$*$ convergence on $\fbm(\R^n)$ or $\rca(\R^n)$ is indicated by $\weakstarto$.

 \item \textbf{Restriction of a measure to a set.} Let $(X,\Mcal,\mu)$ be a measure space and $Y\in\Mcal$. The measure $\mu\restr Y$ is the measure defined by
  \begin{displaymath}
  \mu\restr Y(A) = \mu(A \cap Y)\,.
  \end{displaymath}

 \item \textbf{Pushforward of a measure.} For a measure space $(X,\Mcal,\mu)$, a measurable space $(Y,\Ncal)$, and a measurable map $T : X \to Y$, the \emph{pushforward} of $\mu$ under $T$ is the measure $\pushforward{T}{\mu}$ on $(Y,\Ncal)$ defined by
       \begin{displaymath}
        \pushforward{T}{\mu}(B) = \mu(T^{-1}(B)) \quad \text{for all $B \in \Ncal$}.
       \end{displaymath}

 \item \textbf{Absolutely continuous functions.} $\AC(I)$ denotes the set of \emph{absolutely continuous functions} on the interval $I$.

 \item \textbf{Characteristic function of a set.} The \emph{characteristic function} of a set $A$ is defined as
       \begin{displaymath}
        \setchar{A}(x) = \begin{cases}
                          1 & x \in A,\\
                          0 & x \notin A.
                         \end{cases}
       \end{displaymath}

 \item \textbf{Dirac mass.} Let $x \in \R^n$. The \emph{Dirac mass} in $x$ is the distribution $\delta_x$ defined by
 \begin{displaymath}
 \langle \delta_x,\varphi \rangle = \varphi(x) \text{ for all } \varphi \in \smooth(\R^n)
 \end{displaymath}
 with $\smooth(\R^n)$ the compactly supported smooth functions on $\R^n$.
 Equivalently, it is the measure with $\delta_x(A)=1$ if $x\in A$ and $\delta_x(A)=0$ else.
\end{itemize}

In the upper and lower bound proofs we will furthermore employ the following abbreviations.

\begin{itemize}
 \item \textbf{Cross section.} For every $0 \leq t \leq 1$, the set $\{x_n=t\}=\{x\in\R^n\ :\ x_n=t\}$ will be called a \emph{cross-section}. We will also use analogous notations such as $\{x_n<t\}$ with the obvious meaning.

 \item \textbf{Projection on the hyperplane $\{x_n = 0\}$.} By $x'$ we mean the \emph{projection} of $x = (x_1,\ldots,x_n)$ on the hyperplane $\{x_n = 0\}$, $x'=(x_1,\ldots,x_{n-1},0)$.

 \item \textbf{Tubular neighbourhood.} $B_s(U)$ is the \emph{tubular neighbourhood} of radius $s$ of the set $U\subset\R^n$,
       \begin{displaymath}
        B_s(U) = \{x \in \R^n \ : \ \inf_{y \in U} |x - y| \leq s\}.
       \end{displaymath}

 \item $C\equiv C(n)$ identifies $C$ as a constant depending only on the dimension $n$; if the constant depends on other parameters as well, an analogous notation is used.

 \item When we write $A \lesssim B$ or $A\gtrsim B$, we mean that there exists a constant $C\equiv C(n)$ such that $A \leq C B$ or $B \leq C A$, respectively. $A\sim B$ stands for $A \lesssim B$ and $A\gtrsim B$.

 \item $\omega_n$ denotes the $n$-dimensional measure $\lebesgue^n(B_1(0))$ of the unit ball in $\R^n$.

\end{itemize}

Next we introduce briefly the \emph{Wasserstein distance} and \emph{Kantorovich-Rubinstein duality}. We refer to \cite[Chap.\,1]{Villani-Topics-Optimal-Transport} for a complete account on Wasserstein distances, spaces and their properties.

\begin{definition}[Wasserstein distance]\label{def:Wasserstein}
 Let $\mu_+,\mu_-$ be finite Borel measures with equal mass $\|\mu_+\|_\fbm=\|\mu_-\|_\fbm$. Given $p \geq 1$, the \emph{$p$\textsuperscript{th} Wasserstein distance} between $\mu_+$ and $\mu_-$ is
 \begin{displaymath}
  \Wd{p}(\mu_+,\mu_-) = \left(\inf_{\mu\in\Pi(\mu_+,\mu_-)} \int_{\R^n\times\R^n} |x - y|^p \de\mu(x,y)\right)^{\frac1p},
 \end{displaymath}
 where $\Pi(\mu_+,\mu_-) = \{\mu\in\fbm(\R^n\times\R^n) \ : \ \mu(A \times \R^n) = \mu_+(A),\ \mu(\R^n \times B) = \mu_-(B)\text{ for all Borel sets }A,B\subset\R^n\}$.
\end{definition}

We recall the following theorem, which relates the optimal transport problem with its dual formulation (see \cite[Thm.\,1.14]{Villani-Topics-Optimal-Transport}).

\begin{theorem}[Kantorovich-Rubinstein duality]\label{thm:WoneDual}
 Let $\mu_+,\mu_- \in \fbm(\R^n)$ with equal mass. Then
 \begin{displaymath}
  \Wdone(\mu_+,\mu_-) = \sup \left\{\int_{\R^n} \varphi\,\de(\mu_+-\mu_-) \ : \ \varphi:\R^n\to\R\text{ Lipschitz with constant }1\right\}.
 \end{displaymath}
\end{theorem}

Finally, for the reader's convenience we compile here a reference list of the most important symbols with references to the corresponding definitions.

\begin{itemize}
 \item $I = [0,1]$: The unit interval.

 \item $d_\Sigma$: Urban planning transport metric (see Formula \eqref{eq:d_Sigma}).

 \item $(\reSpace,\Bcal(\reSpace),\reMeasure)$: Reference space of all particles (Definition\,\ref{def:reference_space}).

 \item $\chi$: Irrigation pattern of all particles (Definition\,\ref{def:irrigation_pattern}).

 \item $[x]_\chi$: Solidarity class of $x$ (Definition\,\ref{def:solidarity_classes}).

 \item $\mu_+^\chi, \mu_-^\chi$: Irrigating and irrigated measure (Definition \ref{def:irrigation}).

 \item $t_p(s)$ and $\xiOpt(p,s)$: Crossing time and crossing point for particle $p$ and cross-section $\{x_n = s\}$ (see Formulae \eqref{eq:crossing_time} and \eqref{eq:crossing_point}).
\end{itemize}

\section{Reminder of transport networks and energy scaling}\label{sec:transportModels}
In this section we introduce different formulations of branched transport and urban planning which prove beneficial in deriving energy scaling laws. 

\subsection{Irrigation patterns}

As a preparation we here recall the Lagrangian or pattern-based formulation of particle transport (see \cite{Maddalena-Morel-Solimini-Irrigation-Patterns}, \cite{Bernot-Caselles-Morel-Traffic-Plans}, \cite{Maddalena-Solimini-Synchronic}).

\begin{definition}[Reference space]\label{def:reference_space}
Consider a complete separable uncountable metric space $\reSpace$ endowed with the $\sigma$-algebra $\Bcal(\reSpace)$ of its Borel sets and a positive finite Borel measure $\reMeasure$ with no atoms. We refer to $(\reSpace,\Bcal(\reSpace),\reMeasure)$ as the \emph{reference space}.
\end{definition}

The reference space can be interpreted as the space of all particles that will be transported from a distribution $\mu_+$ to a distribution $\mu_-$.

\begin{remark}
Any reference space can be shown to be isomorphic to the \emph{standard space} $([0,1],\Bcal([0,1]),m\lebesgue^1 \restr [0,1])$ with $m=\reMeasure(\reSpace)$
(see \cite[Prop.\,12 or Thm.\,16 in Sec.\,5 of Chap.\,15]{Royden-Real-Analysis} or \cite[Chap.\,1]{Villani-Transport-Old-New} for a proof).
As a consequence, the following definitions and results are independent of the particular choice of the reference space,
and we will just assume it to be the standard space unless stated otherwise.
\end{remark}

\begin{definition}[Irrigation pattern]\label{def:irrigation_pattern}
Let $I = [0,1]$ and $(\reSpace,\Bcal(\reSpace),\reMeasure)$ be our reference space. An \emph{irrigation pattern} is a measurable function $\chi : \reSpace \times I \to \R^n$ such that for almost all $p\in\reSpace$ we have $\chi_p \in \AC(I)$.

A pattern $\tilde\chi$ is \emph{equivalent} to $\chi$ if the images of $\reMeasure$ through the maps $p \mapsto \chi_p, p \mapsto \tilde\chi_p$ are the same.

For intuition, $\chi_p$ can be viewed as the path of particle $p$. Its image $\chi_p(I)$ is called a \emph{fibre} and will frequently be identified with the particle $p$.
\end{definition}
Here we followed the setting recently introduced in \cite{Maddalena-Solimini-Synchronic}.

\begin{definition}[Irrigating and irrigated measure]\label{def:irrigation}
Let $\chi$ be an irrigation pattern. Let $i_0^\chi,i_1^\chi:\reSpace \to \R^n$ be defined as $i_0^\chi(p) = \chi(p,0)$ and $i_1^\chi(p) = \chi(p,1)$.
The \emph{irrigating measure} and the \emph{irrigated measure} are defined as the pushforward of $\reMeasure$ via $i_0^\chi$ and $i_1^\chi$, respectively: $\mu_+^\chi = \pushforward{(i_0^\chi)}{\reMeasure}$, $\mu_-^\chi = \pushforward{(i_1^\chi)}{\reMeasure}$.
\end{definition}

\begin{definition}[Solidarity class]\label{def:solidarity_classes}
For every $x\in\R^n$, the set of all particles flowing through $x$ is denoted $$[x]_\chi = \{q \in \reSpace \ : \ x \in \chi_q(I)\}$$
and its mass by $m_\chi(x) = \reMeasure([x]_\chi)\,.$
\end{definition}

\subsection{An urban planning model}\label{sec:introUrbPlan}
As our first model of transport networks we consider the so-called urban planning model from \cite{Brancolini-Buttazzo}.
Its original motivation is to optimise the public transport network for the daily commute of employees
from their homes (described by a spatial distribution $\mu_0\in\fbm(\R^n)$) to their workplaces (described by $\mu_1$).

The public transport network is represented by a rectifiable set $\Sigma\subset\R^n$,
and a person travelling along a path $\theta\subset\R^n$ has expenses
\begin{equation*}
a\hdone(\theta\setminus\Sigma)+b\hdone(\theta\cap\Sigma)\,,
\end{equation*}
where $a$ and $b<a$ denote the cost per travelling distance outside and on the network, respectively.
The minimum cost to get from $x\in\R^n$ to $y\in\R^n$ is thus given by
\begin{equation}\label{eq:d_Sigma}
d_\Sigma(x,y)=\inf \{a\hdone(\theta\setminus\Sigma)+b\hdone(\theta\cap\Sigma) \ : \ \theta\in C_{x,y}\}
\end{equation}
for $C_{x,y} = \{\theta : [0,1] \to \R^n \ : \ \theta \ \text{is Lipschitz with} \ \theta(0) = x, \ \theta(1) = y\}$.
The objective functional describing the cost efficiency of the network $\Sigma$ is now given as
\begin{equation*}
\urbPlBB^{\varepsilon,a,b,\mu_0,\mu_1}[\Sigma]=\Wd{d_\Sigma}(\mu_0,\mu_1)+\varepsilon\hdone(\Sigma)\,,
\end{equation*}
where $\varepsilon\hdone(\Sigma)$ models the maintenance cost for the network and the Wasserstein-type distance
\begin{equation*}
\Wd{d_\Sigma}(\mu_0,\mu_1)=\inf_{\mu\in\Pi(\mu_0,\mu_1)}\int_{\R^n\times\R^n}d_\Sigma(x,y)\,\de\mu(x,y)
\end{equation*}
with $\Pi(\mu_0,\mu_1)=\left\{\mu\in\fbm(\R^n\times\R^n)\,:\,\pushforward{\pi_1}\mu=\mu_0,\pushforward{\pi_2}\mu=\mu_1\right\}$
represents the total travel expenses of the population travelling from their homes $\mu_0$ to their workplaces $\mu_1$.
Due to $\urbPlBB^{\varepsilon,a,b,\mu_0,\mu_1} = b\urbPlBB^{\frac{\varepsilon}{b},\frac{a}{b},1,\mu_0,\mu_1}$ it actually suffices to study the case $b=1$.

In \cite{BW15} it has been shown that the urban planning model can equivalently be reformulated in terms of irrigation patterns $\chi$ as detailed below.
The optimal transport network $\SigmaOpt$ can then be extracted from the optimal irrigation pattern $\chiOpt$ via $\SigmaOpt=\{x\in\R^n\,:\,m_\chiOpt(x)>\frac\varepsilon{a-1}\}$.
The pattern formulation is advantageous to prove energy scaling laws, which is why we introduce it here.

\begin{definition}[Urban planning problem, pattern formulation]\label{def:urbPlPatternForm}
Let $(\reSpace,\Bcal(\reSpace),\reMeasure)$ be the reference space and $\chi : \reSpace \times [0,1] \to \R^n$ be an irrigation pattern. For $\varepsilon > 0$ and $a>1$, consider the cost density
\begin{equation*}
\repsilonachi(x)= \begin{cases}
                   \min\left\{1+\tfrac\varepsilon{m_\chi(x)},a\right\} & \text{ if } m_\chi(x) > 0,\\
                   a & \text{ if } m_\chi(x) = 0.
                  \end{cases}
\end{equation*}
The \emph{urban planning cost functional} $\urbPlMMS^{\varepsilon,a}$ is given by
\begin{equation*}
\urbPlMMS^{\varepsilon,a}(\chi)=\int_{\reSpace\times I}\repsilonachi(\chi_p(t))|\dot\chi_p(t)|\,\de \reMeasure(p)\,\de t\,.
\end{equation*}
For $\mu_0,\mu_1 \in \fbm(\R^n)$, the \emph{urban planning problem} is
\begin{equation*}
 \min \urbPlEn^{\varepsilon,a,\mu_0,\mu_1}[\chi],
\end{equation*}
where
\begin{displaymath}
  \urbPlEn^{\varepsilon,a,\mu_0,\mu_1}[\chi] = \begin{cases}
                                                \urbPlMMS^{\varepsilon,a}(\chi)&\text{if $\mu_+^\chi = \mu_0$ and $\mu_-^\chi = \mu_1$},\\
                                                \infty&\text{else.}
                                               \end{cases}
\end{displaymath}
\end{definition}

Existence of minimising patterns is shown in \cite{BW15} as well as the equivalence of the optimisation problems
$\min_\chi\urbPlEn^{\varepsilon,a,\mu_0,\mu_1}[\chi]=\min_\Sigma\urbPlBB^{\varepsilon,a,1,\mu_0,\mu_1}[\Sigma]$.
Note that $\varepsilon$ governs the network maintenance costs.
Thus, a large $\varepsilon$ will lead to coarse networks, while a small $\varepsilon$ allows the network to become very complex.
Our energy scaling law will be concerned with the latter case.

Finally, we will later need the following reparameterisation result from \cite{BW15}.

\begin{proposition}[Constant speed reparameterisation of patterns]\label{prop:constSpeedPatternsUrbPl}
Irrigation patterns of finite cost can be reparameterised such that $\chi_p:I\to\R^n$ is Lipschitz and $|\dot\chi_p|$ is constant for almost all $p\in\reSpace$ without changing the cost $\urbPlMMS^{\varepsilon,a}$.
\end{proposition}

\subsection{Branched transport}\label{subsec:branched_transport}
The second model considered is the \emph{branched transportation problem}. In the literature, the branched transport functional is commonly parameterised by a parameter $\alpha\in[0,1)$,
which describes the economies of scale for mass transport:
The cost for transporting mass $m$ by a unit distance is $m^\alpha$, so the cost per single mass particle gets smaller the more particles are transported together.
This creates a preference for branching networks in which mass is first lumped together and then transported in bulk.
The closer $\alpha$ to one, the less pronounced is this preference and the more complex the transportation networks can become.
Since we will study the asymptotic behaviour of finer and finer networks, we will in our description replace $\alpha$ by $\varepsilon = 1-\alpha$ and consider the case of $\varepsilon$ small.
The following definition of the functional $\MMSEn^\varepsilon$ follows \cite{Bernot-Caselles-Morel-Traffic-Plans}.

\begin{definition}[Branched transport problem]\label{eqn:costDensityBrTpt}
For $0 < \varepsilon \leq 1$, the reference space $(\reSpace,\Bcal(\reSpace),\reMeasure)$, and an irrigation pattern $\chi : \reSpace \times [0,1] \to \R^n$ we consider the \emph{cost density} \begin{equation*}
s_\varepsilon^\chi(x)=\begin{cases}
                                                           m_\chi(x)^{-\varepsilon} & \text{ if } m_\chi(x) > 0,\\
                                                           \infty & \text{ if } m_\chi(x) = 0.
                                                          \end{cases}
\end{equation*}
The \emph{branched transport cost functional} associated with irrigation pattern $\chi$ is
\begin{displaymath}
\MMSEn^\varepsilon(\chi) = \int_{\reSpace \times I} s_\varepsilon^\chi(\chi_p(t)) |\dot\chi_p(t)|\,\de\reMeasure(p)\, \de t\,.
\end{displaymath}
For $\mu_0,\mu_1 \in \fbm(\R^n)$, the \emph{branched transport problem} is
\begin{equation*}
 \min \brTptEn^{\varepsilon,\mu_0,\mu_1}[\chi],
\end{equation*}
where
\begin{displaymath}
 \brTptEn^{\varepsilon,\mu_0,\mu_1}[\chi] = \begin{cases}
                                             \MMSEn^\varepsilon(\chi)&\text{if $\mu_+^\chi = \mu_0$ and $\mu_-^\chi = \mu_1$},\\
                                             \infty&\text{else.}
                                            \end{cases}
\end{displaymath}
\end{definition}

Given $\mu_0,\mu_1 \in \fbm(\R^n)$ with compact support, an optimal irrigation pattern exists \cite{Maddalena-Solimini-Transport-Distances}.

Finally, we will need the fact from \cite{BW15} that patterns can be reparameterised fibrewise without changing the cost.

\begin{proposition}[Constant speed reparameterisation of patterns]\label{prop:reparameterised_patterns_have_the_same_cost}
Irrigation patterns of finite cost can be reparameterised such that $\chi_p:I\to\R^n$ is Lipschitz and $|\dot\chi_p|$ is constant for almost all $p\in\reSpace$ without changing the cost $\MMSEn^\varepsilon$.
\end{proposition}

\subsection{The optimal energy scaling}\label{sec:energyScaling}

Recall from Section \ref{sec:intro} that for $A \subset \R^{n-1}$ a hypersquare we will use
\begin{equation*}
\mu_0=\m\hdcodimone\restr(A\times\{0\})\,,\qquad
\mu_1=\m\hdcodimone\restr(A\times\{L\})
\end{equation*}
Therefore we set
$
 \urbPlEn^{\varepsilon,a,A,\m,L}=\urbPlEn^{\varepsilon,a,\mu_0,\mu_1}
 \ \text{and}\
 \brTptEn^{\varepsilon,A,\m,L}=\brTptEn^{\varepsilon,\mu_0,\mu_1}\,.
$
By rescaling the geometry and the mass flux one easily arrives at the following non-dimensionalisations,
\begin{align*}
\urbPlEn^{\varepsilon,a,A,\m,L}[\chi]
&=L^n\m\urbPlEn^{\frac\varepsilon{L^{n-1}\m},a,\frac1LA,1,1}[\tfrac1L\chi]\,,\\
\brTptEn^{\varepsilon,A,\m,L}[\chi]
&=(L^n\m)^{1-\varepsilon}L\brTptEn^{\varepsilon,\frac1LA,1,1}[\tfrac1L\chi]\,,
\end{align*}
so that without loss of generality we may restrict ourselves to the study of
\begin{equation*}
\mu_0=\hdcodimone\restr(A\times\{0\})\,,\qquad
\mu_1=\hdcodimone\restr(A\times\{1\})\,,
\end{equation*}
\begin{equation*}
 \urbPlEn^{\varepsilon,a,A}:=\urbPlEn^{\varepsilon,a,A,1,1}
 \qquad\text{and}\qquad
 \brTptEn^{\varepsilon,A}:=\brTptEn^{\varepsilon,A,1,1}\,.
\end{equation*}
Furthermore, let us denote the minimum energy for $\varepsilon=0$ by
\begin{equation*}
\urbPlEn^{*,a,A}=\inf_{\chi}\urbPlEn^{0,a,A}[\chi]\,,\qquad
\brTptEn^{*,A}=\inf_{\chi}\brTptEn^{0,A}[\chi]\,.
\end{equation*}
We have 
\begin{equation*}
\urbPlEn^{*,a,A}=\brTptEn^{*,A}=\Wdone(\mu_0,\mu_1)=\hdcodimone(A)
\end{equation*}
for $\Wdone(\mu_0,\mu_1)=\inf_{\mu\in\Pi(\mu_0,\mu_1)}\int_{\R^n\times\R^n}|x-y|\,\de\mu(x,y)$ the Wasserstein distance.
Indeed, our constructions in Section\,\ref{sec:upperBound} will imply $\brTptEn^{*,A},\urbPlEn^{*,a,A}\leq\Wdone(\mu_0,\mu_1)$,
while Sections\,\ref{sec:lwBndUrbPlan} and \ref{sec:lwBndBrnchTrpt} begin with calculations showing $\brTptEn^{*,A},\urbPlEn^{*,a,A}\geq\Wdone(\mu_0,\mu_1)$.
For now, the reader may simply imagine the infimum value being approached by irrigation patterns $\chi:\reSpace\times I\to\R^n$ with $\reSpace=A$ approximating $\chi_p(t)=(p_1,\ldots,p_{n-1},t)^T$, transport through a network with infinitely many vertical pipes of length 1.

For $\varepsilon>0$ the minimum achievable energy deviates from $\brTptEn^{*,A}$ and $\urbPlEn^{*,a,A}$, respectively.
We will call this deviation the excess energy and abbreviate it by
\begin{equation*}
\Delta\urbPlEn^{\varepsilon,a,A}=\min_\chi\urbPlEn^{\varepsilon,a,A}[\chi]-\urbPlEn^{*,a,A}
\qquad\text{and}\qquad
\Delta\brTptEn^{\varepsilon,A}=\min_\chi\brTptEn^{\varepsilon,A}[\chi]-\brTptEn^{*,A}\,.
\end{equation*}
The main results of this article are the following two theorems, which show how the deviation scales in $\varepsilon$.
The first theorem treats the urban planning model.

\begin{theorem}[Energy scaling for the Urban Planning]\label{thm:scalingUrbPlan}
There are constants $C_1,C_2>0$ independent of $\varepsilon,a,A$
such that for $\varepsilon<\min\{1,\hdcodimone(A)^{\frac{n+1}{n-1}}\}$ and $a>1$ we have
\begin{equation*}
C_1\hdcodimone(A)\min\{a-1,f(\varepsilon,a)\}
\leq\Delta\urbPlEn^{\varepsilon,a,A}
\leq C_2\hdcodimone(A)\min\{a-1,f(\varepsilon,a)\}\,,
\end{equation*}
where the function $f$ is given by
\begin{align*}
f(\varepsilon,a)&=\varepsilon^{\frac23}&\text{ if }n=2\,,\\
f(\varepsilon,a)&=(\sqrt a+|\log\tfrac{a-1}{\sqrt\varepsilon}|)\sqrt\varepsilon&\text{ if }n=3\,,\\
f(\varepsilon,a)&=\sqrt a(\sqrt{a-1})^{\frac{n-3}{n-1}}\varepsilon^{\frac1{n-1}}&\text{ if }n>3\,.
\end{align*}
\end{theorem}

The proof is given in Sections\,\ref{sec:upBndUrbPlan} and \ref{sec:lwBndUrbPlan}.
Note that the dependence on the parameter $a$ is also identified, which plays a dominant role if either $a$ is huge or close to $1$.
The three different scaling laws for different dimensions indicate three different regimes, as will become clear in Sections\,\ref{sec:upBndUrbPlan} and \ref{sec:lwBndUrbPlan}:
\begin{itemize}
\item In 2D, the dominant contribution to the excess energy stems from the interior region
between and bounded away from $\mu_0$ and $\mu_1$.
\item In 3D, the dominant energy contribution also occurs in between the two measures,
but this time it is more evenly distributed, up to the boundary.
\item In higher dimensions, the energy scaling is dominated by energy contributions
concentrated at $\mu_0$ and $\mu_1$.
\end{itemize}

The second result is concerned with the branched transport model.
\begin{theorem}[Energy scaling for the Branched Transport]\label{thm:scalingBrnchTrpt}
There are constants $C_1,C_2,C_3,C_4>0$ independent of $\varepsilon$ and $A$
such that for $\varepsilon<\min\{\hdcodimone(A)^{\frac2{n-1}},\hdcodimone(A)^{-C_3},C_4\}$ we have
\begin{equation*}
C_1\hdcodimone(A)\varepsilon|\log\varepsilon|
\leq\Delta\brTptEn^{\varepsilon,A}
\leq C_2\hdcodimone(A)\varepsilon|\log\varepsilon|\,.
\end{equation*}
\end{theorem}

The proof is given in Sections\,\ref{sec:upBndBrnchTrpt} and \ref{sec:lwBndBrnchTrpt} and reveals that the total energy distribution follows a similar pattern as in 3D urban planning.

All above energy scalings will be obtained following variations of a general scheme
originally introduced by Kohn and M\"uller in 1992 \cite{KM92} to model martensite--austenite interfaces.
\begin{itemize}
\item The upper bound is based on a construction with unit cells arranged in hierarchical layers.
The unit cells of the layer closest to the domain boundary are the smallest,
and from layer to layer the unit cell width doubles, thereby leading to a coarsening towards the domain centre.
Each single unit cell contains a branching structure compatible with the width doubling between two layers
(in Figure\,\ref{fig:setting}, bottom left, for instance, a unit cell is a `V'-shaped structure).
The free design parameters are the unit cell widths on the coarsest and the finest level
as well as the aspect ratio of the unit cells on each level.
Now the excess energy due to $\varepsilon>0$ can be computed
for each unit cell as a function of the design parameters,
and their sum yields the total excess energy.
Minimisation for the design parameters yields the optimal construction and energy scaling.
The detailed procedure is provided in Sections\,\ref{sec:upBndUrbPlan} and \ref{sec:upBndBrnchTrpt}.
\item The technique for the lower bound exploits the fact
that the transport costs between any two measures for $\varepsilon>0$
can be bounded below by the transport costs for $\varepsilon=0$,
which can be calculated by solving a (much simpler) convex problem.
In essence, one first characterises a generic cross-section
in terms of how many interfaces or pipes it intersects at most.
The number depends on $\varepsilon$ and usually is a simple consequence of the excess energy being bounded above.
Then, the energy can be bounded below by the costs one would have for $\varepsilon=0$,
knowing the additional information about the generic cross-section.
Those costs can be obtained by solving a \emph{convex} optimisation problem due to $\varepsilon=0$.
This idea in its plainest form is exemplified in Section\,\ref{sec:lwBndUrbPlan} for dimension $n=2$,
and variants of it are applied in Section\,\ref{sec:lwBndUrbPlan} for $n\geq3$ and in Section\,\ref{sec:lwBndBrnchTrpt}.
\end{itemize}

\section{Upper bound via branching construction}\label{sec:upperBound}

In this section we will provide constructions of transport networks with the required energy scaling.
Those constructions will be composed of elementary units, all described as an irrigation pattern,
for which reason we first define several operations how multiple irrigation patterns can be combined to a single one in Section\,\ref{sec:patternGluing}.
Afterwards we define and analyse the single elementary units in Section\,\ref{sec:ElementaryCells}
as well as the optimal constructions for urban planning (Section\,\ref{sec:upBndUrbPlan}) and branched transport (Section\,\ref{sec:upBndBrnchTrpt}).

\subsection{Operations with patterns}\label{sec:patternGluing}

Our constructions will require a combination of irrigation patterns in parallel as well as in series.
Those operations are defined in Definition\,\ref{def:union_of_patterns} and \ref{def:series_of_patterns}, respectively.
Definition \ref{def:union_of_patterns} was introduced in \cite{Bernot-Caselles-Morel-Structure-Branched} and has also been given in an equivalent form in \cite[Lemma 5.5]{BeCaMo09}.
Here the particles of all constituting subpatterns are kept independent.

\begin{definition}[Union of patterns]\label{def:union_of_patterns}
Let $\chi_i : \reSpace_i \times I \to \R^n$, $i\in N\subset\N$, be a sequence of irrigation patterns with reference spaces $(\reSpace_i,\Bcal(\reSpace_i),\reMeasureindexed{\reSpace_i})$.
Suppose that $\sum_i \reMeasureindexed{\reSpace_i}(\reSpace_i) < +\infty$. Let $\reSpace$ be the disjoint union of the $\reSpace_i$,
\begin{displaymath}
 \reSpace = \coprod_{i \in N} \reSpace_i,
\end{displaymath}
endowed with the product topology and the corresponding Borel $\sigma$-algebra. Furthermore consider the measure $\reMeasure$ defined by
\begin{displaymath}
 \reMeasure(A) = \sum_{i \in N} \reMeasureindexed{\reSpace_i}(A \cap \reSpace_i).
\end{displaymath}
The irrigation pattern $\chi : \reSpace \times I \to \R^n$, defined by
\begin{displaymath}
 \chi(p,t) = \chi_i(p,t), \ \text{if} \ p \in \reSpace_i,
\end{displaymath}
is called the \emph{union} of the patterns $\chi_i$ and will be denoted by $\chi = \amalg_i \chi_i$.
\end{definition}

Next we join two patterns $\chi_1$ and $\chi_2$ in series.
Here, every particle $p$ first travels along a path defined by $\chi_1$ and then continues along a path defined by $\chi_2$.

\begin{definition}[Series of patterns]\label{def:series_of_patterns}
Let $\chi_i : \reSpace_i \times I \to \R^n$ for $i = 1,2$ be two irrigation patterns
and $T:\reSpace_1\to\reSpace_2$ an isomorphism of measure spaces such that
\begin{displaymath}
\chi_1(p,1)=\chi_2(T(p),0)\qquad\text{for $\reMeasure$-almost all }p\in\reSpace_1\,.
\end{displaymath}
Now define the irrigation pattern $\tilde\chi: \reSpace_1 \times I \to \R^n$ by
\begin{displaymath}
 \tilde\chi(p,t) = \begin{cases}
              \chi_1(p,2t) & t \in [0,\frac{1}{2}],\\
              \chi_2(T(p),2t-1) & t \in [\frac{1}{2},1].
             \end{cases}
\end{displaymath}
Furthermore, for $p\in\reSpace_1$ let $r_p:I\to I$ be a reparameterisation such that $\chi(p,t)=\tilde\chi(p,r_p(t))$ has constant speed $|\dot\chi_p|$ for almost all $p\in\reSpace_1$.
The irrigation pattern $\chi$ is called the $T$-relative \emph{series} of patterns $\chi_1$ and $\chi_2$ and is denoted by $\chi_1 \circ_T \chi_2$.
\end{definition}

\begin{remark}
A similar construction was given in \cite[Lemma 5.5]{BeCaMo09}.
Ours is a little more technical since we track the dependence on $T$,
while in \cite{BeCaMo09} the authors just require $\mu_-^{\chi_1} = \mu_+^{\chi_2}$
and from this anticipate the existence of some measure-preserving map $T:\reSpace_1\to\reSpace_2$ with $\chi_1(p,1)=\chi_2(T(p),0)$,
which is then used in the construction.
However, such a map $T$ may not exist, as shown by the example $\reSpace_1=\reSpace_2=[0,1]$ with $\reMeasure=\lebesgue^1\restr[0,1]$, $\chi_1(p,1)=2p\,\mathrm{mod}\,1$, $\chi_2(p,0)=p$.
The result in \cite{BeCaMo09} certainly stays valid, but its proof will require to relax the transport map $T$ between particles to a transport plan.
To circumvent those difficulties we shall simply always specify the map $T$ explicitly.
\end{remark}

\begin{remark}\label{rem:cellOpProps}
It is straightforward to verify the following elementary properties,
in which $\J$ either stands for the urban planning energy $\urbPlMMS^{\varepsilon,a}$ or the branched transport energy $\MMSEn^\varepsilon$,
\begin{itemize}
 \item $\J(\chi_1\amalg\chi_2) \leq \J(\chi_1) + \J(\chi_2)$,
 \item $\J(\chi_1 \circ_T \chi_2) \leq \J(\chi_1) + \J(\chi_2)$,
 \item $\mu_+^{\chi_1 \amalg \chi_2} = \mu_+^{\chi_1}+\mu_+^{\chi_2}$, $\mu_-^{\chi_1 \amalg \chi_2} = \mu_-^{\chi_1}+\mu_-^{\chi_2}$,
 \item $\mu_+^{\chi_1 \circ_T \chi_2} = \mu_+^{\chi_1}$, $\mu_-^{\chi_1 \circ_T \chi_2} = \mu_-^{\chi_2}$,
\end{itemize}
where the second inequality uses Propositions \ref{prop:constSpeedPatternsUrbPl} and \ref{prop:reparameterised_patterns_have_the_same_cost} and
where both inequalities become equalities if the patterns do not overlap.
\end{remark}

\subsection{Elementary cells}\label{sec:ElementaryCells}

We introduce here two elementary irrigation patterns which we will use to build more complicated ones.
They are illustrated in Figure\,\ref{fig:elementary_cell_in_2D_and_3D}.

\begin{definition}[Elementary cell in $n$D]
Let $\reSpace_{-1}=[-1,0]$, $\reSpace_{+1}=(0,1]$, and define for a multiindex $j\in\{-1,1\}^{n-1}$ the unit square
\begin{equation*}
\reSpace_j=\bigtimes_{i=1}^{n-1}\reSpace_{j_i}
\end{equation*}
in the orthant defined by $j$.
Its midpoint shall be denoted by $m_j=\frac12j\in\R^{n-1}$.

The \emph{elementary cell} with base point $x=(x',x_n)$, width $w$, height $h$, and flux $f$
is defined as the triple $(\reSpace_f,\reMeasureindexed{\reSpace_f},\chi_{x,w,h,f}^{\mathrm E})$
for the reference space $\reSpace_f=[-1,1]^{n-1}=\bigcup_{j\in\{-1,+1\}^{n-1}}\reSpace_j$, the reference measure $\reMeasureindexed{\reSpace_f}=f\lebesguecodimone\restr\reSpace_f$,
and the irrigation pattern $\chi_{x,w,h,f}^{\mathrm E} : \reSpace_f \times I \to \R^n$,
\begin{displaymath}
\chi_{x,w,h,f}^{\mathrm E}(p,t) = (x' + t \tfrac w2m_j, x_n + t h) \quad \text{for}\ p \in \reSpace_j.
\end{displaymath}
The \emph{initial} and \emph{final measure} of the elementary cell are
\begin{displaymath}
 \mu_+^{\chi_{x,w,h,f}^{\mathrm E}} = 2^{n-1}f\delta_{x},
 \qquad
 \mu_-^{\chi_{x,w,h,f}^{\mathrm E}} = f\sum_{j \in \{-1,1\}^{n-1}} \delta_{(x'+\tfrac w2m_j,x_n+h)}.
\end{displaymath}
\end{definition}

\begin{figure}
\centering
 \setlength{\unitlength}{.76\linewidth}
 \begin{picture}(1.00,0.25)
  \put(0.0,0.00){\includegraphics[scale=0.14]{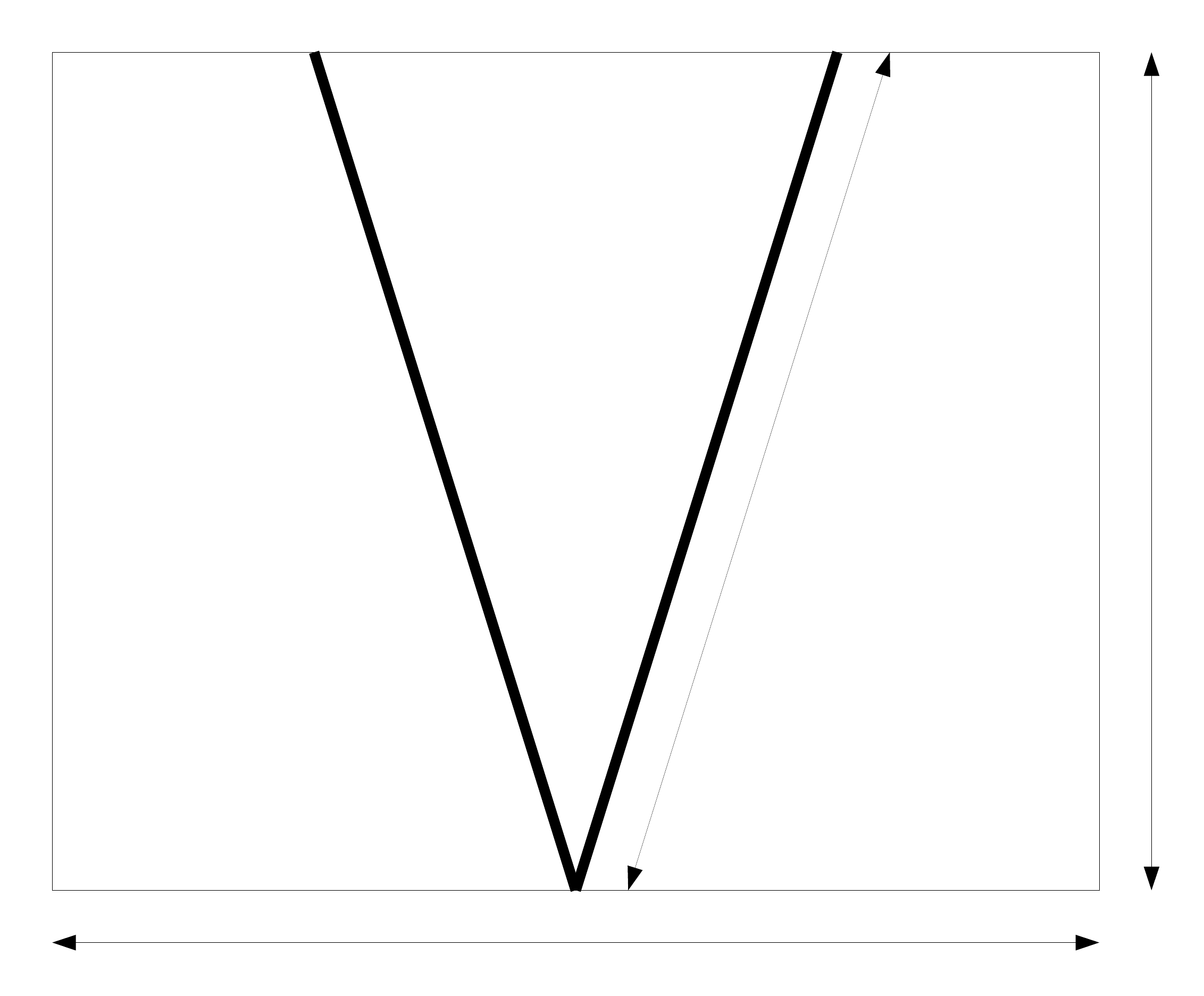}\includegraphics[scale=0.14]{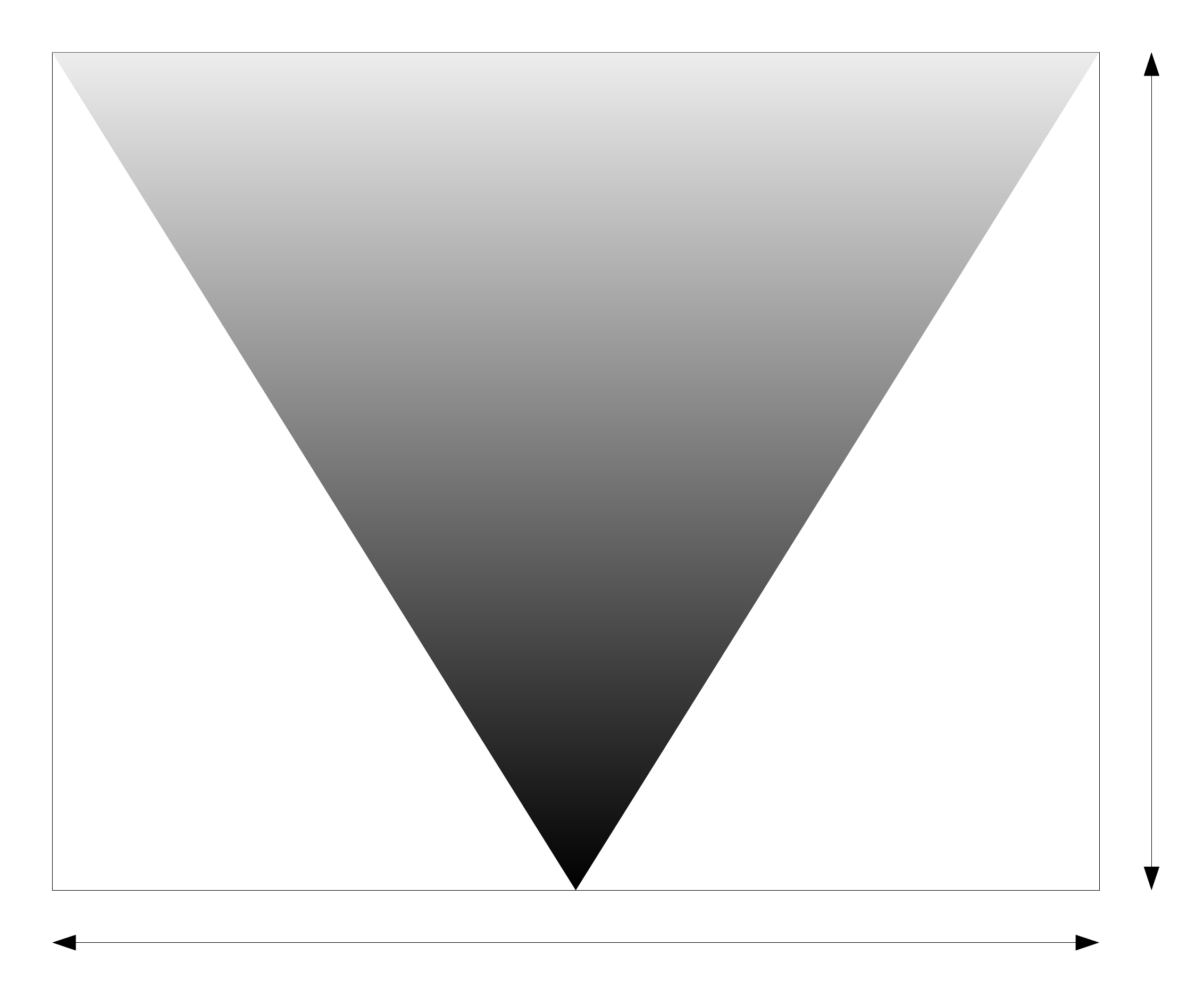}\includegraphics[scale=0.157]{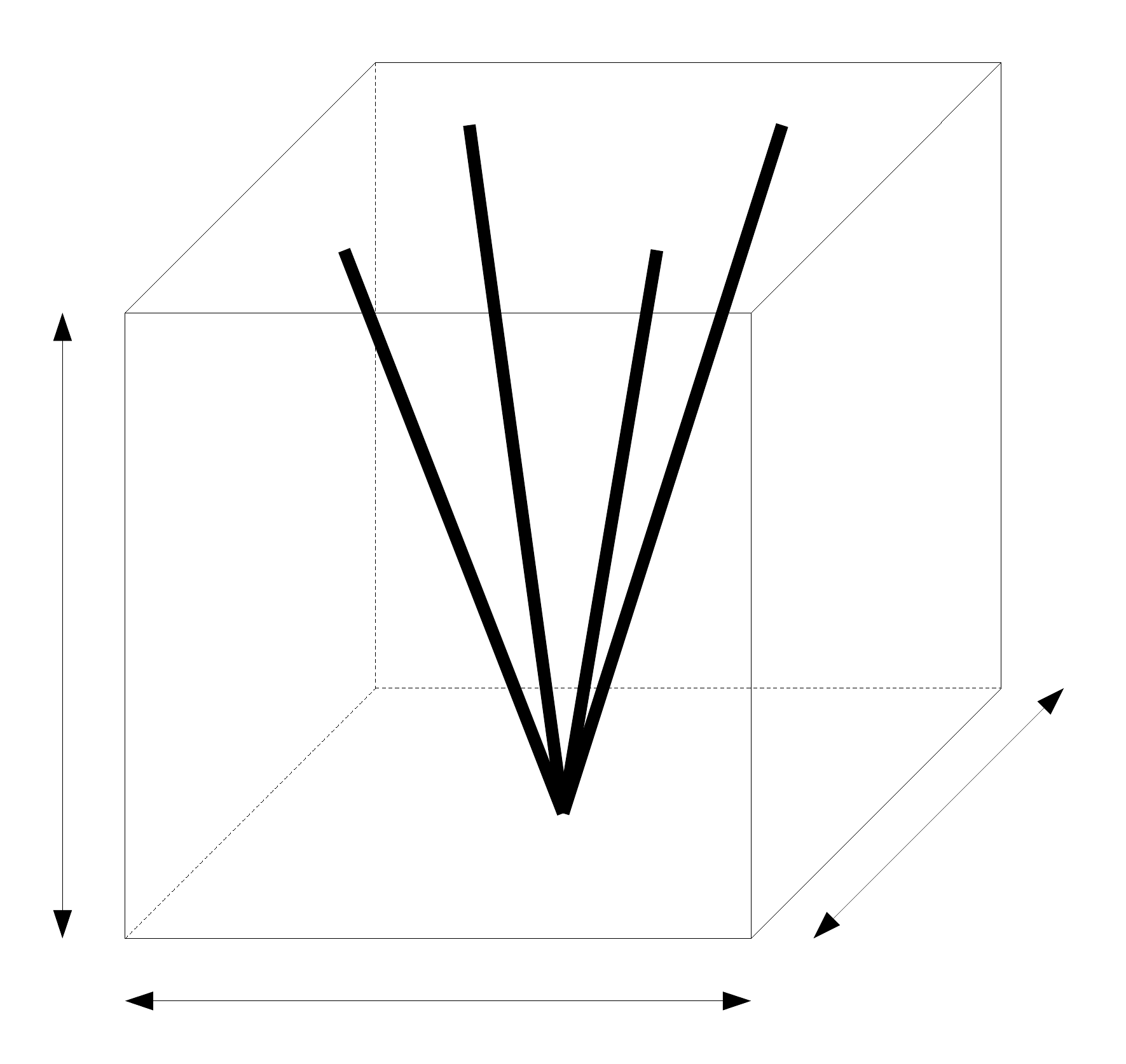}\includegraphics[scale=0.157]{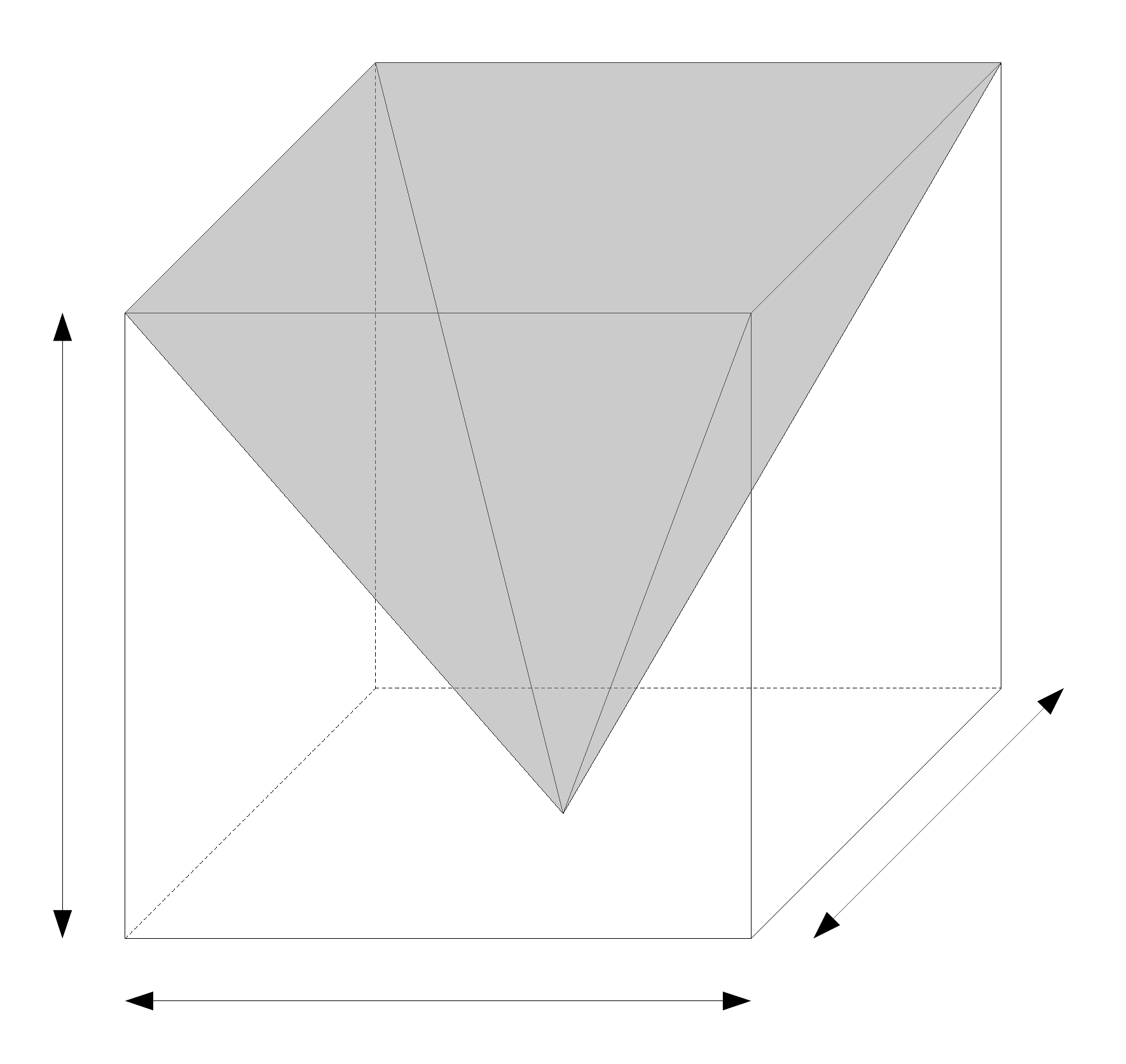}}
  \put(0.115,-0.01){$w$}
  \put(0.380,-0.01){$w$}
  \put(0.615,-0.01){$w$}
  \put(0.845,-0.01){$w$}
  \put(0.730,0.040){$w$}
  \put(0.965,0.040){$w$}
  \put(0.255,0.10){$h$}
  \put(0.76,0.08){$h$}
  \put(0.52,0.10){$h$}
  \put(0.528,0.06){$h$}
  \put(0.175,0.10){$l$}
 \end{picture}
 \caption{Elementary and Wasserstein cells in 2D and 3D.}\label{fig:elementary_cell_in_2D_and_3D}
\end{figure}
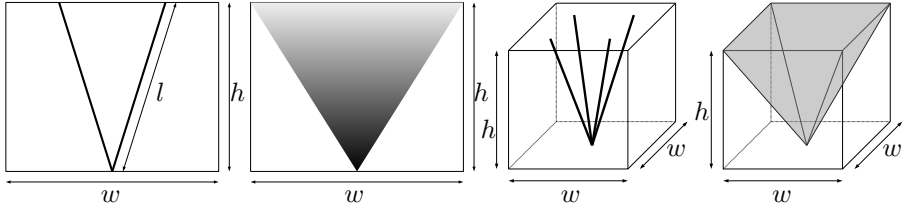

Note that the image of $\chi_{x,w,h,f}^{\mathrm E}$ consists of $2^{n-1}$ branches with length
\begin{equation}\label{eqn:tubeLength}
 l = \sqrt{(n-1)\left(\tfrac{w}{4}\right)^2+h^2} = \sqrt{\tfrac{n-1}{16}w^2+h^2}
\end{equation}
and is fully contained in
\begin{displaymath}
 E = x+\left(\left[-\tfrac{w}{2},\tfrac{w}{2}\right]^{n-1} \times [0,h]\right).
\end{displaymath}
The initial and final measure of the elementary cell are concentrated on the bottom and top boundary of $E$, respectively.

\begin{lemma}[Elementary cell cost]\label{lem:upper_bound_for_elementary_cell_urban_planning_cost_in_nD}
The urban planning and branched transport energy of an elementary cell satisfy
\begin{align*}
 \urbPlMMS^{\varepsilon,a}(\chi_{x,w,h,f}^{\mathrm E}) &= 2^{n-1}\min\{af,f+\varepsilon\}l \\
 \MMSEn^{\varepsilon}(\chi_{x,w,h,f}^{\mathrm E}) &= 2^{n-1}f^{1-\varepsilon}l
\end{align*}
for $l$ from \eqref{eqn:tubeLength}.
\end{lemma}

\begin{proof}
This is a straightforward calculation based on inserting $\chi_{x,w,h,f}^{\mathrm E}$ and $\reMeasureindexed{\reSpace_f}$ into Definition\,\ref{def:urbPlPatternForm} and Definition\,\ref{eqn:costDensityBrTpt}, using that the mass through each point $\tilde x$ of a branch is given by $m_{\chi_{x,w,h,f}^{\mathrm E}}(\tilde x)=f$.
\end{proof}

\begin{definition}[Wasserstein cell in $n$D]
The \emph{Wasserstein cell} with base point $x=(x',x_n)$, width $w$, height $h$, and flux $f$
is defined as the triple $(\reSpace_f,\reMeasureindexed{\reSpace_f},\chi_{x,w,h,f}^{\mathrm W})$ for $\reSpace_f,\reMeasureindexed{\reSpace_f}$ as in the previous definition
and the irrigation pattern $\chi_{x,w,h,f}^{\mathrm W} : \reSpace_f \times I \to \R^n$,
\begin{displaymath}
 \chi_{x,w,h,f}^{\mathrm W}(p,t) = \left(x'+t\tfrac w2p,x_n+th\right).
\end{displaymath}
The \emph{initial} and \emph{final measure} of the Wasserstein cell are
\begin{displaymath}
 \mu_+^{\chi_{x,w,h,f}^{\mathrm W}} = 2^{n-1}f\delta_{x},
 \qquad
 \mu_-^{\chi_{x,w,h,f}^{\mathrm W}} = f(\tfrac w2)^{1-n}\hdcodimone\restr\left(x+([-\tfrac w2,\tfrac w2]^{n-1}\times\{h\})\right).
\end{displaymath}
\end{definition}

Just like the elementary cell, the Wasserstein cell is fully contained in $E$.

\begin{lemma}[Wasserstein cell urban planning cost]\label{lem:upper_bound_for_wasserstein_cell_in_3D}
The urban planning energy of a Wasserstein cell satisfies
\begin{displaymath}
 \urbPlMMS^{\varepsilon,a}(\chi_{x,w,h,f}^{\mathrm W}) \leq
 \begin{cases}
 \min\{af,f+\varepsilon\}w&\text{if }n=2\text{ and }h=0,\\
 2^{n-1}af\sqrt{\tfrac{(n-1)}4w^2+h^2}&\text{else.}
 \end{cases}
\end{displaymath}
\end{lemma}

\begin{proof}
If $n\neq2$ or $h\neq0$ we have $m_{\chi_{x,w,h,f}^{\mathrm W}}(\tilde x)=0$ for all $\tilde x\neq x$.
Using this in Definition\,\ref{def:urbPlPatternForm} we obtain
\begin{multline*}
  \urbPlMMS^{\varepsilon,a}(\chi_{x,w,h,f}^{\mathrm W}) = \int_{\reSpace_f} \int_0^1 \repsilonachi(\chi_{x,w,h,f}^{\mathrm W}(p,t))|\dot\chi_{x,w,h,f}^{\mathrm W}(p,t)|\,\de t\,\de \reMeasureindexed{\reSpace_f}(p)\\
  =a\int_{\reSpace_f}\sqrt{\tfrac{w^2}4|p|^2+h^2}\,\de \reMeasureindexed{\reSpace_f}(p)
  \leq 2^{n-1}af\sqrt{\tfrac{(n-1)}4w^2+h^2}.
\end{multline*}
The case $n=2$, $h=0$ follows analogously, using $m_{\chi_{x,w,h,f}^{\mathrm W}}(\tilde x)=f(1-\frac{|\tilde x_1-x_1|}{w/2})$ for any $\tilde x\in(x_1-\tfrac w2,x_1+\tfrac w2)\times\{x_2\}$.
\end{proof}

\subsection{Upper bound for optimal urban planning energy $\urbPlEn^{\varepsilon,a,A}$}\label{sec:upBndUrbPlan}
To derive an upper bound we have to construct an irrigation pattern with the desired energy.
It turns out that branched transport in dimension $n=2$, $3$, and higher requires slightly different constructions,
however, the basic approach is the same for all of them.
In particular, we will employ a pattern as illustrated in Figures\,\ref{fig:ansatz2d} and \ref{fig:ansatz3d}.
It is symmetric about $x_n=\frac12$ so that it suffices to consider the upper half.
It consists of $\numLev$ layers, where the $k$\textsuperscript{th} layer just represents an array of $N_k^{n-1}$ identical elementary cells.
At the top boundary an additional layer of Wasserstein cells is added.

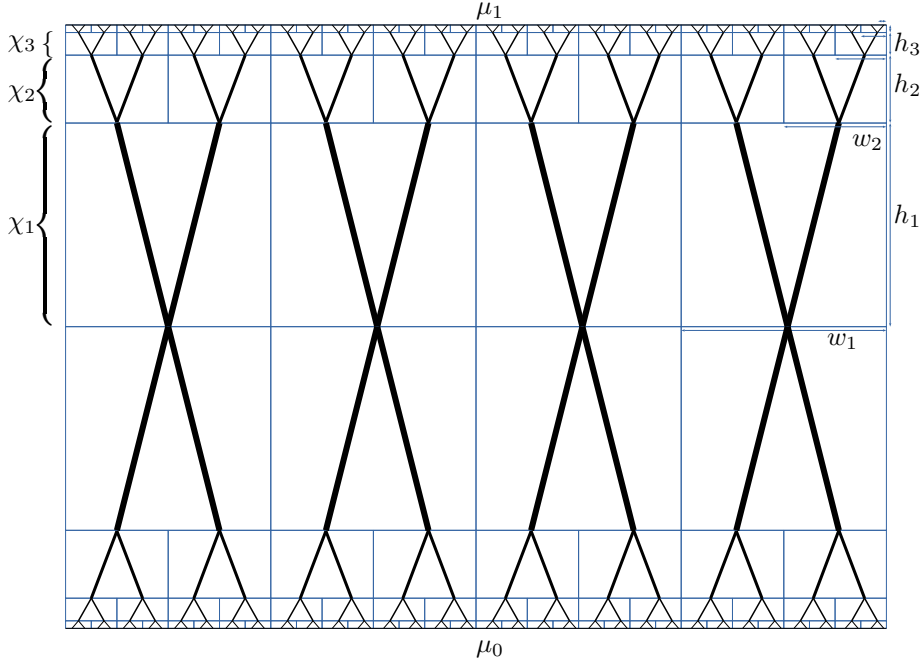
\begin{figure}
 \begin{center}
  \setlength{\unitlength}{110mm}
  \begin{picture}(1.00,0.75)
   \put(0.00,0.00){\includegraphics[scale=0.10]{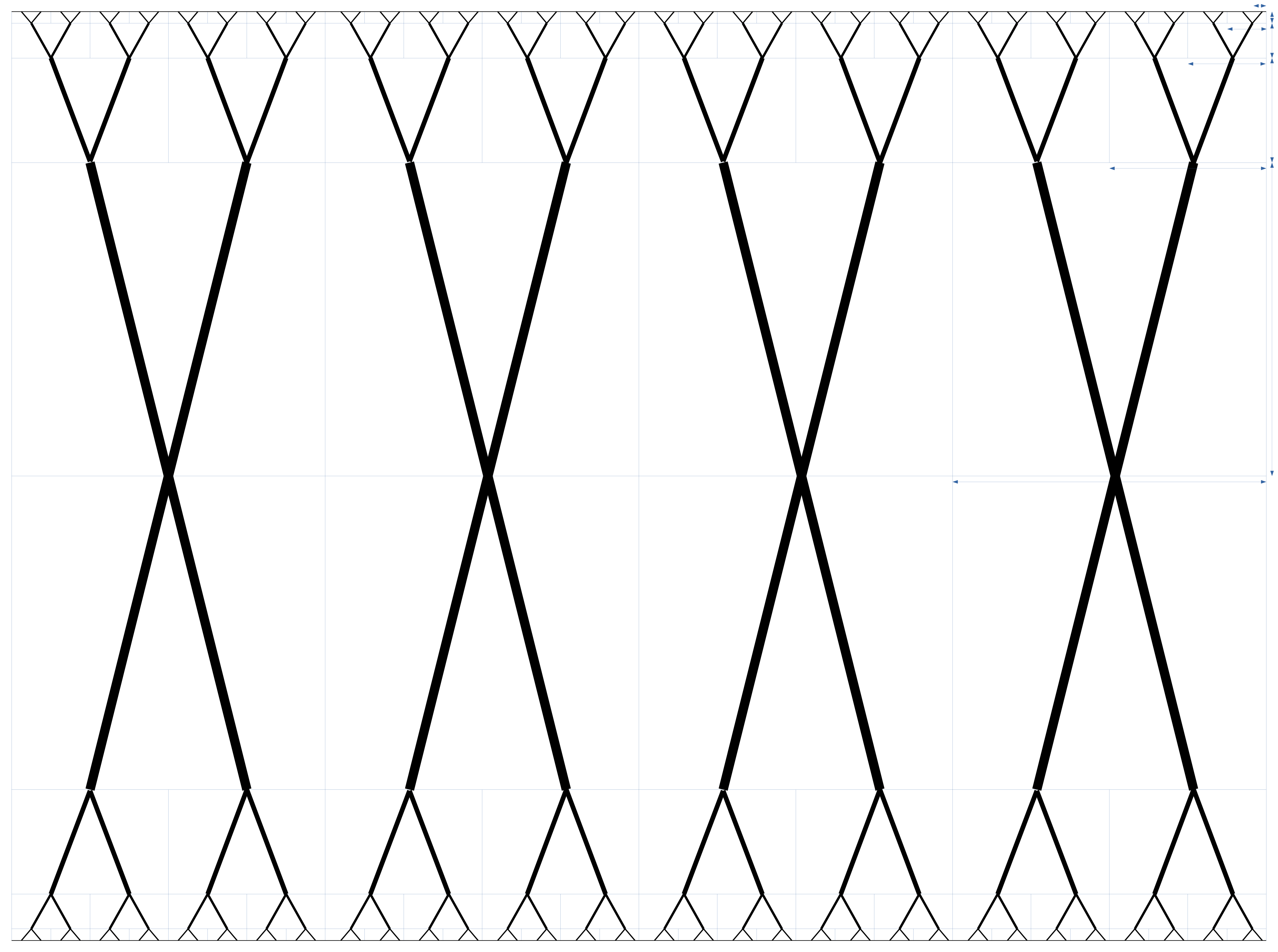}}
   \put(0.50,-0.02){$\mu_0$}
   \put(0.50,0.75){$\mu_1$}
   \put(0.92,0.35){$w_1$}
   \put(1.00,0.498){$h_1$}
   \put(0.95,0.59){$w_2$}
   \put(1.00,0.655){$h_2$}
   \put(-0.06,0.49){$\chi_1$}
   \put(-0.05,0.615){\rotatebox{-90}{$\underbrace{\hspace*{17.5ex}}_{\ }$}}
   \put(-0.06,0.65){$\chi_2$}
   \put(-0.05,0.695){\rotatebox{-90}{$\underbrace{\hspace*{5.5ex}}_{\ }$}}
   \put(-0.06,0.71){$\chi_3$}
   \put(-0.02,0.705){$\{$}
   \put(1.00,0.705){$h_3$}
  \end{picture}
  \caption{Illustrative sketch of the construction ansatz in two dimensions.
  The thin lines indicate the boundaries of the elementary cells.}\label{fig:ansatz2d}
 \end{center}
\end{figure}

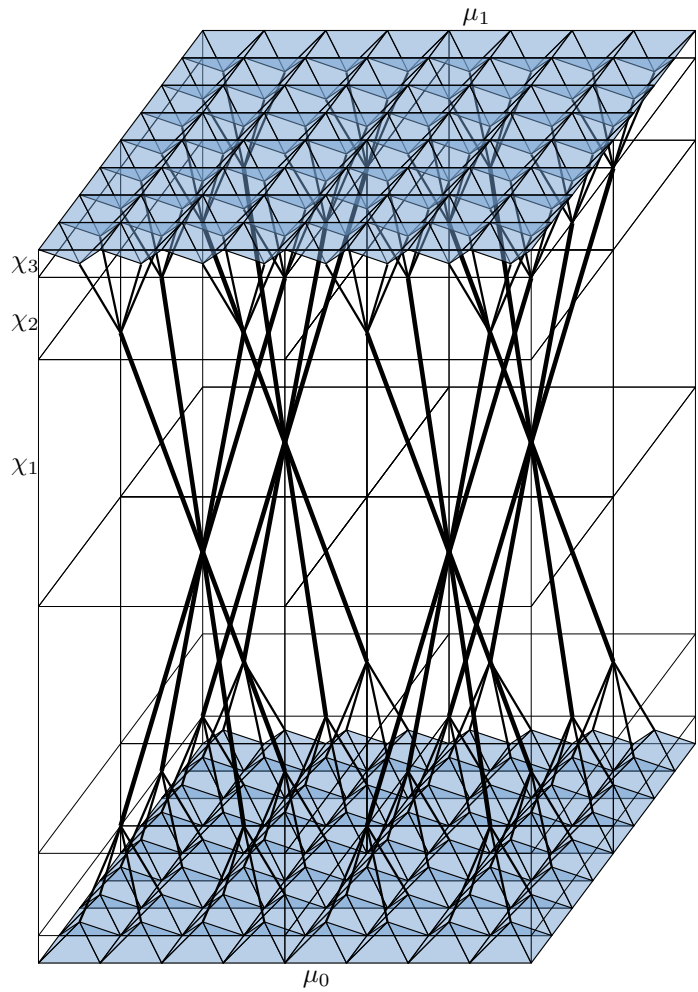
\begin{figure}
 \begin{center}
  \setlength{\unitlength}{.6\linewidth}
  \begin{picture}(1.00,1.33)
   \put(0.00,0.00){\includegraphics[height=1.326\unitlength]{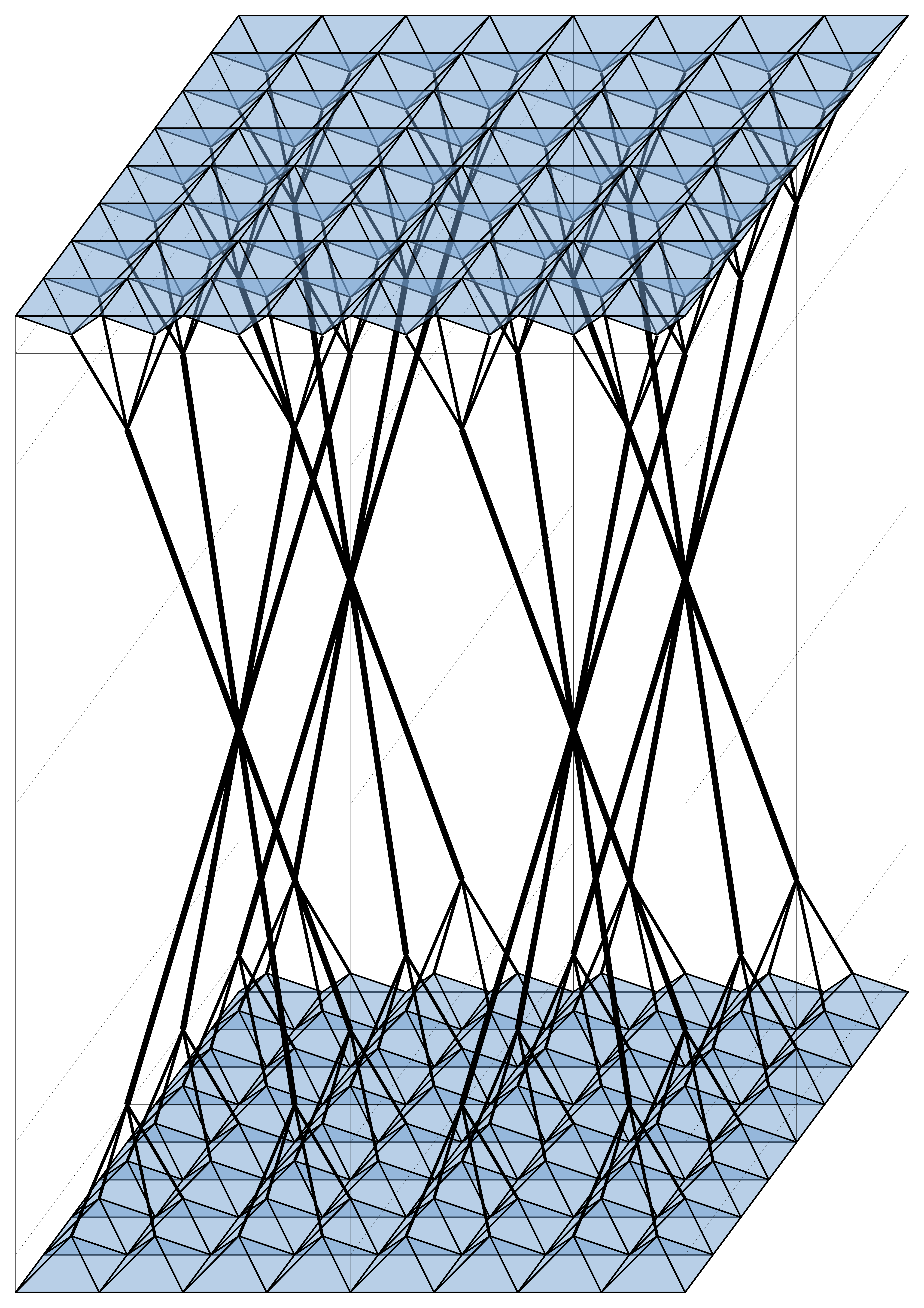}}
   \put(0.38,-0.01){$\mu_0$}
   \put(0.6,1.325){$\mu_1$}
   \put(-0.022,0.7){$\chi_1$}
   \put(-0.022,.9){$\chi_2$}
   \put(-0.022,.98){$\chi_3$}
  \end{picture}
  \caption{Illustrative sketch of the construction ansatz in three dimensions.
  The thin lines indicate the boundaries of the elementary cells, the shaded regions represent Wasserstein cells.}\label{fig:ansatz3d}
 \end{center}
\end{figure}

In more detail, without loss of generality we assume the hypersquare $A$, whose sidelength is $\ell=\sqrt[n-1]{\hdcodimone(A)}$, to be given by $$A=[0,\ell]^{n-1}\,.$$
Let $x_{k,i}$, $w_k$, $h_k$, and $f_k$ denote the base point, width, height, and flux associated with the $i$\textsuperscript{th} elementary cell in the $k$\textsuperscript{th} layer,
where $i\in\{1,\ldots,N_k\}^{n-1}$ is a multiindex.
We choose
\begin{equation}\label{eqn:parametersConstructionUrbPl}
\begin{gathered}
w_k=2^{1-k}w_1,\quad
N_k=\tfrac{\ell}{w_k},\quad
f_k=\left(\tfrac{w_k}2\right)^{n-1},\quad\\
h_k=cw_k^\alpha\varepsilon^\beta a^\gamma(a-1)^\delta,\quad
x_{k,i}=\left((i_1-\tfrac{1}{2})w_k,\ldots,(i_{n-1}-\tfrac12)w_k,\tfrac{1}{2} + \sum_{j=1}^{k-1} h_j\right)
\end{gathered}
\end{equation}
for the coarsest cell width $w_1$ and some constants $c,\alpha,\beta,\gamma,\delta\in\R$ to be determined.
Now we define the irrigation pattern describing the $k$\textsuperscript{th} layer as
\begin{equation*}
 \chi_k = \coprod_{i\in\{1,\ldots,N_k\}^{n-1}} \chi_{x_{k,i},w_k,h_k,f_k}^{\mathrm E}
\end{equation*}
and the top layer of Wasserstein cells as
\begin{equation*}
 \chi_{\numLev+1} = \coprod_{i\in\{1,\ldots,N_k\}^{n-1}} \chi_{x_{{\numLev+1},i},w_{\numLev+1},H,f_{\numLev+1}}^{\mathrm W}
\end{equation*}
for some $K\in\N$, $H\in\R$.
Note that for $k=1,\ldots,\numLev+1$, the reference space of $\chi_k$ is the disjoint union $\reSpace_k=\coprod_{i\in\{1,\ldots,N_k\}^{n-1}}\reSpace_{f_k}^i$,
where $\reSpace_{f_k}^i$ just denotes a copy of $\reSpace_{f_k}$.
Now it is straightforward to check that the function
$$T_k:\reSpace_k\to\R^{n-1}\,,\quad\reSpace_{f_k}^i\ni p\mapsto x_{k,i}+\tfrac w2p$$
represents a measure-preserving map from the original reference space $\reSpace_k$ onto $(A,\Bcal(A),\lebesgue^{n-1}\restr A)$.
Hence, we can identify particles in $\reSpace_k$ with particles in $A$ via the map $T_k$ and thus may take $(A,\Bcal(A),\lebesgue^{n-1}\restr A)$ as the reference space of $\chi_k$.
This simplifies the composition of those patterns, as now we have
$$\chi_k(p,1)=\chi_{k+1}(p,0)\qquad\text{for almost all }p\in A$$
so that we may define the series of all layers in the upper half as
\begin{equation*}
\chi_u=((\ldots((\chi_1\circ_\Id\chi_2)\circ_\Id\chi_3)\circ_\Id\ldots)\circ_\Id\chi_{\numLev+1})\,,
\end{equation*}
using Definition \ref{def:series_of_patterns} with the identity map $\Id:A\to A$.
The irrigation pattern $\chi_l$ of the lower half is constructed in the analogous way,
and the full irrigation pattern is defined as
\begin{equation*}
\chi=\chi_l\circ_\Id\chi_u.
\end{equation*}

It turns out that the constraint of the construction having total height $1$ already fixes the width of the coarsest elementary cells.
Indeed, since the heights of all layers have to add up to $1$ we require
\begin{multline}\label{eqn:heightConstraint}
1=2\left(\sum_{k=1}^\numLev h_k+H\right)
=2c(2w_1)^\alpha\varepsilon^\beta a^\gamma(a-1)^\delta\sum_{k=1}^\numLev2^{-k\alpha}+2H
=2cw_1^\alpha\varepsilon^\beta a^\gamma(a-1)^\delta\frac{1-2^{-\numLev\alpha}}{1-2^{-\alpha}}+2H.
\end{multline}
Assuming $1-2H$ to still be of order $1$ it transpires that we should roughly have $w_1\sim\varepsilon^{-\frac\beta\alpha}a^{-\frac\gamma\alpha}(a-1)^{-\frac\delta\alpha}$.
Based on these heuristics we simply choose
\begin{equation}\label{eqn:coarseWidth}
w_1=\tilde c\varepsilon^{-\frac\beta\alpha}a^{-\frac\gamma\alpha}(a-1)^{-\frac\delta\alpha}\quad\text{for }\tilde c\in(\tfrac12,1],
\end{equation}
where $\tilde c$ accounts for the fact that the total total width $\ell$ of the geometry has to be an integer multiple of $w_1$,
that is, $\tilde c$ is determined such that this holds true.
Of course, this requires in particular $w_1\leq\ell$, which is why for all our constructions we will have the constraint
\begin{equation}\label{eqn:epsConstraint}
\varepsilon\leq a^{-\frac\gamma\beta}(a-1)^{-\frac\delta\beta}\hdcodimone(A)^{-\frac\alpha{\beta(n-1)}}.
\end{equation}
The constant $c$ is now obtained from \eqref{eqn:heightConstraint} as
\begin{equation}\label{eqn:constantC}
c=\frac{1-2H}{2\tilde c^\alpha}\frac{1-2^{-\alpha}}{1-2^{-\numLev\alpha}}\in[\tfrac{1-2^{-\alpha}}4,2^{\alpha-1}].
\end{equation}
Here the possible range of $c$ is based on the conditions
\begin{equation}\label{eqn:feasibilityConditions}
\numLev\geq1,\qquad H\leq\tfrac14,\qquad \alpha>0,
\end{equation}
which will be ensured in all our constructions.

We next have to estimate the energy of $\chi$.
By the symmetry of the construction and Remark\,\ref{rem:cellOpProps} we have
\begin{align*}
\urbPlEn^{\varepsilon,a,A}[\chi]
=2\urbPlMMS^{\varepsilon,a}[\chi_u]
&=2\Bigg[\sum_{k=1}^{\numLev}\sum_{i\in\{1,\ldots,N_k\}^{n-1}}\hspace*{-2ex}\urbPlMMS^{\varepsilon,a}[\chi_{x_{k,i},w_k,h_k,f_k}^{\mathrm E}]
+\sum_{i\in\{1,\ldots,N_{\numLev+1}\}^{n-1}}\hspace*{-2ex}\urbPlMMS^{\varepsilon,a}[\chi_{x_{{\numLev+1},i},w_{\numLev+1},H,f_{\numLev+1}}^{\mathrm W}]\Bigg].
\end{align*}
On the other hand, using $2\left(\sum_{k = 1}^{\numLev}h_k+H\right)=1$ and $N_k=\frac\ell{w_k}$, the energy value $\urbPlEn^{*,a,A}$ can be written as
\begin{displaymath}
 \urbPlEn^{*,a,A}=\Wdone(\mu_0,\mu_1) = \hdcodimone(A) = 2\left(\sum_{k = 1}^{\numLev} N_k^{n-1} w_k^{n-1} h_k+N_{\numLev+1}^{n-1}w_{\numLev+1}^{n-1}H\right).
\end{displaymath}
Thus we obtain
\begin{align}
 \Delta\urbPlEn^{\varepsilon,a,A}
 &\leq\urbPlEn^{\varepsilon,a,A}[\chi]-\urbPlEn^{*,a,A}\nonumber\\
 &\leq2\Bigg[\sum_{k=1}^{\numLev}\sum_{i\in\{1,\ldots,N_k\}^{n-1}}\left(\urbPlMMS^{\varepsilon,a}[\chi_{x_{k,i},w_k,h_k,f_k}^{\mathrm E}]-w_k^{n-1}h_k\right)\nonumber\\
 &\quad+\sum_{i\in\{1,\ldots,N_{\numLev+1}\}^{n-1}}\left(\urbPlMMS^{\varepsilon,a}[\chi_{x_{{\numLev+1},i},w_{\numLev+1},H,f_{\numLev+1}}^{\mathrm W}]-w_{\numLev+1}^{n-1}H\right)\Bigg],\label{eqn:generalConstructionEstimate}
\end{align}
in which now Lemmas\,\ref{lem:upper_bound_for_elementary_cell_urban_planning_cost_in_nD} and \ref{lem:upper_bound_for_wasserstein_cell_in_3D} can be used.

\subsubsection{Upper bound in 2D}

The first results, Theorem \ref{thm:upper_bound_for_urban_planning_minus_wasserstein_in_2D} and Corollary \ref{cor:upper_bound_for_min_urban_planning_minus_wasserstein_in_2D}, are concerned with the two-dimensional situation.

\begin{theorem}[Partial upper bound on $\Delta\urbPlEn^{\varepsilon,a,A}$ in 2D]\label{thm:upper_bound_for_urban_planning_minus_wasserstein_in_2D}
Using the notation from Section \ref{sec:introUrbPlan}, there exists $C>0$ independent of $\varepsilon,a,A$ such that for all $\varepsilon<\min\{1,\hdone(A)^3\}$ we have
\begin{displaymath}
 \Delta\urbPlEn^{\varepsilon,a,A} \leq \hdone(A) C \varepsilon^{\frac{2}{3}}.
\end{displaymath}
\end{theorem}

Before proving Theorem \ref{thm:upper_bound_for_urban_planning_minus_wasserstein_in_2D} let us note that it directly implies the following corollary.

\begin{corollary}[Upper bound on $\Delta\urbPlEn^{\varepsilon,a,A}$ in 2D]\label{cor:upper_bound_for_min_urban_planning_minus_wasserstein_in_2D}
For the same $C$ as before, for all $\varepsilon<\min\{1,\hdone(A)^3\}$ we have
\begin{displaymath}
 \Delta\urbPlEn^{\varepsilon,a,A} \leq \hdone(A) \min\{a-1,C\varepsilon^{\frac{2}{3}}\}.
\end{displaymath}
\end{corollary}

\begin{proof}
This follows immediately from Theorem \ref{thm:upper_bound_for_urban_planning_minus_wasserstein_in_2D} and the fact that
\begin{displaymath}
 \Delta\urbPlEn^{\varepsilon,a,A}= \min_\chi \urbPlEn^{\varepsilon,a,A}[\chi] - \Wdone(\mu_0,\mu_1) \leq (a-1) \Wdone(\mu_0,\mu_1)=\hdone(A)(a-1),
\end{displaymath}
where we have used $\urbPlEn^{\varepsilon,a,A}[\chi] = a\Wdone(\mu_0,\mu_1)$
for the irrigation pattern $\chi : \Gamma \times I \to \R^2$, $\chi(p,t) = (p,t)$,
with reference space $\Gamma = A$ and reference measure $\reMeasure = \lebesgue^1\restr A$.
\end{proof}

\begin{proof}[Proof of Theorem \ref{thm:upper_bound_for_urban_planning_minus_wasserstein_in_2D}]
In the two-dimensional construction we set the height of the top layer of Wasserstein cells to zero,
$$H=0.$$
Using Lemmas\,\ref{lem:upper_bound_for_elementary_cell_urban_planning_cost_in_nD} and \ref{lem:upper_bound_for_wasserstein_cell_in_3D} in \eqref{eqn:generalConstructionEstimate} we thus obtain
\begin{align*}
 \Delta\urbPlEn^{\varepsilon,a,A}
 &\leq2\left[\sum_{k=1}^{\numLev}N_k\left(2(f_k+\varepsilon)\sqrt{\tfrac{w_k^2}{16}+h_k^2}-w_kh_k\right)
  +N_{\numLev+1}(f_{\numLev+1}+\varepsilon)w_{\numLev+1}\right]\\
 &=\hdone(A)2\left[\sum_{k=1}^{\numLev}h_k\left((1+\tfrac{2\varepsilon}{w_k})\sqrt{1+\tfrac{w_k^2}{16h_k^2}}-1\right)
  +(\tfrac{w_{\numLev+1}}2+\varepsilon)\right]\\
 &\leq\hdone(A)2\left[\sum_{k=1}^{\numLev}\left(\tfrac{w_k^2}{32h_k}+\varepsilon\tfrac{h_k}{w_k}(2+\tfrac{w_k^2}{16h_k^2})\right)
  +(\tfrac{w_{\numLev+1}}2+\varepsilon)\right],
\end{align*}
where we have used $\sqrt{1+z}\leq1+\frac z2$.
The parameter $a$ does not occur explicitly in this estimate so that we expect
\begin{equation*}
\gamma=0,\qquad\delta=0
\end{equation*}
to be optimal.
Now intuitively, the mass flux in each elementary cell should have a vertical component at most comparable to the horizontal one,
so we will continue layering only as long as $2cw_k\leq h_k$ or equivalently, assuming $\alpha>1$, as long as $2^{2-k}w_1\geq\varepsilon^{\frac\beta{1-\alpha}}$. Thus we set
\begin{equation*}
\numLev=\left\lfloor2+\log_2(w_1\varepsilon^{\frac\beta{\alpha-1}})\right\rfloor,
\end{equation*}
where the $\lfloor\cdot\rfloor$ denotes the integer part (note that we will have to verify a posteriori that $\numLev\geq1$).
As a result, $\tfrac{w_k^2}{16h_k^2}\leq\tfrac1{64c^2}$ for all $k\leq K$ as well as $w_{\numLev+1}<\varepsilon^{\frac\beta{1-\alpha}}$, and the excess energy is estimated as
\begin{align*}
 \Delta\urbPlEn^{\varepsilon,a,A}
 &\leq\hdone(A)\left[\sum_{k=1}^{\numLev}\left(\tfrac{w_k^2}{16h_k}+\varepsilon\tfrac{h_k}{w_k}(4+\tfrac1{32c^2})\right)
  +(\varepsilon^{\frac\beta{1-\alpha}}+2\varepsilon)\right]\\
 &=\hdone(A)\left[\sum_{k=1}^{\numLev}\left(\tfrac1{16c}w_k^{2-\alpha}\varepsilon^{-\beta}+(4c+\tfrac1{32c})w_k^{\alpha-1}\varepsilon^{1+\beta}\right)
  +(\varepsilon^{\frac\beta{1-\alpha}}+2\varepsilon)\right]\\
 &=\hdone(A)\left[\sqrt2(\tfrac3{32c}+4c)w_1^{\frac12}\varepsilon^{\frac12}\sum_{k=1}^{\numLev}2^{-\frac k2}+3\varepsilon\right]\\
 &\leq\hdone(A)\left[\tfrac{2}{1-\sqrt2}(\tfrac3{32c}+4c)w_1^{\frac12}\varepsilon^{\frac12}+3\varepsilon\right],
\end{align*}
where in the third step we chose the optimal exponents
\begin{equation*}
\alpha=\frac32,\qquad\beta=-\frac12
\end{equation*}
to minimise the summands and in the last step we exploited the boundedness of the geometric series.
Note that for this choice $w_1$ turns into $w_1=\tilde c\varepsilon^{\frac13}$ and the feasibility conditions \eqref{eqn:feasibilityConditions} are indeed fulfilled,
while constraint \eqref{eqn:epsConstraint} becomes $\varepsilon\leq\hdone(A)^3$.
Thus we finally arrive at $\Delta\urbPlEn^{\varepsilon,a,A}\leq C\hdone(A)\varepsilon^{\frac23}$ for some $C>0$, the desired result.
\end{proof}

\subsubsection{Upper bound in 3D}

In three dimensions there are three different parameter regimes as indicated by Theorems \ref{thm:upper_bound_for_urban_planning_minus_wasserstein_in_3D} and \ref{thm:upper_bound_for_urban_planning_minus_wasserstein_in_3D_for_big_a} as well as Corollary \ref{cor:upper_bound_for_min_urban_planning_minus_wasserstein_in_3D}.

\begin{theorem}[Upper bound on $\Delta\urbPlEn^{\varepsilon,a,A}$ in 3D for small $a$]\label{thm:upper_bound_for_urban_planning_minus_wasserstein_in_3D}
Using the notation from Section \ref{sec:introUrbPlan}, there exists a constant $C\geq0$ such that for all $a\leq2$ and $\varepsilon\leq\min\{\frac{a^2(a-1)^2}{64},\hdtwo(A)^2\}$ we have
\begin{equation*}
 \Delta\urbPlEn^{\varepsilon,a,A} \leq \hdtwo(A)\sqrt\varepsilon C\left[\left|\log_2\left(\frac{\sqrt\varepsilon}{a(a-1)}\right)\right|+a(a-1)+1\right].
\end{equation*}
\end{theorem}

\begin{theorem}[Upper bound on $\Delta\urbPlEn^{\varepsilon,a,A}$ in 3D for large $a$]\label{thm:upper_bound_for_urban_planning_minus_wasserstein_in_3D_for_big_a}
There exists a constant $C\geq0$ such that for all $\varepsilon \leq \min\{1,\hdtwo(A)^2\}$ we have
\begin{equation*}
 \Delta\urbPlEn^{\varepsilon,a,A} \leq \hdtwo(A)\sqrt\varepsilon C\left[\left|\log_2(\sqrt\varepsilon)\right| + \sqrt{a} + 1\right].
\end{equation*}
\end{theorem}

Theorems \ref{thm:upper_bound_for_urban_planning_minus_wasserstein_in_3D} and \ref{thm:upper_bound_for_urban_planning_minus_wasserstein_in_3D_for_big_a} as well as a third parameter regime can be summarised in the following corollary.

\begin{corollary}[Upper bound on $\Delta\urbPlEn^{\varepsilon,a,A}$ in 3D]\label{cor:upper_bound_for_min_urban_planning_minus_wasserstein_in_3D}
There exists $C>0$ independent of $a,\varepsilon,A$, such that for all $\varepsilon \leq \min(1,\hdtwo(A)^2)$ we have
\begin{displaymath}
 \Delta\urbPlEn^{\varepsilon,a,A} \leq  \hdtwo(A) \min\left\{a-1,C\sqrt\varepsilon(|\log_2\tfrac{\sqrt\varepsilon}{a-1}|+\sqrt a)\right\}.
\end{displaymath}
\end{corollary}

\begin{proof}
Again, as in  Corollary \ref{cor:upper_bound_for_min_urban_planning_minus_wasserstein_in_2D} we have
\begin{displaymath}
 \Delta\urbPlEn^{\varepsilon,a,A}= \min_\chi \urbPlEn^{\varepsilon,a,A}[\chi] - \Wdone(\mu_0,\mu_1) \leq (a-1) \Wdone(\mu_0,\mu_1)=\hdtwo(A)(a-1)
\end{displaymath}
due to $\urbPlEn^{\varepsilon,a,A}[\chi] = a\Wdone(\mu_0,\mu_1)$
for the irrigation pattern $\chi : \Gamma \times I \to \R^2$, $\chi(p,t) = (p,t)$,
with reference space $\Gamma = A$ and reference measure $\reMeasure = \lebesgue^2\restr A$.

Now let $\varepsilon\leq \min\{1,\hdtwo(A)^2\}$.
Without loss of generality we may assume the constant $C$ in Theorems \ref{thm:upper_bound_for_urban_planning_minus_wasserstein_in_3D} and \ref{thm:upper_bound_for_urban_planning_minus_wasserstein_in_3D_for_big_a} to be the same and to be greater than $2$.
Let us abbreviate
\begin{align*}
 T_1&=\sqrt\varepsilon\left[\left|\log_2\left(\tfrac{\sqrt\varepsilon}{a(a-1)}\right)\right|+a(a-1)+1\right],\\
 T_2&=\sqrt\varepsilon\left[\left|\log_2(\sqrt\varepsilon)\right| + \sqrt{a} + 1\right].
\end{align*}
If $a>2$, then it can easily be checked that $CT_1>CT_2$.
Likewise, if $\varepsilon>\frac{a^2(a-1)^2}{64}$ it is straightforward to see $CT_1>a-1$.
Thus, $\Delta\urbPlEn^{\varepsilon,a,A} \leq \hdtwo(A)\min\{a-1,CT_1,CT_2\}$ does not only hold under the conditions of Theorem \ref{thm:upper_bound_for_urban_planning_minus_wasserstein_in_3D}, but even independently of $a$ and $\varepsilon$.
Furthermore, we trivially have $CT_1\leq 4C\sqrt\varepsilon\frac14(|\log_2\tfrac{\sqrt\varepsilon}{a(a-1)}|+a(a-1)+1)$ and $CT_2\leq 4C\sqrt\varepsilon(|\log_2\sqrt\varepsilon|+\sqrt a)$ due to $a>1$,
so it remains to show
\begin{equation*}
\min\left\{\tfrac14\left(\left|\log_2\tfrac{\sqrt\varepsilon}{a(a-1)}\right|+a(a-1)+1\right),|\log_2\sqrt\varepsilon|+\sqrt a\right\}\leq |\log_2\tfrac{\sqrt\varepsilon}{a-1}|+\sqrt a.
\end{equation*}
For $a\geq2$ we have $|\log_2\sqrt\varepsilon|\leq|\log_2\tfrac{\sqrt\varepsilon}{a-1}|$ so that the above inequality holds true.
For $a\in(1,2)$ on the other hand we have
\begin{equation*}
\left|\log_2\tfrac{\sqrt\varepsilon}{a(a-1)}\right|+a(a-1)+1
\leq\left|\log_2\tfrac{\sqrt\varepsilon}{a-1}\right|+\log_2 a+a(a-1)+1
\leq4\left|\log_2\tfrac{\sqrt\varepsilon}{a-1}\right|+4\sqrt a
\end{equation*}
so that the inequality again holds true, which proves the desired inequality in all cases.
\end{proof}

\begin{proof}[Proof of Theorem \ref{thm:upper_bound_for_urban_planning_minus_wasserstein_in_3D}]
Again we just  appeal to Lemmas\,\ref{lem:upper_bound_for_elementary_cell_urban_planning_cost_in_nD} and \ref{lem:upper_bound_for_wasserstein_cell_in_3D} in \eqref{eqn:generalConstructionEstimate} to obtain
\begin{align*}
\Delta\urbPlEn^{\varepsilon,a,A}
&\leq2\Bigg[\sum_{k=1}^{\numLev}N_k^2\left(4(f_k+\varepsilon)\sqrt{\tfrac{w_k^2}{8}+h_k^2}-w_k^2h_k\right)
+N_{\numLev+1}^2\left(4af_{\numLev+1}\sqrt{\tfrac{w_{\numLev+1}^2}{8}+H^2}-w_{\numLev+1}^2H\right)\Bigg]\\
&=\hdtwo(A)2\left[\sum_{k=1}^{\numLev}h_k\left((1+\tfrac{4\varepsilon}{w_k^2})\sqrt{1+\tfrac{w_k^2}{8h_k^2}}-1\right)
+H\left(a\sqrt{1+\tfrac{w_{\numLev+1}^2}{8H^2}}-1\right)\right]\\
&\leq\hdtwo(A)2\left[\sum_{k=1}^{\numLev}\left(\tfrac{w_k^2}{16h_k}+\varepsilon\tfrac{h_k}{w_k^2}(4+\tfrac{w_k^2}{4h_k^2})\right)
+\left(H(a-1)+a\tfrac{w_{\numLev+1}^2}{16H}\right)\right]
\end{align*}
using $\sqrt{1+z}\leq1+\frac z2$.
We now choose the minimising
\begin{equation*}
H=\sqrt{\frac a{a-1}}\frac{w_{\numLev+1}}4
\end{equation*}
(note that we still need to verify $H\leq\frac14$), thus arriving at
\begin{align*}
\Delta\urbPlEn^{\varepsilon,a,A}
&\leq\hdtwo(A)2\left[\sum_{k=1}^{\numLev}\left(\tfrac{w_k^2}{16h_k}+\varepsilon\tfrac{h_k}{w_k^2}(4+\tfrac{w_k^2}{4h_k^2})\right)
+\sqrt{a(a-1)}\tfrac{w_{\numLev+1}}2\right]\\
&=\hdtwo(A)\left[\sum_{k=1}^{\numLev}\left(\tfrac{1}{8c}w_k^{2-\alpha}\varepsilon^{-\beta}+w_k^{\alpha-2}\varepsilon^{1+\beta}(8c+\tfrac{1}{2c}w_k^{2-2\alpha}\varepsilon^{-2\beta})\right)
+\sqrt{a(a-1)}w_{\numLev+1}\right]\\
&=\hdtwo(A)\left[\sqrt\varepsilon\sum_{k=1}^{\numLev}\left(\tfrac{1}{8c}+8c+\tfrac{1}{2c}w_k^{-2}\varepsilon\right)
+\sqrt{a(a-1)}w_{\numLev+1}\right],
\end{align*}
where we have chosen
\begin{equation*}
\gamma=0,\qquad\delta=0,\qquad\alpha=2,\qquad\beta=-\frac12
\end{equation*}
to balance all terms optimally.
Recall that due to \eqref{eqn:coarseWidth} and \eqref{eqn:epsConstraint} this leads to $w_1=\tilde c\varepsilon^{\frac14}$ as well as the constraint $\varepsilon\leq\hdtwo(A)^2$.

Note that we still have to specify $\numLev$.
The intuition here is the following.
Assuming that $w_k^2$ stays larger than $\varepsilon$, the sum evaluates to a constant times $\numLev$.
One might expect $\numLev$ to increase at most logarithmically in $\varepsilon$ and $a-1$, which is not much worse than having a constant upper bound for $\numLev$.
Thus, the second term may become as large as $\sqrt\varepsilon$ without polluting the energy estimate.
Therefore, we will choose $\numLev$ such that $w_{\numLev+1}\sim\sqrt{\frac{\varepsilon}{a(a-1)}}$, that is,
\begin{equation*}
\numLev=\left\lfloor\log_2\tfrac{\sqrt{a(a-1)}}{\varepsilon^{1/4}}\right\rfloor,\quad\text{and thus }w_{\numLev+1}=\frac{\tilde c}{\hat c}\sqrt{\frac\varepsilon{a(a-1)}},\,H=\frac{\tilde c\sqrt\varepsilon}{4\hat c(a-1)},
\end{equation*}
where again $\lfloor\cdot\rfloor$ denotes the integer part and where $\hat c=\tfrac{\varepsilon^{1/4}}{\sqrt{a(a-1)}}2^{\numLev}\in(\frac12,1]$.
Note that our assumptions on $\varepsilon$ ensure $\numLev\geq1$ and $H\leq\frac18$ so that condition \eqref{eqn:feasibilityConditions} is fulfilled.

Inserting our parameter choices in the energy estimate, we have
\begin{align*}
\Delta\urbPlEn^{\varepsilon,a,A}
&\leq\hdtwo(A)\sqrt\varepsilon\left[\left(\tfrac{1}{8c}+8c\right)\numLev+\sum_{k=1}^{\numLev}\tfrac{1}{2c}w_k^{-2}\varepsilon
+\tfrac{\tilde c}{\hat c}\right]\\
&\leq\hdtwo(A)\sqrt\varepsilon\left[\left(\tfrac{4}{45}+\tfrac{45}4\right)\tfrac12\left|\log_2\tfrac{\varepsilon^{1/2}}{a(a-1)}\right|+\tfrac{2}{c\tilde c^2}\sqrt\varepsilon\sum_{k=1}^{\numLev}4^{-k}
+\tfrac{\tilde c}{\hat c}\right]\\
&\leq\hdtwo(A)\sqrt\varepsilon C\left[\left|\log_2\tfrac{\varepsilon^{1/2}}{a(a-1)}\right|+a(a-1)+1\right]
\end{align*}
for some constant $C>0$, where we have used the condition on $\varepsilon$ and the bounds on $c,\tilde c,\hat c$.
\end{proof}

\begin{proof}[Proof of Theorem \ref{thm:upper_bound_for_urban_planning_minus_wasserstein_in_3D_for_big_a}]
This construction is intended for large values of $a$ so that unlike in the previous construction,
the mass is not discharged from the network at distance $H$ from the top boundary so as to move the last piece on its own.
Instead, the network of pipes goes all the way to the top within ${K_S}$ layers
and then even continues branching inside the boundary plane before discharging the mass.

We use the same definitions \eqref{eqn:parametersConstructionUrbPl} as before, only in layers ${K_S}+1$ through $\numLev$ we use
\begin{equation*}
h_k=0,\quad k={K_S}+1,\ldots,\numLev,\qquad H=0.
\end{equation*}
Applying Lemmas\,\ref{lem:upper_bound_for_elementary_cell_urban_planning_cost_in_nD} and \ref{lem:upper_bound_for_wasserstein_cell_in_3D} in \eqref{eqn:generalConstructionEstimate} we now obtain the energy estimate
\begin{align*}
\Delta\urbPlEn^{\varepsilon,a,A}
&\leq2\Bigg[\sum_{k=1}^{{K_S}}N_k^2\left(4(f_k+\varepsilon)\sqrt{\tfrac{w_k^2}{8}+h_k^2}-w_k^2h_k\right)
+\sum_{k={K_S}+1}^{\numLev}N_k^24(f_k+\varepsilon)\tfrac{w_k}{\sqrt8}
+N_{\numLev+1}^24af_{\numLev+1}\tfrac{w_{\numLev+1}}{\sqrt8}\Bigg]\\
&\leq\hdtwo(A)2\left[\sum_{k=1}^{{K_S}}\left(\tfrac{w_k^2}{16h_k}+\varepsilon\tfrac{h_k}{w_k^2}(4+\tfrac{w_k^2}{4h_k^2})\right)
+\sum_{k={K_S}+1}^{\numLev}(\tfrac{w_k}{\sqrt8}+\tfrac{4\varepsilon}{\sqrt8w_k})
+a\tfrac{w_{\numLev+1}}{\sqrt8}\right]\\
&=\hdtwo(A)\left[\sqrt\varepsilon\sum_{k=1}^{{K_S}}\left(\tfrac{1}{8c}+8c+\tfrac{1}{2c}w_k^{-2}\varepsilon\right)
+\sum_{k={K_S}+1}^{\numLev}(\tfrac{w_k}{\sqrt2}+\tfrac{4\varepsilon}{\sqrt2w_k})
+a\tfrac{w_{\numLev+1}}{\sqrt2}\right]
\end{align*}
for $\alpha=2$, $\beta=-\frac12$, $\gamma=0$, and $\delta=0$ as before.
Note that due to $h_k=0$ for $k>K_S$ the constant $c$ is here defined by \eqref{eqn:constantC} with $H=0$ and $\numLev$ replaced by $K_S$,
which does not enlarge its range, though, as long as we ensure $K_S\geq1$.
Again we still need to determine ${K_S}$ and $\numLev$, and here the heuristics are as follows.
Writing the energy contribution from the top two layers as
\begin{equation*}
\tfrac{w_{\numLev}}{\sqrt2}+\tfrac{4\varepsilon}{\sqrt2w_{\numLev}}+a\tfrac{w_{\numLev}}{2\sqrt2},
\end{equation*}
we see that the first term can be neglected in view of the third.
The minimising $w_\numLev$ thus satisfies $w_{\numLev}\sim\sqrt{\varepsilon/a}$.
Likewise, for the summands of the second sum we would prefer to have the minimising value $w_k\sim2\sqrt\varepsilon$.
This is not possible for all summands, of course, since the $w_k$ have to decrease down to $w_{\numLev}\sim\sqrt{\varepsilon/a}$,
but at least we can require $w_k\in[\sqrt\varepsilon,\sqrt{\varepsilon/a}]$ for all $k>{K_S}$ and $w_{K_S}\sim\sqrt\varepsilon$.
Thus, recalling $w_1=\tilde c\varepsilon^{\frac14}$ from \eqref{eqn:coarseWidth} we need to choose
\begin{equation*}
{K_S}=\left\lfloor\log_2(4\varepsilon^{-\frac14})\right\rfloor,\quad
\numLev={K_S}+\left\lfloor\log_2\sqrt a\right\rfloor
\end{equation*}
such that (using $\hat c=\frac{\varepsilon^{1/4}}42^{{K_S}}\in(\frac12,1]$ and $\bar c=2^{\numLev}/(2^{K_S}\sqrt a)\in(\frac12,1]$) we indeed have
\begin{equation*}
w_{K_S}=\tfrac{\tilde c}{4\hat c}\sqrt\varepsilon,\quad
w_{\numLev}=\tfrac{\tilde c}{4\hat c\bar c}\sqrt{\tfrac\varepsilon a}.
\end{equation*}
This way the energy estimate turns into
\begin{align*}
\Delta\urbPlEn^{\varepsilon,a,A}
&\leq\hdtwo(A)\left[\left(\tfrac{1}{8c}+8c\right)K\sqrt\varepsilon+\sum_{k=1}^{{K_S}}\tfrac{1}{2c}w_k^{-2}\varepsilon^{\frac32}
+\sum_{k=1}^{\numLev-{K_S}}(\tfrac{w_{K_S}}{2^k4\sqrt2}+\tfrac{16\varepsilon2^k}{\sqrt2w_{K_S}})
+\tfrac{\tilde c}{8\sqrt2\hat c\bar c}\sqrt{a\varepsilon}\right]\\
&=\hdtwo(A)\sqrt\varepsilon\left[\left(\tfrac{1}{8c}+8c\right)K+\tfrac{2}{c\tilde c^2}\sqrt\varepsilon\sum_{k=1}^{{K_S}}4^{-k}
+\sum_{k=1}^{\numLev-{K_S}}(\tfrac{\tilde c}{16\hat c2^k\sqrt2}+\tfrac{64\hat c2^k}{\sqrt2\tilde c})
+\tfrac{\tilde c}{8\sqrt2\hat c\bar c}\sqrt{a}\right]\\
&\leq\hdtwo(A)\sqrt\varepsilon\tfrac C2\left[K+\sqrt\varepsilon+2^{\numLev-{K_S}}+\sqrt{a}\right]\\
&\leq\hdtwo(A)\sqrt\varepsilon C\left[|\log_2\sqrt\varepsilon|+1+\sqrt{a}\right]\,,
\end{align*}
where all occurring constants have been subsumed in a large enough constant $C>0$.
\end{proof}

\subsubsection{Upper bound in $n$D for $n>3$}
The following two theorems treat the regimes of small $a$ and of large $a$, respectively.

\begin{theorem}[Upper bound on $\Delta\urbPlEn^{\varepsilon,a,A}$ in $n$D for small $a$]\label{thm:upBndUrbPlanHighDSmallA}
Using the notation from Section \ref{sec:introUrbPlan}, there exists a constant $C\geq0$ such that
for all $a\leq2$, $\varepsilon\leq\min\{1,\frac{\sqrt{a(a-1)}^{n+1}}{\sqrt2^{n^2-1}},\sqrt{\frac1{(2(n-1))^{n-1}}\frac{(a-1)^{n+1}}{a^{n-3}}},\hdcodimone(A)^{\frac{n+1}{n-1}}\}$ we have
\begin{equation*}
 \Delta\urbPlEn^{\varepsilon,a,A} \leq \hdcodimone(A)C\varepsilon^{\frac1{n-1}}\sqrt{a(a-1)}^{\frac{n-3}{n-1}}.
\end{equation*}
\end{theorem}
\begin{theorem}[Upper bound on $\Delta\urbPlEn^{\varepsilon,a,A}$ in $n$D for large $a$]\label{thm:upBndUrbPlanHighDLargeA}
There exists a constant $C\geq0$ such that
for all $a\geq2$, $\varepsilon\leq\min\{1,\hdcodimone(A)^{\frac{n+1}{n-1}}\}$ we have
\begin{equation*}
 \Delta\urbPlEn^{\varepsilon,a,A} \leq \hdcodimone(A)C\varepsilon^{\frac1{n-1}}\sqrt{a(a-1)}^{\frac{n-2}{n-1}}.
\end{equation*}
\end{theorem}
Both theorems can be summarised in the following corollary.
\begin{corollary}[Upper bound on $\Delta\urbPlEn^{\varepsilon,a,A}$ in $n$D]
There exists a constant $C\geq0$ such that
for all $\varepsilon\leq\min\{1,\hdcodimone(A)^{\frac{n+1}{n-1}}\}$ we have
\begin{equation*}
 \Delta\urbPlEn^{\varepsilon,a,A} \leq \hdcodimone(A)\min\left\{a-1,C\varepsilon^{\frac1{n-1}}\sqrt a\sqrt{a-1}^{\frac{n-3}{n-1}}\right\}.
\end{equation*}
\end{corollary}
\begin{proof}
The inequality $\Delta\urbPlEn^{\varepsilon,a,A} \leq \hdcodimone(A)(a-1)$ follows as for the two- and three-dimensional case.
Now let $a\geq1$ and $\varepsilon\leq\min\{1,\hdcodimone(A)^{\frac{n+1}{n-1}}\}$ be given.
Without loss of generality we may assume the constants $C$ in Theorems\,\ref{thm:upBndUrbPlanHighDSmallA} and \ref{thm:upBndUrbPlanHighDLargeA} to be identical
and to be larger than $\tilde C^{\frac{-1}{n-1}}$ for $\tilde C=\min\{{\sqrt2}^{1-n^2},{\sqrt{8(n-1)}}^{1-n}\}$.
If $a\geq2$, Theorem\,\ref{thm:upBndUrbPlanHighDLargeA} implies
$$\Delta\urbPlEn^{\varepsilon,a,A}\leq \hdcodimone(A)C\varepsilon^{\frac1{n-1}}\sqrt{a(a-1)}^{\frac{n-2}{n-1}} \leq \hdcodimone(A)C\varepsilon^{\frac1{n-1}}\sqrt a\sqrt{a-1}^{\frac{n-3}{n-1}}$$
so that we are done.
If $a<2$ and $\varepsilon$ satisfies the conditions in Theorem\,\ref{thm:upBndUrbPlanHighDSmallA}, then this theorem implies 
$$\Delta\urbPlEn^{\varepsilon,a,A}\leq \hdcodimone(A)C\varepsilon^{\frac1{n-1}}\sqrt{a(a-1)}^{\frac{n-3}{n-1}} \leq \hdcodimone(A)C\varepsilon^{\frac1{n-1}}\sqrt a\sqrt{a-1}^{\frac{n-3}{n-1}},$$
again finishing the proof.
Finally, if $a<2$, but the conditions on $\varepsilon$ from Theorem\,\ref{thm:upBndUrbPlanHighDSmallA} are violated,
then in particular we have $\varepsilon>\tilde C\sqrt{a(a-1)}^{n+1}$.
However, this implies $C\varepsilon^{\frac1{n-1}}\sqrt a\sqrt{a-1}^{\frac{n-3}{n-1}}\geq(a-1)$
so that the inequality to be proven reduces to $\Delta\urbPlEn^{\varepsilon,a,A} \leq \hdcodimone(A)(a-1)$,
which concludes the proof.
\end{proof}

\begin{proof}[Proof of Theorem\,\ref{thm:upBndUrbPlanHighDSmallA}]
In analogy to the proof of Theorem \ref{thm:upper_bound_for_urban_planning_minus_wasserstein_in_3D} we obtain
\begin{align}
\Delta\urbPlEn^{\varepsilon,a,A}
&\leq2\Bigg[\sum_{k=1}^{\numLev}N_k^{n-1}\left(2^{n-1}(f_k+\varepsilon)\sqrt{\tfrac{n-1}{16}w_k^2+h_k^2}-w_k^{n-1}h_k\right)\nonumber\\
&\hspace{10ex}+N_{\numLev+1}^{n-1}\left(2^{n-1}af_{\numLev+1}\sqrt{\tfrac{n-1}{16}w_{\numLev+1}^2+H^2}-w_{\numLev+1}^{n-1}H\right)\Bigg]\nonumber\\
&\leq\hdcodimone(A)2\Bigg[\sum_{k=1}^{\numLev}\left(\tfrac{(n-1)w_k^2}{32h_k}+\varepsilon\tfrac{h_k}{w_k^{n-1}}2^{n-1}(1+\tfrac{(n-1)w_k^2}{32h_k^2})\right)
+\left(H(a-1)+a\tfrac{(n-1)w_{\numLev+1}^2}{32H}\right)\Bigg]\nonumber\\
&=\hdcodimone(A)2\Bigg[\sum_{k=1}^{\numLev}\Big(\tfrac{(n-1)w_k^{2-\alpha}}{32c}a^{-\gamma}(a-1)^{-\delta}\varepsilon^{-\beta}+2^{n-1}cw_k^{\alpha-n+1}a^\gamma(a-1)^\delta\varepsilon^{1+\beta}(1+\tfrac{(n-1)w_k^2}{32h_k^2})\Big)\nonumber\\
&\hspace{10ex}+\sqrt{\tfrac{n-1}{8}a(a-1)}w_{\numLev+1}\Bigg]\,,\label{eqn:urbPlEst4D}
\end{align}
where we have chosen the minimising
\begin{equation*}
H=\sqrt{\tfrac{n-1}{32}\tfrac{a}{a-1}}w_{\numLev+1}
\end{equation*}
(for which we still have to verify $H\leq\frac14$).
Next we choose the exponents
\begin{equation*}
\alpha=\tfrac{n+1}2,\qquad
\beta=-\tfrac12,\qquad
\gamma=\delta=0
\end{equation*}
to balance all the terms in the sum, resulting in
\begin{equation*}
\Delta\urbPlEn^{\varepsilon,a,A}
\leq\hdcodimone(A)2\Bigg[\sum_{k=1}^{\numLev}w_k^{\frac{3-n}2}\sqrt\varepsilon\left(\tfrac{(n-1)}{32c}+2^{n-1}c(1+\tfrac{(n-1)w_k^2}{32h_k^2})\right)+\sqrt{\tfrac{n-1}{8}a(a-1)}w_{\numLev+1}\Bigg].
\end{equation*}
Note that by \eqref{eqn:coarseWidth} and \eqref{eqn:epsConstraint} this results in $w_1=\tilde c\sqrt[n+1]\varepsilon$ as well as the constraint $\varepsilon\leq\hdcodimone(A)^{\frac{n+1}{n-1}}$.

To fix $\numLev$ the basic idea is as follows.
Let us assume $h_k\geq w_k$ so that the term of the form $\frac{w_k^2}{h_k^2}$ can be neglected in our estimate.
In the above sum the $w_k$ occur with negative powers so that the sum evaluates to $\sim\sqrt\varepsilon w_{\numLev+1}^{\frac{3-n}2}$.
Balancing this with the term of the form $\sqrt{a(a-1)}w_{\numLev+1}$ we should have $w_{\numLev+1}\sim\sqrt{n-1}{\frac\varepsilon{a(a-1)}}$.
Thus we choose
\begin{equation*}
\numLev=\left\lfloor\log_2\left((a(a-1))^{\frac1{n-1}}\varepsilon^{-\frac2{n^2-1}}\right)\right\rfloor,
\end{equation*}
which due to our assumptions on $\varepsilon$ is strictly positive and thus admissible (the square brackets again denote the integer part).

With this choice (and using $\hat c=(a(a-1))^{-\frac1{n-1}}\varepsilon^{\frac2{n^2-1}}2^K\in(\frac12,1]$) we have
\begin{equation*}
w_{\numLev+1}=\frac{w_1}{2^\numLev}=\frac{\tilde c}{\hat c}\left(\frac\varepsilon{a(a-1)}\right)^{\frac1{n-1}}\text{ and }
H=\frac{\tilde c}{\hat c}\sqrt{\tfrac{n-1}{32}\tfrac{a}{a-1}}\left(\frac\varepsilon{a(a-1)}\right)^{\frac1{n-1}},
\end{equation*}
which indeed satisfies $H\leq\frac14$ due to our constraints on $\varepsilon$.
Thus, all conditions \eqref{eqn:feasibilityConditions} are fulfilled.

Note that $\frac{w_k^2}{h_k^2}\leq\frac{w_{\numLev+1}^2}{h_{\numLev+1}^2}=\frac1{c^2}(\frac{\hat c}{\tilde c})^{n-1}a(a-1)\leq\frac{2^n}{c^2}$
so that the energy estimate becomes
\begin{align*}
\Delta\urbPlEn^{\varepsilon,a,A}
&\leq\hdcodimone(A)2\Bigg[\sum_{k=1}^{\numLev}w_k^{\frac{3-n}2}\sqrt\varepsilon\left(\tfrac{(n-1)}{32c}+2^{n-1}c(1+\tfrac{(n-1)2^{n-5}}{c^2})\right)
+\sqrt{\tfrac{n-1}{8}a(a-1)}w_{\numLev+1}\Bigg]\\
&\leq\hdcodimone(A)\tfrac C2\left[w_{\numLev+1}^{\frac{3-n}2}\sqrt\varepsilon+\sqrt{a(a-1)}w_{\numLev+1}\right]\\
&\leq\hdcodimone(A)C\varepsilon^{\frac1{n-1}}(a(a-1))^{\frac{n-3}{2(n-1)}}
\end{align*}
for some $C>0$.
\end{proof}

\begin{proof}[Proof of Theorem\,\ref{thm:upBndUrbPlanHighDLargeA}]
For large $a$ the previous energy estimate is no longer correct since it is based on $h_k\gtrsim w_k$, which the construction will not always satisfy for large $a$.
This indicates that the term involving $\frac{w_k^2}{h_k^2}$ in \eqref{eqn:urbPlEst4D} will become dominant.
As before, the summands in \eqref{eqn:urbPlEst4D} will be negligible except for the $\numLev$\textsuperscript{th} term so that the energy behaves up to constant factors like
\begin{multline*}
\Delta\urbPlEn^{\varepsilon,a,A}
\sim\hdcodimone(A)\Bigg[\Big(w_\numLev^{2-\alpha}a^{-\gamma}(a-1)^{-\delta}\varepsilon^{-\beta}
+w_\numLev^{\alpha-n+1}a^\gamma(a-1)^\delta\varepsilon^{1+\beta}(1+\tfrac{w_\numLev^2}{h_\numLev^2})\Big)
+\sqrt{a(a-1)}w_{\numLev}\Bigg].
\end{multline*}
Let us assume the first term to be negligible.
If we choose optimal parameters, the remaining three terms should balance,
that is, we should have $1\sim\frac{w_\numLev^2}{h_\numLev^2}$ as well as $w_\numLev^{\alpha-n+1}a^\gamma(a-1)^\delta\varepsilon^{1+\beta}\sim\sqrt{a(a-1)}w_{\numLev}$.
The former results in $w_\numLev\sim h_\numLev$ or equivalently
$$w_\numLev\sim a^{\frac\gamma{1-\alpha}}(a-1)^{\frac\delta{1-\alpha}}\varepsilon^{\frac\beta{1-\alpha}}$$
while the latter is equivalent to
$$w_\numLev\sim\left(a^{\frac12-\gamma}(a-1)^{\frac12-\delta}\varepsilon^{-1-\beta}\right)^{\frac1{\alpha-n}}.$$
Equating both scalings and picking $\alpha$ as in the previous proof we obtain the exponents
\begin{equation*}
\alpha=\tfrac{n+1}2,\qquad
\beta=-\tfrac12,\qquad
\gamma=\delta=\tfrac14.
\end{equation*}
The above scaling of $w_\numLev$ together with $w_\numLev=w_12^{1-\numLev}=\tilde c\big(\frac\varepsilon{\sqrt{a(a-1)}}\big)^{\frac1{n+1}}2^{1-\numLev}$ now suggests to choose
\begin{equation*}
\numLev=\left\lfloor\log_2\left(2\tilde c\sqrt{n-1}(a(a-1))^{\frac1{n^2-1}}\varepsilon^{-\frac{2}{n^2-1}}\right)\right\rfloor\geq1.
\end{equation*}
For some constant $\hat c\in(\frac12,1]$ this then results in
\begin{equation*}
w_{\numLev+1}=\tfrac{1}{2\hat c\sqrt{n-1}}\Big(\tfrac\varepsilon{\sqrt{a(a-1)}}\Big)^{\frac1{n-1}}
\text{ and }
H=\tfrac{1}{\hat c\sqrt{128}}\sqrt{\tfrac a{a-1}}\Big(\tfrac\varepsilon{\sqrt{a(a-1)}}\Big)^{\frac1{n-1}}\leq\tfrac14,
\end{equation*}
satisfying \eqref{eqn:feasibilityConditions}.
Finally, the energy estimate \eqref{eqn:urbPlEst4D} turns into
\begin{equation*}
\Delta\urbPlEn^{\varepsilon,a,A}
\leq\hdcodimone(A)C\sqrt{a(a-1)}^{\frac{n-2}{n-1}}\varepsilon^{\frac1{n-1}}
\end{equation*}
for some constant $C>0$.
\end{proof}

\begin{remark}
It is a straightforward exercise to show that in the proof of Theorem\,\ref{thm:upBndUrbPlanHighDSmallA} we could have equally well used any exponent $\alpha\in(2,n-1)$
as well as $\beta=\frac{1-\alpha}{n-1}$, $\gamma=\delta=\frac{\alpha-1}{n-1}-\frac12$, and $\numLev\sim\log_2\left((a(a-1))^{\frac{n+1}{2\alpha(n-1)}}\varepsilon^{-\frac1{\alpha(n-1)}}\right)$.
Likewise, in the proof of Theorem\,\ref{thm:upBndUrbPlanHighDLargeA} one could also choose
$\alpha\in(2,n-1)$, $\beta=\frac{1-\alpha}{n-1}$, $\gamma=\delta=-\frac\beta2$, and $\numLev\sim\log_2\left((a(a-1))^{\frac1{2\alpha(n-1)}}\varepsilon^{-\frac1{\alpha(n-1)}}\right)$.
This tells us that the energy scaling does not give highly precise information about how optimal geometries and irrigation patterns look like,
since there are multiple rather different geometries with the same energy scaling.
\end{remark}

\subsection{Upper bound for optimal branched transport energy $\brTptEn^{\varepsilon,A}$}\label{sec:upBndBrnchTrpt}
Here the construction of an irrigation pattern with the desired energy scaling is similar to the urban planning case,
however, this time the mass cannot travel outside the pipe network.
This implies that no Wasserstein cells can be used and instead the pipe network has to refine infinitely.
Again we will assume $A=[0,\ell]^{n-1}$ without loss of generality.

\begin{theorem}[Upper bound on $\Delta\brTptEn^{\varepsilon,A}$]
Using the notation from Section \ref{subsec:branched_transport}, there exists $C>0$ independent of $\varepsilon,A$ such that for all $\varepsilon<\min\{\frac1{2(n-1)},\hdcodimone(A)^{\frac2{n-1}}\}$ we have
\begin{displaymath}
 \Delta\brTptEn^{\varepsilon,A} \leq \hdcodimone(A) C \varepsilon|\log\varepsilon|.
\end{displaymath}
\end{theorem}

\begin{proof}
Again, the width, number, and flux of elementary cells in the $k$\textsuperscript{th} layer are specified as
\begin{equation*}
w_k=2^{1-k}w_1,\quad
N_k=\tfrac{\ell}{w_k},\quad
f_k=\left(\tfrac{w_k}2\right)^{n-1}.
\end{equation*}
Furthermore, letting $i\in\{1,\ldots,N_k\}^{n-1}$ a multiindex, the height of an elementary cell and the $i$\textsuperscript{th} base point shall be
\begin{equation*}
h_k=\begin{cases}cw_k^\alpha\varepsilon^\beta,&k\leq\numLev,\\0,&k>\numLev,\end{cases}\quad
x_{k,i}=\left((i_1-\tfrac{1}{2})w_k,\ldots,(i_{n-1}-\tfrac12)w_k,\tfrac{1}{2} + \sum_{j=1}^{k-1} h_j\right)
\end{equation*}
for some constants $c,\alpha,\beta\in\R$, $K\in\N$ to be determined.
As in urban planning, the irrigation pattern of the $k$\textsuperscript{th} layer is given by
\begin{equation*}
 \chi_k = \coprod_{i\in\{1,\ldots,N_k\}^{n-1}} \chi_{x_{k,i},w_k,h_k,f_k}^{\mathrm E},
\end{equation*}
and these patterns are again compatible to be composed in series according to
\begin{equation*}
\chi_u=((\ldots((\chi_1\circ_\Id\chi_2)\circ_\Id\ldots\circ_\Id\chi_\numLev)\circ_\Id\chi_{\numLev+1})\circ_\Id\chi_{\numLev+2})\circ_\Id\ldots.
\end{equation*}
In contrast to urban planning this is an infinite series of patterns, but it is straightforward to see that it is actually well-defined (using, for example, that the lengths of the fibres converge).
As in the urban planning case, the full irrigation pattern is then defined as
\begin{equation*}
\chi=\chi_l\circ_\Id\chi_u,
\end{equation*}
where $\chi_l$ describes the lower half and is constructed analogously to $\chi_u$.

Since the total height of the construction is $1$ we have
\begin{equation*}
1=2\sum_{k=1}^\infty h_k
=2\sum_{k=1}^\numLev cw_k^\alpha\varepsilon^\beta
=2c(2w_1)^\alpha\varepsilon^\beta\sum_{k=1}^\numLev2^{-k\alpha}
=2cw_1^\alpha\varepsilon^\beta\frac{1-2^{-\numLev\alpha}}{1-2^{-\alpha}},
\end{equation*}
which implies that we should choose
\begin{equation*}
w_1=\tilde c\varepsilon^{-\frac\beta\alpha},\qquad
c=\tfrac1{2\tilde c^\alpha}\tfrac{1-2^{-\alpha}}{1-2^{-\numLev\alpha}}\in[\tfrac{1-2^{-\alpha}}2,2^{\alpha-1}],
\end{equation*}
where we have used $\numLev\geq1$, $\alpha>0$ (which will be verified later)
and where $\tilde c\in(\tfrac12,1]$ is such that the total geometry width $\ell$ is a multiple of $w_1$.
Note that this requires $\varepsilon\leq\ell^{-\frac\alpha\beta}=\hdcodimone(A)^{-\frac\alpha{\beta(n-1)}}$.

Next let us specify $\numLev$.
The idea behind taking $h_k=0$ for all $k>\numLev$ is to separate all elementary cell layers
into those for which a substantial vertical mass flux takes place and those for which the vertical mass flux can be neglected compared to the horizontal one.
Thus we should have $h_k\gtrsim w_k$ for $k\leq\numLev$ and $h_\numLev\sim w_\numLev$ or equivalently $w_\numLev\sim\varepsilon^{\frac\beta{1-\alpha}}$.
Due to $w_\numLev=w_12^{1-\numLev}=\tilde c\varepsilon^{-\frac\beta\alpha}2^{1-\numLev}$ this suggest to fix
\begin{equation*}
\numLev=1+\left\lfloor\log_2\left(2\tilde c\varepsilon^{\frac\beta{\alpha(\alpha-1)}}\right)\right\rfloor
\text{ and thus }
w_\numLev=\tfrac1{2\hat c}\varepsilon^{-\frac\beta{\alpha-1}}
\end{equation*}
for some $\hat c\in(\frac12,1]$, where we have $\numLev\geq1$ as long as $\alpha>1$ and $\beta\leq0$.

It remains to specify $\alpha$ and $\beta$ and to estimate the energy of $\chi$.
As in the urban planning case we exploit the fact
\begin{displaymath}
 \brTptEn^{*,A}=\hdone(A) = 2\sum_{k = 1}^{\numLev} N_k^{n-1} w_k^{n-1} h_k
\end{displaymath}
as well as Remark\,\ref{rem:cellOpProps} and Lemma\,\ref{lem:upper_bound_for_elementary_cell_urban_planning_cost_in_nD} to obtain
\begin{align*}
\Delta\brTptEn^{\varepsilon,A}
&=2\MMSEn^{\varepsilon}[\chi_u]-\brTptEn^{*,A}\\
&=2\bigg[\sum_{k=1}^{\numLev}\sum_{i\in\{1,\ldots,N_k\}^{n-1}}\left(\MMSEn^{\varepsilon}[\chi_{x_{k,i},w_k,h_k,f_k}^{\mathrm E}]-w_k^{n-1}h_k\right)
+\sum_{k=\numLev+1}^{\infty}\sum_{i\in\{1,\ldots,N_k\}^{n-1}}\MMSEn^{\varepsilon}[\chi_{x_{k,i},w_k,h_k,f_k}^{\mathrm E}]\bigg]\\
&\leq2\bigg[\sum_{k=1}^{\numLev}N_k^{n-1}\left(2^{n-1}(\tfrac{w_k}2)^{(n-1)(1-\varepsilon)}\sqrt{\tfrac{n-1}{16}w_k^2+h_k^2}-w_k^{n-1}h_k\right)\\&\hspace{20ex}
+\sum_{k=\numLev+1}^{\infty}N_k^{n-1}2^{n-1}(\tfrac{w_k}2)^{(n-1)(1-\varepsilon)}\tfrac{\sqrt{n-1}}{4}w_k\bigg]\\
&\leq2\hdcodimone(A)\bigg[\sum_{k=1}^{\numLev}h_k\left((\tfrac{w_k}2)^{-\varepsilon(n-1)}\sqrt{\tfrac{n-1}{16}\tfrac{w_k^2}{h_k^2}+1}-1\right)
+\sum_{k=\numLev+1}^{\infty}(\tfrac{w_k}2)^{-\varepsilon(n-1)}\tfrac{\sqrt{n-1}}{4}w_k\bigg].
\end{align*}
Note that we can estimate
\begin{multline*}
\log\left((\tfrac{w_k}2)^{-\varepsilon(n-1)}\right)
\leq\log\left((\tfrac{w_\numLev}2)^{-\varepsilon(n-1)}\right)
\leq(n-1)\varepsilon\left(\log(4\hat c)+|\tfrac\beta{\alpha-1}||\log\varepsilon|\right)
\leq\log4+(n-1)|\tfrac\beta{\alpha-1}|.
\end{multline*}
Abbreviating the right-hand side as $Z$, we now employ the inequalities
\begin{equation*}
\sqrt{1+z}\leq1+\tfrac z2\;\forall z\geq0
\qquad\text{ and }\qquad
e^{\tilde z}\leq1+\bar c\tilde z\;\forall\tilde z\in[0,Z]
\end{equation*}
for $\bar c=\frac{e^Z-1}{Z}$.
Using $z=\tfrac{n-1}{16}\tfrac{w_k^2}{h_k^2}$ and $\tilde z=\log\left((\tfrac{w_k}2)^{-\varepsilon(n-1)}\right)$, the energy estimate turns into
\begin{align*}
\Delta\brTptEn^{\varepsilon,A}
&\leq2\hdcodimone(A)\bigg[\sum_{k=1}^{\numLev}h_k\left(\left(1-\bar c\varepsilon(n-1)\log(\tfrac{w_k}2)\right)\left(\tfrac{n-1}{32}\tfrac{w_k^2}{h_k^2}+1\right)-1\right)
+\tfrac{\sqrt{n-1}}{4}\sum_{k=\numLev+1}^{\infty}(\tfrac{w_k}2)^{1-\varepsilon(n-1)}\bigg]\\
&\leq2\hdcodimone(A)\bigg[\bar c\varepsilon(n-1)\left|\log\tfrac{w_K}2\right|\sum_{k=1}^{\numLev}h_k
+\left(\tfrac{n-1}{32}+Z\right)\sum_{k=1}^{\numLev}\tfrac{w_k^2}{h_k}
+\tfrac{\sqrt{n-1}}{4}\sum_{k=1}^{\infty}(\tfrac{w_\numLev}{2^k})^{1-\varepsilon(n-1)}\bigg]\\
&=2\hdcodimone(A)\bigg[\tfrac{\bar c\varepsilon(n-1)}2\left|\log\tfrac1{4\tilde c}\varepsilon^{-\frac\beta{\alpha-1}}\right|
+\left(\tfrac{n-1}{32}+Z\right)(2w_1)^{2-\alpha}\varepsilon^{-\beta}\sum_{k=1}^{\numLev}2^{-k(2-\alpha)}\\
&\hspace{10ex}+\tfrac{\sqrt{n-1}}{4}\left(\tfrac1{2\tilde c}\varepsilon^{-\frac\beta{\alpha-1}}\right)^{1-\varepsilon(n-1)}\tfrac1{2^{1-\varepsilon(n-1)}-1}\bigg].
\end{align*}
Choosing
\begin{equation*}
\alpha=2,\qquad\beta=-1,
\end{equation*}
and letting $\tilde C,C>0$ denote constants independent of $\varepsilon,A$, this turns into
\begin{align*}
\Delta\brTptEn^{\varepsilon,A}
&\leq\hdcodimone(A)\tilde C\varepsilon[|\log\varepsilon|+K+\varepsilon^{-(n-1)\varepsilon}]
\leq\hdcodimone(A)C\varepsilon|\log\varepsilon|
\end{align*}
as desired.
\end{proof}

\section{Lower bound based on convex analysis}\label{sec:lowerBound}
In this section we prove the lower bounds from Theorems\,\ref{thm:scalingUrbPlan} and \ref{thm:scalingBrnchTrpt}.
The technique employed is based on the one briefly explained in the introduction and heavily exploits convex duality.
In particular, we will several times need the following convex duality bound on the Wasserstein distance.

\subsection{Bound on Wasserstein distance}

\begin{lemma}\label{thm:WoneBound}
Given $s,t\in\R$, let $\mu=g\hdcodimone\restr\{x_n=s\}$ for some measurable function $g:\{x_n=s\}\to[0,\m]$
and $\nu$ be a non-negative discrete measure with $\spt\nu\subset\{x_n=t\}$ and $\hd^0(\spt\nu)=N$.
Then there is a positive constant $C\equiv C(n)$ such that
\begin{equation*}
\Wdone(\mu,\nu)\geq|t-s|\|\nu\|_\fbm+C\|\nu\|_\fbm\min\{\tfrac{R^2}{|t-s|},R\}
\end{equation*}
for $R=\sqrt[n-1]{\frac{\|\nu\|_\fbm}{\m N\omega_{n-1}}}$
the radius of a single disk if $\mu$ were a uniform measure of density $\m$ on $N$ disks of equal size.
\end{lemma}
\begin{proof}
If $\mu$ and $\nu$ have different mass, $\Wdone(\mu,\nu)=\infty$ and there is nothing to prove,
so we assume $\|\mu\|_{\fbm}=\|\nu\|_{\fbm}$.
Also, without loss of generality we shall assume $s=0$ and $t>0$.

Let us define the test function
\begin{equation*}
\psi(x)=\dist(x,\spt\nu)\,.
\end{equation*}
The dual formulation of the Wasserstein distance (Theorem\,\ref{thm:WoneDual}) yields
\begin{multline*}
\Wdone(\mu,\nu)-t\|\nu\|_\fbm
\geq\int_{\R^n}\psi\,\de(\mu-\nu)-t\|\nu\|_\fbm
=\int_{\R^n}\dist(x,\spt\nu)-t\,\de\mu
=\int_0^\infty\mu(\R^n\setminus B_{t+r}(\spt\nu))\,\de r\,,
\end{multline*}
where we employed the layer cake formula.
Note that the integrand is bounded below by
\begin{equation*}
\mu(\R^n)-\mu(B_{t+r}(\spt\nu))
\geq\max\left\{0,\|\nu\|_{\fbm}-\m N\omega_{n-1}\left(\sqrt{(t+r)^2-t^2}\right)^{n-1}\right\}
=:M(r)\,.
\end{equation*}
Introducing the new variable $\rho=\sqrt{(t+r)^2-t^2}$
and the maximum radius $R=\left(\frac{\|\nu\|_\fbm}{\m N\omega_{n-1}}\right)^{\frac1{n-1}}$ we thus obtain the estimate
\begin{multline*}
\Wdone(\mu,\nu)-t\|\nu\|_\fbm
\geq\int_0^\infty M(r)\,\de r
=\int_0^R\frac{\|\nu\|_{\fbm}-\m N\omega_{n-1}\rho^{n-1}}{\sqrt{1+(\tfrac t\rho)^2}}\,\de\rho
=\m N\omega_{n-1}\int_0^R\frac{R^{n-1}-\rho^{n-1}}{\sqrt{1+(\tfrac t\rho)^2}}\,\de\rho\,.
\end{multline*}
Now if $R\leq2t$, we have $\sqrt{1+(\frac t\rho)^2}\leq\sqrt5\frac t\rho$ for all $\rho\leq R$ so that
\begin{equation*}
\Wdone(\mu,\nu)-t\|\nu\|_\fbm
\geq\frac{\m N\omega_{n-1}}{\sqrt5t}\int_0^RR^{n-1}\rho-\rho^{n}\,\de\rho
=\frac{n-1}{\sqrt5(2n+2)}\|\nu\|_\fbm\frac{R^2}t\,,
\end{equation*}
while for $R>2t$ we obtain
\begin{equation*}
\Wdone(\mu,\nu)-t\|\nu\|_\fbm
\geq\m N\omega_{n-1}\int_{R/2}^R\frac{R^{n-1}-\rho^{n-1}}{\sqrt2}\,\de\rho
=\frac{1+(n-2)2^{n-1}}{2^n\sqrt2n}\|\nu\|_\fbm R\,.
\end{equation*}
\end{proof}

We are now in a situation to proceed with the lower bound proofs.

\subsection{Lower bound for $\urbPlEn^{\varepsilon,a,A}$}\label{sec:lwBndUrbPlan}
In this section we prove the lower bound from Theorem\,\ref{thm:scalingUrbPlan}.
We will treat the different dimensions separately, as different phenomena occur in different dimensions.
In two dimensions, the bulk of the excess energy is contributed from the centre of the domain $A\times[0,1]$,
while in dimension greater than three, the dominant part of the excess energy stems from boundary layers at $x_n=0$ and $x_n=1$.
The three-dimensional situation represents an intermediate case, where the dominant energy contribution is distributed all over $A\times[0,1]$.
In all cases the proofs follow variations of the same theme, namely the classical technique from \cite{KoMu94}.
In two dimensions, where the energy is concentrated near the centre,
we will bound the excess energy from below using the relaxed energy between $\mu_0,\mu_1$ and a generic cross-section $\{x_2=t\}$ with $t\sim\frac12$.
In three dimensions, where the energy is distributed over different layers, we will employ the same technique, only now considering various cross-sections.
In higher dimensions, where the energy is concentrated near the boundary, we will apply the technique to a cross-section near the boundary.
The estimates require lower bounds of the urban planning cost in terms of the Wasserstein distance.
In detail, for two finite Borel measures $\mu_+,\mu_-$ of equal mass and an optimal connecting network $\SigmaOpt$ we will use the bound
\begin{equation}\label{eqn:W1bound}
\min_\chi\urbPlEn^{\varepsilon,a,\mu_+,\mu_-}[\chi]
=\urbPlEn^{\varepsilon,a,\mu_+,\mu_-}[\SigmaOpt]
\geq\Wd{d_\SigmaOpt}(\mu_+,\mu_-)
\geq\Wdone(\mu_+,\mu_-)\,.
\end{equation}

In the following we denote by $\SigmaOpt$ and $\chiOpt$ the optimal network and irrigation pattern for $\urbPlEn^{\varepsilon,a,A}$, respectively,
which are related as detailed in Section\,\ref{sec:introUrbPlan}.
Furthermore, we use the following type of pattern reparameterisation.
For $s\in(0,1)$, $p\in\reSpace$ we define the \emph{section crossing time} and the \emph{section crossing point},
\begin{align}
t_p(s)&=\min\left\{t\in I\,:\,\chiOpt_p(t)\in\{x_n=s\}\right\}\,,\label{eq:crossing_time}\\
\xiOpt(p,s)&=\chiOpt_p(t_p(s))\,,\label{eq:crossing_point}
\end{align}
to be the time when particle $p$ reaches the cross-section $\{x_n=s\}$ for the first time
and its position at that time, respectively.
Both the above are well-defined since for almost all $p\in\reSpace$, $\chiOpt_p\in\AC(I)$ and $\chiOpt_p(t)\in\{x_n=t\}$ for $t=0,1$.
If $\chiOpt_p$ visits each cross-section just once, $\xiOpt_p=\xiOpt(p,\cdot)$ is indeed a reparameterisation of $\chiOpt_p(I)$, else it just reparameterises part of $\chiOpt_p(I)$.
In addition we set
\begin{align*}
\reSpace_s&=\{p\in\reSpace\,:\,\xiOpt_p(s)\in\SigmaOpt\}\,,\\
\Delta F_s&=\reMeasure(\reSpace\setminus\reSpace_s)\,,\\
\mu_s&=\pushforward{\xiOpt(\cdot,s)}{(\reMeasure\restr\reSpace_s)}\,,\\
N_s&=\hd^0(\spt\mu_s)
\end{align*}
to be the set of particles travelling through the network $\SigmaOpt$,
the amount of particles bypassing the network,
the mass distribution inside the network,
and the number of network pipes,
all at height $s$.
Note that $\reSpace_s$ is well-defined up to a null set, that $N_s$ might possibly be infinite, and that $\|\mu_s\|_\fbm = \reMeasure(\reSpace)-\Delta F_s$.

Bounds on $N_s$ and $\Delta F_s$ can be established via the following estimate.
First note that
\begin{align}
\int_I r_{\varepsilon,a}^{\chiOpt}(\chiOpt(p,t))|\dot\chiOpt(p,t)|\,\de t
 &\geq \hdone(\chiOpt_p(I))+\int_{\{\chiOpt_p(I) \cap \SigmaOpt\}} \frac{\varepsilon}{m_{\chiOpt}(x)}\,\de\hdone(x) + \int_{\{\chiOpt_p(I) \setminus \SigmaOpt\}} a-1\,\de\hdone(x)\nonumber\\
  &\geq 1+\varepsilon\int_{\{\chiOpt_p(I) \cap \SigmaOpt\}} \frac{1}{m_{\chiOpt}(x)}\,\de\hdone(x) + (a-1)\int_{\{\xiOpt_p(I) \setminus \SigmaOpt\}} \,\de\hdone(x)\,.\label{eq:inequality_one}
\end{align}
Note also that
\begin{multline}\label{eq:inequality_two}
\int_{\reSpace}\int_{\{\chiOpt_p(I) \cap \SigmaOpt\}} \frac{1}{m_{\chiOpt}(x)}\,\de\hdone(x)\,\de \reMeasure(p)
=\int_\SigmaOpt \int_{\{p \in \reSpace \ : \ x \in \chiOpt_p(I)\}} \frac{1}{m_{\chiOpt}(x)} \,\de \reMeasure(p) \,\de \hdone(x)\\
=\int_\SigmaOpt \frac{m_{\chiOpt}(x)}{m_{\chiOpt}(x)} \,\de \hdone(x)
=\hdone(\SigmaOpt)
\geq\int_0^1N_s\,\de s.
\end{multline}
We also have
\begin{multline}\label{eq:inequality_three}
\int_\reSpace\int_{\{\xiOpt_p(I) \setminus \SigmaOpt\}} \,\de\hdone(x)\,\de \reMeasure(p)
\geq\int_\reSpace\int_0^1\setchar{\R^n\setminus\SigmaOpt}(\xiOpt_p(s)) \,\de s\,\de \reMeasure(p)\\
=\int_0^1\int_{\reSpace\setminus\reSpace_s}\,\de \reMeasure(p) \,\de s
=\int_0^1\Delta F_s\,\de s\,.
\end{multline}
Integrating inequality \eqref{eq:inequality_one} over $\reSpace$ and using inequalities \eqref{eq:inequality_two} and \eqref{eq:inequality_three} we thus obtain
\begin{equation*}
\urbPlEn^{\varepsilon,a,A}[\chiOpt]
\geq\reMeasure(\reSpace)+\int_0^1\varepsilon N_s+(a-1)\Delta F_s\,\de s\,.
\end{equation*}

\subsubsection{Lower bound in 2D}
The lower bound derived here actually holds in any dimension, only it is not sharp for $n>2$.
Therefore we will keep the more general notation $n$ for the dimension.

The proof proceeds in the earlier described way:
we first state some properties of a generic cross-section
and then solve the relaxed problem below and above that cross-section, exploiting the cross-section properties.
The cross-section properties are bounds in terms of the excess energy $\Delta \urbPlEn^{\varepsilon,a,A}$.

\begin{lemma}\label{thm:crossSection2D}
There is a generic $s\in(0,1)$ such that
\begin{equation*}
N_s\leq\tfrac{\Delta \urbPlEn^{\varepsilon,a,A}}\varepsilon
\quad\text{and}\quad
\Delta F_s\leq\tfrac{\Delta \urbPlEn^{\varepsilon,a,A}}{a-1}\,.
\end{equation*}
\end{lemma}
\begin{proof}
For a proof by contradiction, assume that almost all cross-sections $s\in(0,1)$ satisfy
$N_s>\tfrac{\Delta \urbPlEn^{\varepsilon,a,A}}\varepsilon$
or $\Delta F_s>\tfrac{\Delta \urbPlEn^{\varepsilon,a,A}}{a-1}$.
Then
\begin{align*}
\urbPlEn^{\varepsilon,a,A}[\chiOpt]
&\geq\reMeasure(\reSpace)+\int_0^1\varepsilon N_s+(a-1)\Delta F_s\,\de s\\
&=\hdone(A)+\int_0^1\varepsilon N_s+(a-1)\Delta F_s\,\de s\\
&\geq \urbPlEn^{*,a,A}+\int_0^1 \max\{\varepsilon N_s,(a-1)\Delta F_s\}\de s\\
&>\urbPlEn^{*,a,A}+\Delta \urbPlEn^{\varepsilon,a,A}\,,
\end{align*}
the desired contradiction.
\end{proof}

\begin{proposition}
For $\varepsilon\leq1$ and $a>1$ we have
\begin{displaymath}
 \Delta \urbPlEn^{\varepsilon,a,A} \gtrsim \hdcodimone(A)\min\{\varepsilon^{\frac2{n+1}},a-1\}\,.
\end{displaymath}
\end{proposition}

\begin{proof}
Let $s\in(0,1)$ be the generic cross-section of the previous lemma, and let
\begin{align}
\tilde\mu_0&=\pushforward{\chiOpt(\cdot,0)}{(\reMeasure\restr\reSpace_s)}\,,\label{eqn:measureLow}\\
\tilde\mu_1&=\pushforward{\chiOpt(\cdot,1)}{(\reMeasure\restr\reSpace_s)}\label{eqn:measureHigh}
\end{align}
be the irrigating and the irrigated measure, respectively, of all particles in $\reSpace_s$.
Abbreviating $R=\sqrt[n-1]{\frac{\|\mu_s\|_\fbm}{N_s\omega_{n-1}}}$, we have (see Remark \ref{rem:cellOpProps})
\begin{align*}
\urbPlEn^{\varepsilon,a,A}[\chiOpt]
&= \urbPlEn^{\varepsilon,a,\tilde\mu_0,\mu_s}[\chiOpt|_{\reSpace_s\times[0,s]}]+\urbPlEn^{\varepsilon,a,\mu_s,\tilde\mu_1}[\chiOpt|_{\reSpace_s\times[s,1]}]+\urbPlEn^{\varepsilon,a,\mu_0-\tilde\mu_0,\mu_1-\tilde\mu_1}[\chiOpt|_{(\reSpace\setminus\reSpace_s)\times I}]\\
&\geq \Wdone(\tilde\mu_0,\mu_s)+\Wdone(\mu_s,\tilde\mu_1)+\Wdone(\mu_0-\tilde\mu_0,\mu_1-\tilde\mu_1)\\
&\geq s\|\mu_s\|_\fbm+C\|\mu_s\|_\fbm\min\{\tfrac{R^2}{s},R\}
+(1-s)\|\mu_s\|_\fbm+C\|\mu_s\|_\fbm\min\{\tfrac{R^2}{1-s},R\}+\|\mu_0-\tilde\mu_0\|_\fbm\\
&=\|\mu_0\|_\fbm+CR\|\mu_s\|_\fbm\left(\min\{\tfrac{R}{s},1\}+\min\{\tfrac{R}{1-s},1\}\right)\\
&\geq \urbPlEn^{*,a,A}+CR\|\mu_s\|_\fbm\min\{2R,1\}\,,
\end{align*}
where we have employed \eqref{eqn:W1bound} and Lemma\,\ref{thm:WoneBound}.
Now the two cases $R>\frac12$ and $R\leq\frac12$ are treated separately.
For $R>\frac12$, the above inequality turns into
\begin{equation*}
\Delta \urbPlEn^{\varepsilon,a,A}
\geq\tfrac C2\|\mu_s\|_\fbm
=\tfrac C2(\hdcodimone(A)-\Delta F_s)
\geq\tfrac C2\hdcodimone(A)-\tfrac{C}{2(a-1)}\Delta \urbPlEn^{\varepsilon,a,A}
\end{equation*}
by virtue of Lemma\,\ref{thm:crossSection2D},
which implies
\begin{equation*}
\Delta \urbPlEn^{\varepsilon,a,A}\geq\hdcodimone(A)/(\tfrac2C+\tfrac1{a-1})\gtrsim\min\{1,a-1\}\hdcodimone(A)\,.
\end{equation*}
If $R\leq\tfrac12$, the above inequality turns into
\begin{equation*}
\Delta \urbPlEn^{\varepsilon,a,A}\geq 2CR^2\|\mu_s\|_\fbm = 2C\|\mu_s\|_\fbm^{1+\frac2{n-1}}\left(\tfrac1{N_s\omega_{n-1}}\right)^{\frac2{n-1}}
\gtrsim\|\mu_s\|_\fbm^{1+\frac2{n-1}}(\tfrac\varepsilon{\Delta \urbPlEn^{\varepsilon,a,A}})^{\frac2{n-1}}\,,
\end{equation*}
which can be solved for $\Delta \urbPlEn^{\varepsilon,a,A}$ to yield
\begin{equation*}
\Delta \urbPlEn^{\varepsilon,a,A}
\gtrsim\|\mu_s\|_\fbm\varepsilon^{\frac2{n+1}}
\geq\hdcodimone(A)\varepsilon^{\frac2{n+1}}-\tfrac1{a-1}\Delta \urbPlEn^{\varepsilon,a,A}\varepsilon^{\frac2{n+1}}
\end{equation*}
by Lemma\,\ref{thm:crossSection2D}.
This finally implies
\begin{equation*}
\Delta \urbPlEn^{\varepsilon,a,A}
\gtrsim\hdcodimone(A)/(\varepsilon^{-\frac2{n+1}}+\tfrac1{a-1})
\gtrsim\hdcodimone(A)\min\{\varepsilon^{\frac2{n+1}},a-1\}\,.
\end{equation*}
\end{proof}

\subsubsection{Lower bound in 3D}
The previous estimate in three dimensions becomes $\Delta \urbPlEn^{\varepsilon,a,A}\gtrsim\hdtwo(A)\min\{\sqrt\varepsilon,a-1\}$.
However, the upper bound and its proof suggest that we are missing an additional factor $|\log\varepsilon|$,
which comes from the fact that---when the optimal network $\SigmaOpt$ refines from the centre $x_3=\frac12$ to the boundary---every refinement layer contributes the same excess energy.
Therefore, it no longer suffices to consider a single cross-section and the excess energy induced by it,
but the previous argument has to be refined to take into account multiple cross-sections.
The following lemma characterises the different cross-sections.

\begin{lemma}\label{thm:crossSection3D}
For any $K>0$ there are $s_k\in[2^{-k-1},2^{-k}]$, $k=1,\ldots,2K$, with $s_k-s_{k+1}>2^{-k-2}$ and such that at least half of the $s_k$ satisfy
\begin{equation*}
N_{s_k}\leq4\frac{\Delta \urbPlEn^{\varepsilon,a,A}}{K\varepsilon{s_k}}
\qquad\text{and}\qquad
\Delta F_{s_k}\leq4\frac{\Delta \urbPlEn^{\varepsilon,a,A}}{K(a-1)s_k}\,.
\end{equation*}
\end{lemma}
\begin{proof}
Consider the dyadic intervals $J_k=[2^{-k-1},2^{-k}]$, $k=1,\ldots,2K$,
and denote by $J\subset\{1,\ldots,2K\}$ the set of indices $k$ for which
$N_s>4\frac{\Delta \urbPlEn^{\varepsilon,a,A}}{K\varepsilon s}$
or $\Delta F_s>4\frac{\Delta \urbPlEn^{\varepsilon,a,A}}{K(a-1)s}$
for no less than half the $s\in J_k$.
Then we estimate
\begin{equation*}
\Delta \urbPlEn^{\varepsilon,a,A}
\geq\sum_{k=1}^{2K}\int_{J_k}\varepsilon N_s+(a-1)\Delta F_s\,\de s
>\sum_{k\in J}\int_{3\cdot2^{-k-2}}^{2^{-k}}4\tfrac{\Delta \urbPlEn^{\varepsilon,a,A}}{Ks}\,\de s
=\frac{|J|}K4\log\tfrac43\Delta \urbPlEn^{\varepsilon,a,A}\,,
\end{equation*}
analogously to Lemma\,\ref{thm:crossSection2D}.
This can only be satisfied if the number $|J|$ of elements in $J$ is less than $K$.
Thus we can pick at least $K$ cross-sections $s_k\in J_k$ with the desired property.
\end{proof}

\begin{proposition}
In $n=3$ dimensions, for $\varepsilon<1$ and $a>1$ we have
\begin{displaymath}
 \Delta \urbPlEn^{\varepsilon,a,A} \gtrsim\hdtwo(A)\min\left\{\sqrt\varepsilon\left(\sqrt a+\left|\log\tfrac{a-1}{\sqrt\varepsilon}\right|\right),a-1\right\}\,.
\end{displaymath}
\end{proposition}

\begin{proof}
Let us assume
\begin{equation*}
\Delta \urbPlEn^{\varepsilon,a,A}\leq2^{-13}\hdtwo(A)\min\{1,a-1\}\,.
\end{equation*}
We now consider cross-sections $\{x_3=s_k\}$, $k=1,\ldots,K$,
where $\frac12\geq s_1>\ldots>s_K\geq2^{-2K-1}$ denote the $K$ selected cross-sections from the previous lemma and $K\geq1$ is chosen such that
\begin{equation*}
128\frac a{a-1}\frac{\Delta \urbPlEn^{\varepsilon,a,A}}{\hdtwo(A)}\in[2^{-2K-2},2^{-2K-1}]\,,
\end{equation*}
which is possible since $\Delta\urbPlEn^{\varepsilon,a,A}\leq2^{-13}\hdtwo(A)\min\{a-1,1\}$.
The total amount of particles bypassing $\SigmaOpt$ on at least one of the cross-sections $\{x_3=s_k\}$, $k=1,\ldots,K$, can be bounded above by
\begin{equation*}
\sum_{k=1}^K\Delta F_{s_k}
\leq\frac{4\Delta \urbPlEn^{\varepsilon,a,A}}{K(a-1)}\sum_{k=1}^K\frac1{s_k}
\leq\frac{16\Delta \urbPlEn^{\varepsilon,a,A}}{K(a-1)s_K}
\leq\frac18\hdtwo(A)\,.
\end{equation*}
Likewise, we can bound the set $\Lambda_k\subset\reSpace$ of particles that between the cross-sections $\{x_3=s_{k+1}\}$ and $\{x_3=s_k\}$ travel horizontally by more than $s_k-s_{k+1}$.
Indeed, we have
\begin{multline*}
\urbPlEn^{\varepsilon,a,A}[\chiOpt]
=\int_{\reSpace\times I}r^{\varepsilon,a}_{\chiOpt}(\chiOpt(p,t))|\dot\chiOpt(p,t)|\,\de t\,\de\reMeasure(p)
\geq\int_\reSpace\hdone(\chiOpt_p(I))\,\de\reMeasure(p)\\
\geq\int_\reSpace1+(\sqrt2-1)(s_k-s_{k+1})\setchar{\Lambda_k}(p)\,\de\reMeasure(p)
=\reMeasure(\reSpace)+(\sqrt2-1)(s_k-s_{k+1})\reMeasure(\Lambda_k)
\end{multline*}
so that $\Delta\urbPlEn^{\varepsilon,a,A}=\urbPlEn^{\varepsilon,a,A}[\chiOpt]-\reMeasure(\reSpace)\geq(\sqrt2-1)(s_k-s_{k+1})\reMeasure(\Lambda_k)$.
Thus,
\begin{equation*}
\sum_{k=1}^K\reMeasure(\Lambda_k)
\leq\frac{\Delta \urbPlEn^{\varepsilon,a,A}}{\sqrt2-1}\sum_{k=1}^K\frac1{s_k-s_{k+1}}
\leq\frac{8\Delta \urbPlEn^{\varepsilon,a,A}}{(\sqrt2-1)s_K}
\leq\frac{\hdtwo(A)}4\,.
\end{equation*}
In the following, we will only consider particles that on each cross-section flow through $\SigmaOpt$
\emph{and} that between any $\{x_3=s_{k+1}\}$ and $\{x_3=s_k\}$ travel horizontally by no more than $s_k-s_{k+1}$,
that is, particles in
\begin{equation*}
\bar\reSpace=\bigcap_{k=1}^K(\reSpace_{s_k}\setminus\Lambda_k)
\qquad\text{with}\quad
\reMeasure(\bar\reSpace)\geq\reMeasure(\reSpace)-\sum_{k=1}^K\Delta F_{s_k}-\sum_{k=1}^K\reMeasure(\Lambda_k)\geq\frac58\hdtwo(A)\,.
\end{equation*}

Now, for $s\in[0,1]$ let
\begin{equation*}
\tilde\mu_s=\pushforward{\xiOpt(\cdot,s)}{(\reMeasure\restr\bar\reSpace)}
\end{equation*}
be the mass flux distribution on $\{x_3=s\}$ of all above selected particles.
We have $\|\tilde\mu_0\|_\fbm=\|\tilde\mu_1\|_\fbm=\|\tilde\mu_{s_1}\|_\fbm=\ldots=\|\tilde\mu_{s_K}\|_\fbm\geq\frac58\hdtwo(A)$.
The total urban planning cost can now be estimated as follows (see Remark \ref{rem:cellOpProps}),
\begin{align*}
\urbPlEn^{\varepsilon,a,A}[\chiOpt]
&\geq\inf_\chi\urbPlEn^{\varepsilon,a,\mu_0-\tilde\mu_{0},\mu_1-\tilde\mu_{1}}[\chi]
+\inf_\chi\urbPlEn^{\varepsilon,a,\tilde\mu_{0},\tilde\mu_{s_K}}[\chi]
+\sum_{k=1}^{K-1}\inf_\chi\urbPlEn^{\varepsilon,a,\tilde\mu_{s_{k+1}},\tilde\mu_{s_k}}[\chi]
+\inf_\chi\urbPlEn^{\varepsilon,a,\tilde\mu_{s_1},\tilde\mu_{1}}[\chi]\\
&\geq\Wdone(\mu_0-\tilde\mu_{0},\mu_1-\tilde\mu_{1})
+\Wdone(\tilde\mu_{0},\tilde\mu_{s_K})
+\sum_{k=1}^{K-1}\Wdone(\tilde\mu_{s_{k+1}},\tilde\mu_{s_k})
+\Wdone(\tilde\mu_{s_1},\tilde\mu_{1})
\end{align*}
so that the excess energy is bounded below by
\begin{align*}
\Delta \urbPlEn^{\varepsilon,a,A}
&=\urbPlEn^{\varepsilon,a,A}[\chiOpt]-\hdtwo(A)\\
&=\urbPlEn^{\varepsilon,a,A}[\chiOpt]-\|\mu_0-\tilde\mu_0\|_\fbm-\|\tilde\mu_0\|_\fbm\left[s_K+\sum_{k=1}^{K-1}(s_k-s_{k+1})+1-s_1\right]\\
&\geq\Wdone(\tilde\mu_{s_K},\tilde\mu_0)-\|\tilde\mu_{s_K}\|s_K+\sum_{k=1}^{K-1}\left(\Wdone(\tilde\mu_{s_k},\tilde\mu_{s_{k+1}})-\|\tilde\mu_{s_k}\|(s_k-s_{k+1})\right)\,,
\end{align*}
where in the last step we used $\Wdone(\mu_0-\tilde\mu_{0},\mu_1-\tilde\mu_{1})\geq\|\mu_0-\tilde\mu_0\|_\fbm$ as well as $\Wdone(\tilde\mu_{s_1},\tilde\mu_{1})\geq\|\tilde\mu_0\|_\fbm(1-s_1)$.
Now let $(\cdot)'$ denote the projection onto $\R^2$, that is,
\begin{equation*}
(\cdot)':\R^3\to\R^2,\,x'=(x_1,x_2)
\qquad\text{and}\qquad
\mu'=\pushforward{(\cdot)'}{\mu}
\end{equation*}
for $\mu\in\fbm(\R^3)$.
Let us define also
\begin{displaymath}
\tilde\Pi(\tilde\mu_s,\tilde\mu_t) = \left\{\mu \in \fbm(\R^3\times\R^3) \ : \ \begin{aligned}
                                                                                &\mu(A \times \R^3) = \tilde\mu_s(A),\, \mu(\R^3 \times B) = \tilde\mu_t(B) \ \text{for all}\ A,B \in \Bcal(\R^3)\\
                                                                                &\spt \mu \subseteq \{(x,y) \in \R^3 \times \R^3 \ : \ |x'-y'| \leq |x_3-y_3|\}
                                                                               \end{aligned}\right\}.
\end{displaymath}
Using the inequality
\begin{displaymath}
 \sqrt{\alpha^2+\beta^2}-\beta \geq \frac{\alpha^2}{4\beta} \qquad \text{for all} \ \frac{\alpha}{\beta} \in [0,\sqrt{8}]
\end{displaymath}
and the notation from Definition\,\ref{def:Wasserstein} we get
\begin{align*}
\Wdone(\tilde\mu_{s},\tilde\mu_{t})-\|\tilde\mu_{s}\|(s-t)
&=\inf_{\mu\in\tilde\Pi(\tilde\mu_{s},\tilde\mu_{t})} \int_{\R^3\times\R^3} |x - y|-(s-t) \,\de\mu(x,y)\\
&=\inf_{\mu\in\tilde\Pi(\tilde\mu_{s},\tilde\mu_{t})} \int_{\R^3\times\R^3}\sqrt{|x'-y'|^2+(s-t)^2}-(s-t) \,\de\mu(x,y)\\
&\geq\inf_{\mu\in\tilde\Pi(\tilde\mu_{s},\tilde\mu_{t})} \int_{\R^3\times\R^3}\frac{|x'-y'|^2}{4(s-t)} \,\de\mu(x,y)\\
&\geq\frac{\Wd2(\tilde\mu_{s}',\tilde\mu_{t}')^2}{4(s-t)}
\end{align*}
for any $s,t\in[0,1]$.
Thus, the previous estimate becomes
\begin{align*}
\Delta \urbPlEn^{\varepsilon,a,A}
&\geq\sum_{k=1}^{K-1}\frac{{\Wd2}(\tilde\mu_{s_k}',\tilde\mu_{s_{k+1}}')^2}{4(s_k-s_{k+1})}+\frac{{\Wd2}(\tilde\mu_{s_K}',\tilde\mu_0')^2}{4s_K}\\
&\geq\sum_{k=1}^{K-1}\frac{\left[{\Wd2}(\tilde\mu_{s_k}',\tilde\mu_0')-{\Wd2}(\tilde\mu_0',\tilde\mu_{s_{k+1}}')\right]^2}{4(s_k-s_{k+1})}+\frac{{\Wd2}(\tilde\mu_{s_K}',\tilde\mu_0')^2}{4s_K}\,,
\end{align*}
where we have used the triangle inequality.
Now Jensen's inequality, Lemma\,\ref{thm:WoneBound}, and Lemma\,\ref{thm:crossSection3D} imply
\begin{equation*}
{\Wd2}(\tilde\mu_{s_k}',\tilde\mu_0')^2
\geq\left(\frac{\Wdone(\tilde\mu_{s_k}',\tilde\mu_0')}{\sqrt{\|\tilde\mu_{s_k}'\|_\fbm}}\right)^2
\geq C^2\frac{\|\tilde\mu_{s_k}'\|_\fbm^2}{\omega_2N_{s_k}}
\geq\frac{25C^2K\hdtwo(A)^2\varepsilon s_k}{256\omega_2\Delta\urbPlEn^{\varepsilon,a,A}}=:(d_k^*)^2\,,
\end{equation*}
thus the bound on $\Delta\urbPlEn^{\varepsilon,a,A}$ can be expressed as
\begin{equation*}
\Delta\urbPlEn^{\varepsilon,a,A}
\geq\inf_{d_k\geq d_k^*,\,k=1,\ldots,K}\sum_{k=1}^{K-1}\frac{\left[d_k-d_{k+1}\right]^2}{4(s_k-s_{k+1})}+\frac{d_K^2}{4s_K}\,.
\end{equation*}
Using then first order optimality conditions one can check that this convex optimisation problem is minimised by the choice $d_k=d_k^*$, $k=1,\ldots,K$, so that
\begin{multline*}
\Delta \urbPlEn^{\varepsilon,a,A}
\geq\sum_{k=1}^{K-1}\frac{\left[d_k^*-d_{k+1}^*\right]^2}{4(s_k-s_{k+1})}+\frac{(d_K^*)^2}{4s_K}
\sim\frac{K\hdtwo(A)^2\varepsilon}{\Delta\urbPlEn^{\varepsilon,a,A}}\left[\sum_{k=1}^{K-1}\frac{(\sqrt{s_k}-\sqrt{s_{k+1}})^2}{s_k-s_{k+1}}+1\right]\\
=\frac{K\hdtwo(A)^2\varepsilon}{\Delta\urbPlEn^{\varepsilon,a,A}}\left[\sum_{k=1}^{K-1}\frac{1-\sqrt{\frac{s_{k+1}}{s_k}}}{1+\sqrt{\frac{s_{k+1}}{s_k}}}+1\right]
\gtrsim\frac{\hdtwo(A)^2\varepsilon}{\Delta \urbPlEn^{\varepsilon,a,A}}\left|\log\left(128\frac a{a-1}\frac{\Delta \urbPlEn^{\varepsilon,a,A}}{\hdtwo(A)}\right)\right|^2
\end{multline*}
using $2^{-K}\geq128\frac a{a-1}\frac{\Delta \urbPlEn^{\varepsilon,a,A}}{\hdtwo(A)}$.
Abbreviating $z=128\frac a{a-1}\frac{\Delta \urbPlEn^{\varepsilon,a,A}}{\hdtwo(A)}$, this inequality can be transformed into $\frac{z^2}{|\log z|^2}\gtrsim\big(\frac a{a-1}\big)^2\varepsilon$,
which can be solved for $z$ to yield $z\gtrsim\frac a{a-1}\sqrt\varepsilon\max\left\{1,|\log\frac{a\sqrt\varepsilon}{a-1}|\right\}$.
Thus,
\begin{equation*}
\Delta \urbPlEn^{\varepsilon,a,A}\gtrsim\hdtwo(A)\sqrt\varepsilon\max\left\{1,\left|\log\tfrac{a\sqrt\varepsilon}{a-1}\right|\right\}\,.
\end{equation*}
This bound was derived under the assumption that $\Delta \urbPlEn^{\varepsilon,a,A}\leq 2^{-13}\hdtwo(A)\min\{1,a-1\}$
so that together we obtain
\begin{equation*}
\Delta \urbPlEn^{\varepsilon,a,A}\gtrsim\hdtwo(A)\min\left\{\sqrt\varepsilon\left(1+\left|\log\tfrac{a-1}{\sqrt\varepsilon}\right|\right),a-1,1\right\}\,.
\end{equation*}
In addition we will derive the bound
$$\Delta \urbPlEn^{\varepsilon,a,A}\gtrsim\hdcodimone(A)\min\{\sqrt{a^2-1}^{\frac{n-2}{n-1}}\varepsilon^{\frac1{n-1}},a-1\}=\hdtwo(A)\min\{\sqrt\varepsilon\sqrt[4]{a^2-1},a-1\}$$
in Section\,\ref{sec:bdyContribution}.
Now assume, the minimum in the former bound were achieved by $1$, that is, $a-1>1$ as well as $\sqrt\varepsilon\left(1+\left|\log\tfrac{a-1}{\sqrt\varepsilon}\right|\right)>1$,
and the estimate reads $\Delta \urbPlEn^{\varepsilon,a,A}\gtrsim\hdtwo(A)$.
As long as $\sqrt\varepsilon\sqrt[4]{a^2-1}<1$, the latter bound does not overrule this estimate.
However, this can only happen if $\varepsilon\sim1$ and $a-1\sim1$ so that actually $\Delta \urbPlEn^{\varepsilon,a,A}\gtrsim\hdtwo(A)(a-1)$.
Thus, the $1$ in the former bound may be neglected, and we finally arrive at
\begin{align*}
\Delta \urbPlEn^{\varepsilon,a,A}
&\gtrsim\hdtwo(A)\max\left\{\min\left\{\sqrt\varepsilon\left(1+\left|\log\tfrac{a-1}{\sqrt\varepsilon}\right|\right\},a-1\right\},\min\left\{\sqrt\varepsilon\sqrt[4]{a^2-1},a-1\right\}\right\}\\
&\gtrsim\hdtwo(A)\min\left\{\sqrt\varepsilon\left(1+\left|\log\tfrac{a-1}{\sqrt\varepsilon}\right|+\sqrt[4]{a^2-1}\right),a-1\right\}\\
&\gtrsim\hdtwo(A)\min\left\{\sqrt\varepsilon\left(\sqrt a+\left|\log\tfrac{a-1}{\sqrt\varepsilon}\right|\right),a-1\right\}\,.
\end{align*}
\end{proof}

\subsubsection{Lower bound in $n$D for $n>3$}
In two dimensions we bounded the excess energy from below using the relaxed energy between $\mu_0$, $\mu_1$, and a generic cross-section $\{x_2=s\}$ with $s\sim\frac12$.
In three dimensions we employed the same technique, only now considering various cross-sections.
In higher dimensions, where the energy will be concentrated near the boundary, we apply the technique to a cross-section near the boundary.
The next lemma is a generalisation of Lemma\,\ref{thm:crossSection2D}.
\begin{lemma}
Let $T\in(0,1)$. There is a generic $s\in(0,T)$ such that
\begin{equation*}
N_s\leq\frac{\Delta \urbPlEn^{\varepsilon,a,A}}{\varepsilon T}
\qquad\text{and}\qquad
\Delta F_s\leq\frac{\Delta \urbPlEn^{\varepsilon,a,A}}{(a-1)T}\,.
\end{equation*}
\end{lemma}
\begin{proof}
If almost all $s\in(0,T)$ satisfied $N_s>\frac{\Delta \urbPlEn^{\varepsilon,a,A}}{\varepsilon T}$ or $\Delta F_s>\frac{\Delta \urbPlEn^{\varepsilon,a,A}}{(a-1)T}$,
then we had $\Delta \urbPlEn^{\varepsilon,a,A}\geq\int_0^T\varepsilon N_s+(a-1)\Delta F_s\,\de t>\Delta \urbPlEn^{\varepsilon,a,A}$, a contradiction.
\end{proof}

\begin{proposition}
For $\varepsilon<1$ and $a>1$ we have
\begin{displaymath}
 \Delta \urbPlEn^{\varepsilon,a,A}\gtrsim\hdtwo(A)\min\left\{\sqrt a\sqrt{a-1}^{\frac{n-3}{n-1}}\varepsilon^{\frac1{n-1}},a-1\right\}\,.
\end{displaymath}
\end{proposition}

\begin{proof}
The lower bound proof now is as follows (note how the procedure parallels that of the case $n=2$).
Assume $\Delta \urbPlEn^{\varepsilon,a,A}\leq\frac14\hdcodimone(A)(a-1)$ and pick $$T=\frac{2\Delta \urbPlEn^{\varepsilon,a,A}}{\hdcodimone(A)(a-1)}\,.$$
Let $s\in(0,T)$ be the cross-section from the previous lemma.
Let $\tilde\mu_0,\tilde\mu_1$ be the irrigating and irrigated measure of all particles in $\reSpace_s$ according to \eqref{eqn:measureLow} and \eqref{eqn:measureHigh}.
Now
\begin{align*}
\urbPlEn^{\varepsilon,a,A}[\chiOpt]
&\geq\inf_\chi\urbPlEn^{\varepsilon,a,\tilde\mu_0,\mu_s}[\chi]+\inf_\chi\urbPlEn^{\varepsilon,a,\mu_s,\tilde\mu_1}[\chi]+\Wdone(\mu_0-\tilde\mu_0,\mu_1-\tilde\mu_1)\\
&\geq \Wdone(\tilde\mu_0,\mu_s)+\Wdone(\mu_s,\tilde\mu_1)+\Wdone(\mu_0-\tilde\mu_0,\mu_1-\tilde\mu_1)\\
&\geq s\|\mu_s\|_\fbm+C\|\mu_s\|_\fbm\min\{\tfrac{R^2}{s},R\}+(1-s)\|\mu_s\|_\fbm+\|\mu_0-\tilde\mu_0\|_\fbm\\
&\geq \urbPlEn^{*,a,A}+CR\|\mu_s\|_\fbm\min\{\tfrac{R}{T},1\}\,,
\end{align*}
where we have employed Lemma\,\ref{thm:WoneBound} with $R=\sqrt[n-1]{\frac{\|\mu_s\|_\fbm}{N_s\omega_{n-1}}}$.
By virtue of the previous lemma, our assumption on $\Delta \urbPlEn^{\varepsilon,a,A}$, and our choice of $T$, $\|\mu_s\|_\fbm=\hdcodimone(A)-\Delta F_s\geq\frac{\hdcodimone(A)}2$ and $N_s\leq\frac{\hdcodimone(A)(a-1)}{2\varepsilon}$
so that $R\geq\sqrt[n-1]{\frac{\varepsilon}{\omega_{n-1}(a-1)}}$.
Inserting everything into the above estimate,
\begin{equation}\label{eq:above_inquality}
\Delta \urbPlEn^{\varepsilon,a,A}
\geq C\tfrac{\hdcodimone(A)}2\sqrt[n-1]{\tfrac{\varepsilon}{\omega_{n-1}(a-1)}}\min\left\{\tfrac{\hdcodimone(A)}{2\sqrt[n-1]{\omega_{n-1}}\Delta \urbPlEn^{\varepsilon,a,A}}\varepsilon^{\frac1{n-1}}(a-1)^{\frac{n-2}{n-1}},1\right\}\,.
\end{equation}
Now we consider two cases.
\begin{itemize}
\item If $\Delta \urbPlEn^{\varepsilon,a,A}\leq\tfrac{\hdcodimone(A)}{2\sqrt[n-1]{\omega_{n-1}}}\varepsilon^{\frac1{n-1}}(a-1)^{\frac{n-2}{n-1}}$, inequality \eqref{eq:above_inquality} yields
\begin{equation*}
\Delta \urbPlEn^{\varepsilon,a,A}
\gtrsim\hdcodimone(A)\sqrt[n-1]{\tfrac{\varepsilon}{a-1}}\,.
\end{equation*}
Note that this is only possible for $a-1\gtrsim1$.
\item If $\Delta \urbPlEn^{\varepsilon,a,A}>\tfrac{\hdcodimone(A)}{2\sqrt[n-1]{\omega_{n-1}}}\varepsilon^{\frac1{n-1}}(a-1)^{\frac{n-2}{n-1}}$, inequality \eqref{eq:above_inquality} can be solved for $\Delta \urbPlEn^{\varepsilon,a,A}$ to give
\begin{equation*}
\Delta \urbPlEn^{\varepsilon,a,A}
\gtrsim\hdcodimone(A)\varepsilon^{\frac1{n-1}}\sqrt{a-1}^{\frac{n-3}{n-1}}\,.
\end{equation*}
\end{itemize}
Both cases can be summarised as
\begin{equation*}
\Delta \urbPlEn^{\varepsilon,a,A}
\gtrsim\hdcodimone(A)\varepsilon^{\frac1{n-1}}
\cdot\begin{cases}
(a-1)^{-\frac1{n-1}}&\text{if }a-1>1,\\
\sqrt{a-1}^{\frac{n-3}{n-1}}&\text{else.}
\end{cases}
\end{equation*}
This bound was derived under the assumption that $\Delta \urbPlEn^{\varepsilon,a,A}\lesssim\hdcodimone(A)(a-1)$.
Together with the bound $\Delta \urbPlEn^{\varepsilon,a,A}\gtrsim\hdcodimone(A)\min\{\sqrt{a^2-1}^{\frac{n-2}{n-1}}\varepsilon^{\frac1{n-1}},a-1\}$ derived in Section\,\ref{sec:bdyContribution}, the estimates can be combined into
\begin{equation*}
\Delta \urbPlEn^{\varepsilon,a,A}\gtrsim\hdtwo(A)\min\left\{\sqrt a\sqrt{a-1}^{\frac{n-3}{n-1}}\varepsilon^{\frac1{n-1}},a-1\right\}\,,
\end{equation*}
as can be easily checked for the different cases $a-1\ll1$, $a-1\sim1$, and $a-1\gg1$.
\end{proof}

\subsubsection{Boundary contribution}\label{sec:bdyContribution}
Here we estimate the energy associated with the particles travelling from the boundary to the network $\SigmaOpt$.
The estimate holds for all $n\geq3$.
We need a few preparatory lemmas, the first of which is a refinement of \cite[Thm. 3.16]{Buttazzo-Oudet-Stepanov}.
\begin{lemma}\label{thm:avDist}
Let $S\subset\R^n$ be connected with $\hdone(S)=l$ and $K\subset\R^n$ with $\hdndim(K)=V$.
Then there is a positive constant $C\equiv C(n)$ such that
\begin{equation*}
r_{S,K}:=\frac1V\int_K\dist(x,S)\,\de\hdndim(x)\geq C\min\left\{V^{\frac1n},\left(\tfrac Vl\right)^{\frac1{n-1}}\right\}\,,
\end{equation*}
where in the case $V = 0$ the normalised integral is to be interpreted as $r_{S,K} = 0$.
\end{lemma}
\begin{proof}
Let us pick $\delta=\zeta\left(\tfrac Vl\right)^{\frac1{n-1}}$ for some constant $\zeta\equiv\zeta(n)\in(0,1)$ to be specified later.
Now we cover $K$ and the surrounding space by a grid of cubes with side length $\delta$.
In \cite[Lem. 3.17]{Buttazzo-Oudet-Stepanov} the authors prove that $S$ intersects at most $k=(2^n+1)\frac l\delta+(2^n+1)$ cubes
by noting that one needs a curve segment of at least length $\delta$ to intersect any union of $2^n+1$ cubes.
Those $k$ cubes together have a volume of
\begin{equation*}
V_S=k\delta^n=(2^n+1)\zeta^{n-1}V+(2^n+1)\zeta^n\left(\tfrac Vl\right)^{\frac n{n-1}}\,.
\end{equation*}
We shall call a cube a neighbour of $S$ if it is a neighbour of a cube intersecting $S$,
and we shall call a cube a second neighbour if it is adjacent to a neighbour of $S$.
We distinguish three cases:
\begin{itemize}
\item If $(2^n+1)\zeta^{n-1}V\leq(2^n+1)\zeta^n\left(\tfrac Vl\right)^{\frac n{n-1}}$
or equivalently $l\leq \zeta^{1-\frac1n}V^{\frac1n}$,
then the average distance $r_{S,K}$ of points in $K$ to $S$ is greater than $V^{\frac1n}$ up to a constant factor.
Indeed, it is clear that for $\xi\geq0$ we have $r_{\xi S,\xi K}=\xi r_{S,K}$.
Furthermore,
\begin{equation}\label{eq:r_1}
r_1:=\inf_{\substack{\tilde K\subset\R^n\,,\ \hdndim(\tilde K)=1\\\tilde S\subset\R^n\text{ connected},\ \hdone(\tilde S)\leq\zeta^{1-\frac1n}}}\int_{\tilde K}\dist(x,\tilde S)\,\de\hdndim(x)>0
\end{equation}
so that $r_{S,K}=V^{\frac1n}r_{V^{-\frac1n}S,V^{-\frac1n}K}\geq V^{\frac1n}r_1$. The positivity of $r_1$ will be proved in Lemma \ref{lem:positivity_of_r_1} below.

\item If $l>\zeta^{1-\frac1n}V^{\frac1n}$ and if the intersected cubes, the neighbours, and the second neighbours contain less than $\frac V2$ of $K$,
then at least $\frac V2$ of $K$ is at distance larger than $2\delta$ from $S$ so that $r_{S,K}\geq (\frac{V}{2}(2\delta)+\frac{V}{2}0)/V = \delta$.

\item Now let $l>\zeta^{1-\frac1n}V^{\frac1n}$ and let the intersected cubes and first and second neighbours contain more than $\frac V2$ of $K$.
      The total number of intersecting cubes and neighbours is less than $k 3^n$, where $3^n$ is the number of neighbours a cube has (including the cube itself).
      Thus, the second neighbours actually intersect $K$ on a volume of more than
      \begin{equation*}
       \tfrac V2-k3^n\delta^n
       =\tfrac V2-3^nV_S
       \geq\tfrac V2-3^n\cdot2(2^n+1)\zeta^{n-1}V
       \geq\tfrac V4
      \end{equation*}
      if we choose $\zeta$ small enough.
      Thus, $r_{S,K}\geq\frac\delta4$.
\end{itemize}
\end{proof}

\begin{remark}
Intuitively it is clear how in the previous lemma $S$ and $K$ have to be arranged to give the smallest $r_{S,K}$;
$S$ should be chosen as a straight line of length $l$ and $K=B_r(S)$, where $r$ is such that $\hdndim(K)=V$.
If $S$ is very short, then $K$ is almost a ball, and the average distance to $S$ behaves like the average distance to the ball centre, which scales like $V^{\frac1n}$.
If $S$ is very long, then $K$ is almost a cylinder, and the average distance to its midaxis scales like $(\frac Vl)^{\frac1{n-1}}$.
\end{remark}

\begin{lemma}\label{lem:positivity_of_r_1}
Let $r_1$ be given by equation \eqref{eq:r_1}, then $r_1 > 0$.
\end{lemma}

\begin{proof}
For a contradiction assume there exist sequences $\tilde K_i,\tilde S_i$ with $\hdndim(\tilde K_i) = 1, \hdone(\tilde S_i) \leq \zeta^{1-\frac{1}{n}}$ such that
\begin{displaymath}
 F(\tilde K_i, \tilde S_i) = \int_{\tilde K_i} \dist(x,\tilde S_i) \,\de \hdndim(x) \to 0
\end{displaymath}
as $i\to\infty$, where without loss of generality we may assume the $\tilde K_i,\tilde S_i$ to be uniformly bounded.
Thanks to Go\l ab's Theorem we may assume that, up to a subsequence, $\tilde S_i \to S$ in the Hausdorff sense and that $\hdone(S) \leq \zeta^{1-\frac{1}{n}}$. Moreover,
\begin{equation}\label{eq:inf_F(K,S)=0}
 \inf_{\tilde K, \hdndim(\tilde K) = 1} F(\tilde K, S) = 0,
\end{equation}
since
\begin{multline*}
 F(\tilde K_i,S) = \int_{\tilde K_i} \dist(x,S) \,\de\hdndim(x) \geq \int_{\tilde K_i} \dist(x,\tilde S_i) - \dist(S,\tilde S_i) \,\de \hdndim(x)\\
 = F(\tilde K_i, \tilde S_i) - \dist(S,\tilde S_i)\hdndim(\tilde K_i) \to 0\,.
\end{multline*}
This leads to a contradiction. In fact, let $K \subseteq \R^n$ with $\hdndim(K) = 1$ and set $K_\delta = \{x \in K \ : \dist(x,S) > \delta\}$ as well as $S_\delta = \{x \in \R^n \ : \ \dist(x,S) \leq \delta\}$. Then,
\begin{displaymath}
 F(K,S) \geq \delta\hdndim(K_\delta) = \delta(1-\hdndim(K\setminus K_\delta)) \geq \delta(1-\hdone(S_\delta))\,.
\end{displaymath}
Now \cite[Lemma 4.1]{Brancolini-Solimini-Fractal} implies $\hdndim(S_\delta) \leq C\hdone(S)\delta^{n-1}$ for some constant $C\equiv C(n)$
so that we can choose $\delta > 0$ such that $\hdndim(S_\delta) < \frac{1}{2}$.
Then $F(K,S) \geq \delta(1-\frac{1}{2}) > 0$ independent of $K$, which contradicts equation \eqref{eq:inf_F(K,S)=0}.
\end{proof}

The next lemma is a regularity result for the optimal irrigation pattern $\chiOpt$ and corresponding network $\SigmaOpt$.
\begin{lemma}\label{thm:connPath}
There are an optimal irrigation pattern $\chiOpt$ and corresponding network $\SigmaOpt$ such that $\SigmaOpt$ is symmetric across $\{x_n=\frac12\}$
and $\chiOpt_p(I)\cap\SigmaOpt$ is connected for almost all $p\in\reSpace$.
\end{lemma}
\begin{proof}
We may assume $\chiOpt$ to have the single path property (see \cite[Proposition 3.4.1]{BW15}).
Furthermore, due to the problem symmetry we may assume $\chiOpt_p(I)$ to be symmetric across $\{x_n=\frac12\}$ for all $p\in\reSpace$.
Indeed, suppose without loss of generality and potentially reparameterising the pattern that $\chiOpt_p(\frac12)\in\{x_n=\frac12\}$ for all $p\in\reSpace$.
We now consider a pattern with the property that the position of a particle $p$ at time $\frac{1}{2}+t$ is the same as at time $\frac{1}{2}-t$, but reflected across $\{x_n=\frac12\}$. As an explicit formula,
\begin{displaymath}
 \chi(p,t) = \begin{cases}
              \chiOpt(p,t) & t \in [0,\frac{1}{2}],\\
              \left((\chiOpt(p,1-t))_1,\ldots,(\chiOpt(p,1-t))_{n-1},1-(\chiOpt(p,1-t))_n\right) & t \in [\frac{1}{2},1].
             \end{cases}
\end{displaymath}
If $\chiOpt$ is optimal, then this new symmetric pattern must also be optimal with the same cost.
The corresponding $\SigmaOpt$ is also symmetric.

Now assume $\chiOpt_p(I)\cap\SigmaOpt$ to be not connected.
Then there are $s_1,s_2,\delta>0$ with $\chiOpt_p(t)\notin\SigmaOpt$ for $t\in(s_1,s_2)$ and $\chiOpt_p(t)\in\SigmaOpt$ for $t\in(s_1-\delta,s_1)\cup(s_2,s_2+\delta)$.
Let us abbreviate $\hat\Sigma=\chiOpt_p((s_1-\delta,s_1))\subset\SigmaOpt$.
Without loss of generality, assume $\hat\Sigma\subset\{x_n<\frac12\}$ and $\hdone(\hat\Sigma)>0$.

Let $[x]_{\chiOpt}\subset\reSpace$ with $\reMeasure([x]_{\chiOpt})=m_{\chiOpt}(x)$ denote the set of particles flowing through $x\in\hat\Sigma$.
Note $m_{\chiOpt}(x)>\frac\varepsilon{a-1}$ by the relation between $\chiOpt$ and $\SigmaOpt$ from Section\,\ref{sec:introUrbPlan}.
Due to the single path property, all particles $q\in[x]_{\chiOpt}$ follow the same path between $x$ and its reflection $\bar x=\chiOpt_p(s_3)$ across $\{x_n=\frac12\}$,
in other words, $\chiOpt_p((s_1,s_3))\subset\chiOpt_q(I)$ for all $q\in[x]_{\chiOpt}$.
Thus, $m_{\chiOpt}(x)>\tfrac\varepsilon{a-1}$ for all $x\in\chiOpt_p((s_1,s_3))$
so that $\chiOpt_p((s_1,s_3))\subset\SigmaOpt$.
This however contradicts $\chiOpt_p((s_1,s_2))\subset\R^n\setminus\SigmaOpt$.
\end{proof}

\begin{proposition}[Boundary contribution]
In $n\geq3$ dimensions the contribution to the excess energy of particles travelling from the initial distribution $\mu_0$ to the optimal network $\SigmaOpt$ can be estimated as
\begin{displaymath}
 \Delta\urbPlEn^{\varepsilon,a,A}
\gtrsim\hdcodimone(A)\min\left\{a-1,\sqrt{a^2-1}^{\frac{n-2}{n-1}}\varepsilon^{\frac1{n-1}}\right\}\,.
\end{displaymath}
\end{proposition}

\begin{proof}
By Lemma\,\ref{thm:connPath} we may assume that the path $\chiOpt_p(I)$ of any particle $p\in\reSpace$ can be subdivided into three connected segments,
one outside $\SigmaOpt$, one within, and again one outside.
Let
\begin{align*}
t(p)&=\inf\{t\in I\,:\,\chiOpt_p(t)\in\SigmaOpt\}\,,\\
x(p)&=\chiOpt_p(t(p))\,,\\
r(p)&=|(x(p)-\chiOpt_p(0))'|
\end{align*}
denote the time and position at which particle $p\in\reSpace$ reaches $\SigmaOpt$ for the first time, as well as the accumulated horizontal motion until then.
We have
\begin{multline*}
\int_Ir^{\varepsilon,a}_{\chiOpt}(\chiOpt(p,t))|\dot\chiOpt(p,t)|\,\de t
\geq\int_0^{t(p)}a|\dot\chiOpt(p,t)|\,\de t+\int_{t(p)}^1|\dot\chiOpt(p,t)|\,\de t\\
\geq a\sqrt{(r(p))^2+((x(p))_n)^2}+(1-(x(p))_n)
\geq1+\min\{\sqrt{a^2-1}r(p),a-1\}
\end{multline*}
upon optimising for $(x(p))_n\in[0,1]$, which would yield $(x(p))_n=\min\{\frac{r(p)}{\sqrt{a^2-1}},1\}$.
Let $\tilde\reSpace\subset\reSpace$ denote the particles for which $r(p)\leq\sqrt{\frac{a-1}{a+1}}$.
The excess energy can now be estimated below by
\begin{align*}
\Delta \urbPlEn^{\varepsilon,a,A}
&=\int_\reSpace\int_Ir^{\varepsilon,a}_{\chiOpt}(\chiOpt(p,t))|\dot\chiOpt(p,t)|\,\de t\,\de\reMeasure(p)-\reMeasure(\reSpace)\\
&\geq\int_{\tilde\reSpace}1+\sqrt{a^2-1}r(p)\,\de\reMeasure(p)
+\int_{\reSpace\setminus\tilde\reSpace}1+a-1\,\de\reMeasure(p)-\reMeasure(\reSpace)\\
&=\sqrt{a^2-1}r\reMeasure(\tilde\reSpace)+(a-1)\reMeasure(\reSpace\setminus\tilde\reSpace)\,,
\end{align*}
where $r=\frac1{\reMeasure(\tilde\reSpace)}\int_{\tilde\reSpace}r(p)\,\de\reMeasure(p)$ denotes the average horizontal motion of particles in $\tilde\reSpace$ before they reach $\SigmaOpt$.
Now combining Lemmas\,\ref{thm:crossSection2D} and \ref{thm:connPath} we know that $\SigmaOpt$ intersects $\{x_n=\frac12\}$ in at most $N=\frac{\Delta \urbPlEn^{\varepsilon,a,A}}\varepsilon$ points $x_1,\ldots,x_N$.
All particles in $\tilde\reSpace$ flow through such a point,
and $\tilde\reSpace\subset\bigcup_{i=1}^N[x_i]_{\chiOpt}$ using the notation from Definition\,\ref{def:solidarity_classes}.
Denote by $\Sigma_i\subset\SigmaOpt$ the connected component of $\SigmaOpt$ containing $x_i$ (by the single path property of the minimiser from Lemma\,\ref{thm:connPath}, the $\Sigma_i$ are disjoint).
If $\Sigma_i'$ denotes the orthogonal projection of $\Sigma_i$ onto $\{x_n=0\}$, we have
\begin{equation*}
r\geq\tfrac1{\reMeasure(\tilde\reSpace)}\sum_{i=1}^N\int_{[x_i]_{\chiOpt}\cap\tilde\reSpace}\dist(\chiOpt_p(0),\Sigma_i')\,\de \reMeasure(p)\,.
\end{equation*}
Using Lemma\,\ref{thm:avDist} in dimension $n-1$ and for $S=\Sigma_i'$ and $K=\{\chiOpt_p(0)\ :\ p\in[x_i]_{\chiOpt}\cap\tilde\reSpace\}$ we thus obtain
\begin{align*}
r&\geq\tfrac C{\reMeasure(\tilde\reSpace)}\sum_{i=1}^N\reMeasure([x_i]_{\chiOpt}\cap\tilde\reSpace)\min\left\{\reMeasure([x_i]_{\chiOpt}\cap\tilde\reSpace)^{\frac1{n-1}},\left(\tfrac{\reMeasure([x_i]_{\chiOpt}\cap\tilde\reSpace)}{\hdone(\Sigma_i)}\right)^{\frac1{n-2}}\right\}\\
&\geq\tfrac C{\reMeasure(\tilde\reSpace)}\inf_{\substack{V_1,\ldots,V_N\geq0\,,\ V_1+\ldots+V_N=\reMeasure(\tilde\reSpace)\\l_1,\ldots,l_N\geq0\,,\ l_1+\ldots+l_N\leq\hdone(\SigmaOpt)}}
\sum_{i=1}^NV_i\min\left\{V_i^{\frac1{n-1}},\left(\tfrac{V_i}{l_i}\right)^{\frac1{n-2}}\right\}\,.
\end{align*}
Without loss of generality assume $V_i^{\frac1{n-1}}\leq\left(\tfrac{V_i}{l_i}\right)^{\frac1{n-2}}$ for $i=1,\ldots,M$, then
\begin{align*}
r&\geq\tfrac C{\reMeasure(\tilde\reSpace)}\inf_{\substack{V_1,\ldots,V_N\geq0\,,\ V_1+\ldots+V_N=\reMeasure(\tilde\reSpace)\\l_1,\ldots,l_N\geq0\,,\ l_1+\ldots+l_N\leq\hdone(\SigmaOpt)}}
\sum_{i=1}^M\left(V_i^{\frac n{n-1}}\right)+\sum_{i=M+1}^NV_i\left(\tfrac{V_i}{l_i}\right)^{\frac1{n-2}}\\
&=\tfrac C{\reMeasure(\tilde\reSpace)}\inf_{\tilde V,\hat V\geq0\,,\ M\tilde V+(N-M)\hat V=\reMeasure(\tilde\reSpace)}
M\tilde V^{\frac n{n-1}}+[(N-M)\hat V]^{\frac{n-1}{n-2}}\hdone(\SigmaOpt)^{\frac{-1}{n-2}}\,,
\end{align*}
where we have used that the optimisation is convex in the $V_i$ and $l_i$ so that $V_1=\ldots=V_M=\tilde V$, $V_{M+1}=\ldots=V_N=\hat V$, and $l_{M+1}=\ldots=l_N=\tfrac{\hdone(\SigmaOpt)}{N-M}$ are optimal.
Substituting $W=M\tilde V$ and replacing $(N-M)\hat V$ by $\reMeasure(\tilde\reSpace)-W$, the above estimate turns into
\begin{align*}
r&\geq\tfrac C{\reMeasure(\tilde\reSpace)}\inf_{0\leq W\leq \reMeasure(\tilde\reSpace)\,,\ \tilde V\geq\frac WN}
W\tilde V^{\frac1{n-1}}+[\reMeasure(\tilde\reSpace)-W]^{\frac{n-1}{n-2}}\hdone(\SigmaOpt)^{\frac{-1}{n-2}}\\
&\geq\tfrac C{\reMeasure(\tilde\reSpace)}\inf_{0\leq W\leq \reMeasure(\tilde\reSpace)}
W^{\frac n{n-1}}N^{\frac{-1}{n-1}}+[\reMeasure(\tilde\reSpace)-W]^{\frac{n-1}{n-2}}\hdone(\SigmaOpt)^{\frac{-1}{n-2}}\\
&\geq\tfrac{C2^{\frac{1-n}{n-2}}}{\reMeasure(\tilde\reSpace)}\min\left\{\reMeasure(\tilde\reSpace)^{\frac n{n-1}}N^{\frac{-1}{n-1}},\reMeasure(\tilde\reSpace)^{\frac{n-1}{n-2}}\hdone(\SigmaOpt)^{\frac{-1}{n-2}}\right\}\,,
\end{align*}
where the last step follows from distinguishing the two cases $W>\frac{\reMeasure(\tilde\reSpace)}2$ and $W\leq\frac{\reMeasure(\tilde\reSpace)}2$.
Inserting this estimate into the lower bound on $\Delta \urbPlEn^{\varepsilon,a,A}$ and using the fact $N\leq\frac{\Delta \urbPlEn^{\varepsilon,a,A}}\varepsilon$ and $\hdone(\SigmaOpt)\leq\frac{\Delta \urbPlEn^{\varepsilon,a,A}}\varepsilon$, we obtain
\begin{align*}
\Delta \urbPlEn^{\varepsilon,a,A}
&\geq C2^{\frac{1-n}{n-2}}\sqrt{a^2-1}\min\left\{\reMeasure(\tilde\reSpace)^{\frac n{n-1}}N^{\frac{-1}{n-1}},\reMeasure(\tilde\reSpace)^{\frac{n-1}{n-2}}\hdone(\SigmaOpt)^{\frac{-1}{n-2}}\right\}+(a-1)\reMeasure(\reSpace\setminus\tilde\reSpace)\\
&\geq C2^{\frac{1-n}{n-2}}\sqrt{a^2-1}\reMeasure(\tilde\reSpace)
\min\left\{\left(\varepsilon\tfrac{\reMeasure(\tilde\reSpace)}{\Delta \urbPlEn^{\varepsilon,a,A}}\right)^{\frac1{n-1}},\left(\varepsilon\tfrac{\reMeasure(\tilde\reSpace)}{\Delta \urbPlEn^{\varepsilon,a,A}}\right)^{\frac1{n-2}}\right\}
+(a-1)\reMeasure(\reSpace\setminus\tilde\reSpace)\,.
\end{align*}
Now we distinguish three cases.
\begin{itemize}
\item If $\reMeasure(\tilde\reSpace)<\frac12\hdcodimone(A)$, we have $$\Delta \urbPlEn^{\varepsilon,a,A}\geq(a-1)\reMeasure(\reSpace\setminus\tilde\reSpace)\geq\tfrac12\hdcodimone(A)(a-1)\,.$$
\item If $\reMeasure(\tilde\reSpace)\geq\frac12\hdcodimone(A)$ and $\varepsilon\tfrac{\reMeasure(\tilde\reSpace)}{\Delta \urbPlEn^{\varepsilon,a,A}}<1$, we have
$$\Delta \urbPlEn^{\varepsilon,a,A}\geq C2^{\frac{1-n}{n-2}}\sqrt{a^2-1}\tfrac{\hdcodimone(A)}2
\left(\varepsilon\tfrac{\hdcodimone(A)}{2\Delta \urbPlEn^{\varepsilon,a,A}}\right)^{\frac1{n-2}}\,,$$
which can be solved to yield $\Delta \urbPlEn^{\varepsilon,a,A}\gtrsim\hdcodimone(A)\sqrt{a^2-1}^{\frac{n-2}{n-1}}\varepsilon^{\frac1{n-1}}$.
\item If $\reMeasure(\tilde\reSpace)\geq\frac12\hdcodimone(A)$ and $\varepsilon\tfrac{\reMeasure(\tilde\reSpace)}{\Delta \urbPlEn^{\varepsilon,a,A}}\geq1$, we have
$$\Delta \urbPlEn^{\varepsilon,a,A}\geq C2^{\frac{1-n}{n-2}}\sqrt{a^2-1}\tfrac{\hdcodimone(A)}2
\left(\varepsilon\tfrac{\hdcodimone(A)}{2\Delta \urbPlEn^{\varepsilon,a,A}}\right)^{\frac1{n-1}}\,,$$
which can be solved to yield $$\Delta \urbPlEn^{\varepsilon,a,A}\gtrsim\hdcodimone(A)\sqrt{a^2-1}^{\frac{n-1}{n}}\varepsilon^{\frac1{n}}
\geq\hdcodimone(A)\min\left\{a-1,\sqrt{a^2-1}^{\frac{n-2}{n-1}}\varepsilon^{\frac1{n-1}}\right\}\,.$$
\end{itemize}
In summary,
$$\Delta\urbPlEn^{\varepsilon,a,A}
\gtrsim\hdcodimone(A)\min\left\{a-1,\sqrt{a^2-1}^{\frac{n-2}{n-1}}\varepsilon^{\frac1{n-1}}\right\}\,.$$
\end{proof}

\subsection{Lower bound for $\brTptEn^{\varepsilon,A}$}\label{sec:lwBndBrnchTrpt}
To prove the lower bound of Theorem\,\ref{thm:scalingBrnchTrpt},
it is actually easier to use a different non-dimensionalisation of the energy $\brTptEn^{\varepsilon,\m,A,L}$ introduced in Section\,\ref{sec:energyScaling}.
In fact, we shall prove a lower bound using the energy
\begin{equation*}
\tilde\brTptEn^{\varepsilon,A}\equiv \brTptEn^{\varepsilon,\m,A,1}\quad\text{with }\m=\min\left\{1,\tfrac1{\hdcodimone(A)}\right\}\,,\qquad
\tilde\brTptEn^{*,A}\equiv \inf_{\chi}\tilde\brTptEn^{0,A}[\chi]=\mathfrak m\,,
\end{equation*}
where $\mathfrak m$ denotes the total transported mass,
\begin{equation*}
\mathfrak m=\|\mu_0\|_\fbm=\m\hdcodimone(A)=\min\{1,\hdcodimone(A)\}\,.
\end{equation*}
Since optimal irrigation patterns $\chiOpt$ are known to be simple
(that is, the patterns contain no loops, see \cite[Proposition 4.6]{BeCaMo09}) and $\mathfrak m\leq1$,
the maximum possible mass flux $m_{\chiOpt}(x)$ through any point $x$ is no larger than one.
This will allow us to bound the transport costs below by the costs for $\varepsilon=0$
since $m_{\chiOpt}(x)\leq1$ implies $s_{1-\varepsilon}^{\chiOpt}(x)=m_{\chiOpt}(x)^{-\varepsilon}\geq1$ (see Definition\,\ref{eqn:costDensityBrTpt}).
Indeed, let $\chiOpt$ be an optimal irrigation pattern between a given source $\mu_+$ and sink $\mu_-$ of equal mass no larger than one.
For almost all $p\in\reSpace$ we have
\begin{equation*}
\int_{I} s_{1-\varepsilon}^{\chiOpt}(\chiOpt(p,t)) |\dot\chiOpt(p,t)|\,\de t
=\int_{\chiOpt_p(I)} s_{1-\varepsilon}^{\chiOpt}(x)\,\de\hdone(x)
\geq\int_{\chiOpt_p(I)}\,\de\hdone(x)
\geq|\chiOpt_p(1)-\chiOpt_p(0)|\,.
\end{equation*}
Thus,
\begin{multline*}
\brTptEn^{\varepsilon,\mu_+,\mu_-}[\chiOpt]
= \int_\reSpace\int_{I} s_{1-\varepsilon}^{\chiOpt}(\chiOpt(p,t)) |\dot\chiOpt(p,t)|\,\de t\,\de\reMeasure(p)\\
\geq\int_\reSpace|\chiOpt_p(1)-\chiOpt_p(0)|\,\de\reMeasure(p)
=\int_{\R^n\times\R^n}|x-y|\,\de\mu(x,y)
\geq\Wdone(\mu_+,\mu_-)
\end{multline*}
for $\mu=\pushforward{(\chiOpt(\cdot,0),\chiOpt(\cdot,1))}{\reMeasure}$,
whose marginals are the irrigating and irrigated measures $\mu_+$ and $\mu_-$.

Based on this relaxation result we will show the following.

\begin{theorem}\label{thm:scalingBrnchTrpt2}
There is a constant $\tilde C_1\equiv\tilde C_1(n)$such that for $\varepsilon<\tilde C_1$
\begin{equation*}
\min_\chi\tilde\brTptEn^{\varepsilon,A}[\chi]-\tilde\brTptEn^{*,A}\geq\tilde C_1\mathfrak m\varepsilon\max\{|\log\varepsilon|,\log\hdcodimone(A)\}\,.
\end{equation*}
\end{theorem}

If $\hdcodimone(A)\leq1$, then $\tilde\brTptEn^{\varepsilon,A}=\brTptEn^{\varepsilon,A}$ so that the lower bound in Theorem\,\ref{thm:scalingBrnchTrpt} reduces to Theorem\,\ref{thm:scalingBrnchTrpt2}.
If $\hdcodimone(A)>1$, then the relation
\begin{equation*}
\brTptEn^{\varepsilon,A}[\chi]=\hdcodimone(A)^{1-\varepsilon}\tilde\brTptEn^{\varepsilon,A}[\chi]\,,\qquad
\brTptEn^{*,A}=\hdcodimone(A)\tilde\brTptEn^{*,A}\,,
\end{equation*}
immediately implies
\begin{equation*}
\min_\chi \brTptEn^{\varepsilon,A}[\chi]-\brTptEn^{*,A}\geq
\hdcodimone(A)\left(\frac{1+\tilde C_1\varepsilon\max\left\{|\log\varepsilon|,\log\hdcodimone(A)\right\}}{\hdcodimone(A)^\varepsilon}-1\right)\,,
\end{equation*}
from which the lower bound in Theorem\,\ref{thm:scalingBrnchTrpt} follows
via $\hdcodimone(A)^{-\varepsilon}\geq1-\varepsilon\log\hdcodimone(A)$.

As usual, we would like to start by characterising a generic cross-section
in terms of how many network pipes it intersects.
However, unlike for the urban planning model,
the maximum number of intersecting pipes cannot be deduced directly from the energy
since the excess cost per pipe depends on the mass flux through it.
Thus, we will for each flux separately treat the number of pipes with that flux.
Analogously to before let $\chiOpt$ denote an optimal irrigation pattern of $\tilde\brTptEn^{\varepsilon,A}$ with corresponding optimal network $\SigmaOpt$,
abbreviate $$\Delta\tilde\brTptEn^{\varepsilon,A}=\tilde\brTptEn^{\varepsilon,A}[\chiOpt]-\tilde\brTptEn^{*,A}\,,$$ and recall that
\begin{align*}
t_p(s)&=\min\left\{t\in I\,:\,\chiOpt_p(t)\in\{x_n=s\}\right\}\,,\\
\xiOpt(p,s)&=\chiOpt_p(t_p(s))
\end{align*}
for $s\in(0,1)$, $p\in\reSpace$ denote the time and position of particle $p$ in cross-section $\{x_n=s\}$.
For an easier notation, let us also introduce the corresponding solidarity classes and mass fluxes through $x\in\R^n$,
\begin{align*}
[x]_{\xiOpt} &= \{q \in \reSpace \ : \ x \in \xiOpt_q(I)\}\,,\\
m_{\xiOpt}(x) &= \reMeasure([x]_{\xiOpt})\,.
\end{align*}
For each cross-section $\{x_n=s\}$ we now introduce the number of intersecting pipes with mass flux $c$
as well as a corresponding measure (discrete and with integer multiplicity) on the interval of possible mass fluxes,
\begin{equation*}
\tilde N_s(c)=\hd^0\left(\left\{x\in\{x_n=s\}\,:\,m_{\xiOpt}(x)=c\right\}\right)\,,\qquad
N_s=\tilde N_s\hd^0\restr(0,\mathfrak m]\,.
\end{equation*}
It is known that $\SigmaOpt=\{x\in\R^n\,:\,m_{\chiOpt}(x)>0\}$ is rectifiable and represents a finite graph away from $x_n=0$ and $x_n=1$ (see \cite[Proposition 4.6]{BeCaMo09} or \cite[Proposition 6.6]{Bernot-Caselles-Morel-Traffic-Plans}).
Thus, for almost all $s\in(0,1)$, $\SigmaOpt\cap\{x_n=s\}$ is finite.
Let us abbreviate
\begin{equation*}
I_0=\{s\in(0,1)\,:\,\SigmaOpt\cap\{x_n=s\}\text{ is finite}\}\,.
\end{equation*}
Clearly, for almost all $s\in I_0$ the total flux through $\{x_n=s\}$ is given by $\int_{\{x_n=s\}}m_{\xiOpt}(x)\,\de\hd^0(x)=\reMeasure(\reSpace)=\|\mu_0\|_\fbm=\mathfrak m$, thus
\begin{equation*}
\int_0^{\mathfrak m}c\,\de N_s(c)=\mathfrak m\,.
\end{equation*}
The following lemma is the desired characterisation of a generic cross-section.

\begin{lemma}
There is a generic $s\in(0,1)$ such that $N_s(c)$ satisfies
\begin{equation*}
\int_0^{\mathfrak m}-c\log c\,\de N_s(c)\leq\frac{\Delta\tilde\brTptEn^{\varepsilon,A}}\varepsilon\,.
\end{equation*}
\end{lemma}
\begin{proof}
As mentioned above, $\SigmaOpt=\{x\in\R^n\,:\,m_{\chiOpt}(x)>0\}$ is rectifiable and represents a finite graph away from $x_n=0$ and $x_n=1$ so that we can calculate
\begin{multline*}
\tilde\brTptEn^{\varepsilon,A}[\chiOpt]
=\int_\reSpace\int_Is_{1-\varepsilon}^{\chiOpt}(\chiOpt(p,t))|\dot\chiOpt(p,t)|\,\de t\,\de\reMeasure
=\int_\reSpace\int_{\chiOpt_p(I)\cap\SigmaOpt}m_{\chiOpt}(x)^{-\varepsilon}\,\de\hdone(x)\,\de\reMeasure\\
=\int_{\SigmaOpt}\int_\reSpace m_{\chiOpt}(x)^{-\varepsilon}\setchar{\chiOpt_p(I)}(x)\,\de\reMeasure\,\de\hdone(x)
=\int_{\SigmaOpt}m_{\chiOpt}(x)^{1-\varepsilon}\,\de\hdone(x)
\geq\int_{\SigmaOpt}m_{\xiOpt}(x)^{1-\varepsilon}\,\de\hdone(x)\\
\geq\int_{I_0}\int_{\SigmaOpt\cap\{x_n=s\}}m_{\xiOpt}(x)^{1-\varepsilon}\,\de\hd^0(x)\,\de s
=\int_{I_0}\int_0^{\mathfrak m} c^{1-\varepsilon}\,\de N_s(c)\,\de s\,.
\end{multline*}
Assume by contradiction that $\int_0^{\mathfrak m}-c\log c\,\de N_s(c)>\frac{\Delta\tilde\brTptEn^{\varepsilon,A}}\varepsilon$ for almost every $s\in(0,1)$. Then
\begin{multline*}
\Delta\tilde\brTptEn^{\varepsilon,A}
=\tilde\brTptEn^{\varepsilon,A}[\chiOpt]-\mathfrak m
\geq\int_{I_0}\int_0^{\mathfrak m}c^{1-\varepsilon}\,\de N_s(c)\,\de s-\mathfrak m
=\int_{I_0}\int_0^{\mathfrak m}(c^{1-\varepsilon}-c)\,\de N_s(c)\,\de s\\
=\int_{I_0}\int_0^{\mathfrak m}c(c^{-\varepsilon}-1)\,\de N_s(c)\,\de s
\geq\int_{I_0}\int_0^{\mathfrak m}-c\varepsilon\log c\,\de N_s(c)\,\de s
>\Delta\tilde\brTptEn^{\varepsilon,A}\,,
\end{multline*}
where we linearised in $\varepsilon$ in the second last step.
This yields a contradiction.
\end{proof}

To show Theorem\,\ref{thm:scalingBrnchTrpt2}, as earlier we bound the energy below
by the transport costs for $\varepsilon=0$, using the additional information from the previous lemma.

\begin{proof}[Proof of Theorem\,\ref{thm:scalingBrnchTrpt2}.]
Let $\mu_s=\pushforward{\xiOpt(\cdot,s)}{\reMeasure}$ for the generic cross-section $s$ from the previous lemma.
By Remark \ref{rem:cellOpProps} we have
\begin{equation*}
\Delta\tilde\brTptEn^{\varepsilon,A}
=\tilde\brTptEn^{\varepsilon,A}[\chiOpt]-\tilde\brTptEn^{*,A}
=\inf_\chi\brTptEn^{\varepsilon,A,\mu_0,\mu_s}[\chi]+\inf_\chi\brTptEn^{\varepsilon,A,\mu_s,\mu_1}[\chi]-\mathfrak m
\,.
\end{equation*}
Now $\mu_s$ is a linear combination of delta-distributions,
\begin{equation*}
\mu_s=\sum_{\substack{c\in(0,\mathfrak m]\\\tilde N_s(c)>0}}\sum_{\substack{x\in\SigmaOpt\cap\{x_n=s\}\\m_{\xiOpt}(x)=c}}c\delta_x =: \sum_{\substack{c\in(0,\mathfrak m]\\\tilde N_s(c)>0}} \nu_c\,.
\end{equation*}
Let $\reSpace_c$ be the particles flowing through $\nu_c$, that is, $\reSpace_c$ is defined by the relation $\nu_c=\pushforward{\xiOpt(\cdot,s)}{\reMeasure\restr\reSpace_c}$,
and let $\mu_{0,c}$ and $\mu_{1,c}$ be initial and final distribution of these particles,
\begin{equation*}
\mu_{0,c} = \pushforward{\xiOpt(\cdot,0)}{\reMeasure\restr\reSpace_c}\,,\qquad
\mu_{1,c} = \pushforward{\xiOpt(\cdot,1)}{\reMeasure\restr\reSpace_c}\,.
\end{equation*}
Lemma\,\ref{thm:WoneBound} together with $\sum_{c\in(0,\mathfrak m],\,\tilde N_s(c)>0}c\tilde N_s(c)=\mathfrak m$ then implies
\begin{align*}
\inf_\chi\brTptEn^{\varepsilon,A,\mu_0,\mu_s}[\chi]+\inf_\chi\brTptEn^{\varepsilon,A,\mu_s,\mu_1}[\chi]
&\geq \sum_{\substack{c\in(0,\mathfrak m]\\\tilde N_s(c)>0}} \Wdone(\mu_{0,c},\nu_c)+\Wdone(\nu_c,\mu_{1,c})\\
&\geq\sum_{\substack{c\in(0,\mathfrak m]\\\tilde N_s(c)>0}}
c\tilde N_s(c)\left[s+C\min\left\{\big(\tfrac c{\m\omega_{n-1}}\big)^{\frac2{n-1}}/s,\big(\tfrac c{\m\omega_{n-1}}\big)^{\frac1{n-1}}\right\}\right]\\
&\quad+\!\!\!\!\sum_{\substack{c\in(0,\mathfrak m]\\\tilde N_s(c)>0}}\!\!\!\!
c\tilde N_s(c)\left[1-s+C\min\left\{\big(\tfrac c{\m\omega_{n-1}}\big)^{\frac2{n-1}}/(1-s),\big(\tfrac c{\m\omega_{n-1}}\big)^{\frac1{n-1}}\right\}\right]\\
&\geq\mathfrak m+\int_0^{\mathfrak m}c\tilde C\min\left\{\big(\tfrac c{\m}\big)^{\frac2{n-1}},\big(\tfrac c{\m}\big)^{\frac1{n-1}}\right\}\,\de N_s(c)
\end{align*}
for some constant $\tilde C\equiv \tilde C(n)>0$.
Consequently, exploiting $\m,\mathfrak m\leq1$,
\begin{multline*}
\Delta\tilde\brTptEn^{\varepsilon,A}
\geq \tilde C\int_0^{\mathfrak m}c\min\left\{\big(\tfrac c{\m}\big)^{\frac2{n-1}},\big(\tfrac c{\m}\big)^{\frac1{n-1}}\right\}\,\de N_s(c)\\
\geq \frac{\tilde C}{\m^{\frac1{n-1}}}\int_0^{\mathfrak m}c\min\left\{c^{\frac2{n-1}},c^{\frac1{n-1}}\right\}\,\de N_s(c)
\geq\frac{\tilde C}{\m^{\frac1{n-1}}}\int_0^{\mathfrak m}c^{\frac{n+1}{n-1}}\,\de N_s(c)\,.
\end{multline*}
Together with the previous lemma and $\int_0^{\mathfrak m}c\,\de N_s(c)=\mathfrak m$ this implies
\begin{equation*}
\Delta\tilde\brTptEn^{\varepsilon,A}
\geq\frac{\tilde C}{\m^{\frac1{n-1}}}\inf_{\substack{
N\in\fbm((0,\mathfrak m])\\
\int_0^{\mathfrak m}c\,\de N(c)=\mathfrak m\\
\int_0^{\mathfrak m}-c\log c\,\de N(c)\leq\frac{\Delta\tilde\brTptEn^{\varepsilon,A}}\varepsilon}}
\int_0^{\mathfrak m}c^{\frac{n+1}{n-1}}\,\de N(c)\,,
\end{equation*}
where the right-hand side represents a linear program.
The next lemma now implies
\begin{equation*}
\Delta\tilde\brTptEn^{\varepsilon,A}\geq\frac{\tilde C}{\m^{\frac1{n-1}}}\mathfrak m\exp\left(-\frac{2\Delta\tilde\brTptEn^{\varepsilon,A}}{(n-1)\mathfrak m\varepsilon}\right)
\end{equation*}
and thus $-\log\frac{\Delta\tilde\brTptEn^{\varepsilon,A}}{\tilde C\mathfrak m}-\frac{\log\m}{n-1}\leq\frac{\tilde C}{n-1}\frac1\varepsilon\frac{2\Delta\tilde\brTptEn^{\varepsilon,A}}{\tilde C\mathfrak m}$.
From the upper bound we already know $\Delta\tilde\brTptEn^{\varepsilon,A}\leq C_2\mathfrak m\varepsilon|\log\varepsilon|$ for a constant $C_2\equiv C_2(n)>0$,
hence $\log\frac{\Delta\tilde\brTptEn^{\varepsilon,A}}{\tilde C\mathfrak m}\leq0$ for $\varepsilon$ small enough such that $\varepsilon|\log\varepsilon|\leq\frac{\tilde C}{C_2}$.
Thus, the above inequality implies $\Delta\tilde\brTptEn^{\varepsilon,A}\geq\frac12\mathfrak m\varepsilon|\log\m|$
as well as $\left|\log\frac{\Delta\tilde\brTptEn^{\varepsilon,A}}{\tilde C\mathfrak m}\right|\leq\frac{\tilde C}{n-1}\frac1\varepsilon\frac{2\Delta\tilde\brTptEn^{\varepsilon,A}}{\tilde C\mathfrak m}$
from which we infer $\frac{\Delta\tilde\brTptEn^{\varepsilon,A}}{\mathfrak m}\geq\tilde C_1\varepsilon|\log\varepsilon|$ for a constant $\tilde C_1\equiv\tilde C_1(n)$.
\end{proof}

Obviously, the proof of Theorem\,\ref{thm:scalingBrnchTrpt2} reduces the energy estimate to a convex optimisation problem.
The role of the following lemma is simply to provide the necessary lower bound on that problem via convex duality.

\begin{lemma}\label{thm:convProg}
For any $\mathfrak m,D>0$ we have
\begin{equation*}
\inf_{\substack{
N\in\fbm([0,\mathfrak m])\\
\int_0^{\mathfrak m}c\,\de N(c)=\mathfrak m\\
\int_0^{\mathfrak m}-c\log c\,\de N(c)\leq D}}
\int_0^{\mathfrak m}c^{\frac{n+1}{n-1}}\,\de N(c)
\geq\mathfrak m\exp\left(-\frac{2D}{(n-1)\mathfrak m}\right)\,.
\end{equation*}
\end{lemma}
\begin{proof}
The optimisation problem is equivalent to finding
\begin{equation*}
\rho:=\inf_{N\in\rca((0,\mathfrak m])}G^*(N)+H^*(N)
\end{equation*}
for the Legendre--Fenchel duals
\begin{align*}
G^*:\rca([0,\mathfrak m])\to\R,\;&G^*(N)=\int_0^{\mathfrak m}c^{\frac{n+1}{n-1}}\,\de N+I_{\{N\geq0\}}(N)\\
H^*:\rca([0,\mathfrak m])\to\R,\;&H^*(N)=I_{\{\int_0^{\mathfrak m}c\,\de N=\mathfrak m\}}(N)+I_{\{\int_0^{\mathfrak m}-c\log c\,\de N\leq D\}}(N)
\end{align*}
of
\begin{align*}
G:\cont([0,\mathfrak m])\to\R,\;G(\phi)&=I_{\{\phi\leq0\}}\left(c\mapsto\phi(c)-c^{\frac{n+1}{n-1}}\right)\\
H:\cont([0,\mathfrak m])\to\R,\;H(\phi)&=\begin{cases}\lambda \mathfrak m-\kappa D&\text{if }\phi(c)=\lambda c+\kappa c\log c\text{ for }\lambda\in\R,\kappa\leq0,\\
\infty&\text{else,}\end{cases}
\end{align*}
where $\cont$ denotes the set of continuous functions, $I_S$ denotes the indicator function of a set $S$, that is, $I_S(x)=0$ if $x\in S$ and $I_S(x)=\infty$ else.
By Fenchel--Rockafellar duality,
\begin{equation*}
\rho\geq\sup_{\phi\in\cont([0,\mathfrak m])}-G(-\phi)-H(\phi)
=\sup_{\substack{\lambda\in\R,\kappa\leq0,\\\lambda c+\kappa c\log c\geq-c^{\frac{n+1}{n-1}}}}-\lambda \mathfrak m+\kappa D
=\sup_{\substack{\lambda\in\R,\kappa\leq0,\\\lambda-\kappa\log c\leq c^{\frac2{n-1}}}}\lambda \mathfrak m+\kappa D\,.
\end{equation*}
Any admissible choice of $\lambda,\kappa$ yields a lower bound.
We choose $\lambda,\kappa$ such that $\lambda-\kappa\log c$ actually touches the curve $c^{\frac2{n-1}}$ in a point $c^*$ tangentially,
that is, the function values and derivatives of both curves coincide,
\begin{equation*}
\lambda-\kappa\log c^*=(c^*)^{\frac2{n-1}}\,,\qquad
-\tfrac\kappa{c^*}=\tfrac2{n-1}(c^*)^{\frac{3-n}{n-1}}\,.
\end{equation*}
In particular, $\kappa=-\frac2{n-1}(c^*)^{\frac2{n-1}}$ and $\lambda=(c^*)^{\frac2{n-1}}(1-\frac2{n-1}\log c^*)$.
A maximisation of $\lambda \mathfrak m+\kappa D$ for $c^*$ now yields $c^*=\exp(-\frac D{\mathfrak m})$ and thus
\begin{equation*}
\lambda=\exp\left(-\tfrac{2D}{\mathfrak m(n-1)}\right)\left(\tfrac{2D}{\mathfrak m(n-1)}+1\right)\,,\qquad\kappa=-\tfrac2{n-1}\exp\left(-\tfrac{2D}{\mathfrak m(n-1)}\right)\,,
\end{equation*}
which produces
\begin{equation*}
\rho\geq\lambda \mathfrak m+\kappa D=\mathfrak m\exp(-\tfrac{2D}{\mathfrak m(n-1)})\,.
\end{equation*}
\end{proof}

\begin{remark}
The test function $N(c)=\frac{\mathfrak m}{c^*}\delta(c-c^*)$ for a properly chosen $c^*$ actually shows that the inequality in Lemma\,\ref{thm:convProg} is an equality.
\end{remark}

\section{Discussion}\label{sec:discussion}
In the following we briefly discuss different geometric settings as well as the relation of network optimisation to other pattern formation problems.

\subsection{Geometry variations}
It is out of the scope of this paper to rigorously prove the scaling laws also for variations of our geometric setting,
however, we would like to provide a little discussion of the dependence on geometry.
First of all notice that all lower bound proofs are still valid if the hypersquare $A$ is replaced by an arbitrary subset of $\R^{n-1}$.
Likewise, the upper bound constructions can easily be adapted to non-square geometries of $A$,
one only has to take care that the width of the tree-like structures does not extend beyond $A$ and that the construction has to be adapted to properly follow $\partial A$.
It seems also rather clear that the analysis can be adapted to a non-uniform mass distribution on $A\times\{0\}$ and $A\times\{L\}$
as long as the density is uniformly bounded and strictly bounded away from $0$.
Another geometric setting which may be relevant for various applications (thinking for instance of plants or blood vessel systems)
is where $\mu_0$ is replaced by a Dirac measure,
\begin{equation*}
\mu_0=\hdcodimone(A)\delta_0\,,\qquad\mu_1=\hdcodimone\restr(A\times\{1\})\,.
\end{equation*}
Here, the upper bound construction can be obtained in a similar manner as for our usual geometry:
One can choose elementary cells which are mapped into the standard rectangular elementary cells after applying the transformation
\begin{equation*}
(x_1,\ldots,x_n)\mapsto(\tfrac{x_1}{x_n},\ldots,\tfrac{x_{n-1}}{x_n},x_n)\,,
\end{equation*}
which is affine on each cross-section $\{x_n=t\}$
and which maps the cone $\{t(A\times\{1\})\,:\,t\in[0,1]\}$ onto the rectangle $A\times[0,1]$.
Such constructions provide the same energy scaling as for the geometry actually considered in this article.
For the lower bound, for instance on the branched transport model, we may again consider the mass flux $\mu_t$ on a generic cross-section $t$
and then bound the energy from below by the transport costs from $\mu_0$ to $\mu_t$,
$\inf_\chi\brTptEn^{\varepsilon,\mu_0,\mu_t}[\chi]\geq\Wdone(\mu_0,\mu_t)$,
and from $\mu_t$ to $\mu_1$,
$\inf_\chi\brTptEn^{\varepsilon,\mu_t,\mu_1}[\chi]\geq\Wdone(\mu_t,\mu_1)$.
Now, however, it is no longer possible to separately estimate $\Wdone(\mu_0,\mu_t)$ and $\Wdone(\mu_t,\mu_1)$.
Instead, a variation of Lemma\,\ref{thm:WoneBound} yields the necessary bound on
$\Wdone(\mu_0,\mu_t)+\Wdone(\mu_t,\mu_1)-\Wdone(\mu_0,\mu_1)$,
resulting in the same energy scaling.

\subsection{Properties characteristic for network optimisation}
Here we would like to briefly recapitulate the observed phenomena and relate network optimisation to other pattern formation problems.
Both urban planning and branched transport lead to branched pipe networks.
However, one can observe a difference between both models:
While for branched transport the excess energy over the limit case $\varepsilon=0$ is more or less evenly distributed across all branching levels,
this is true for urban planning only in three dimensions.
In two dimensions, the major contribution to the urban planning excess energy stems from the coarsest network structures,
while in dimensions higher than three the dominant energetic effect stems from the even distribution of mass at the boundary.

The urban planning model is closer in structure to classical pattern formation models
(such as martensite--austenite transformations \cite{KM92} or intermediate states in type-I superconductors \cite{ChCoKo08})
than branched transport.
Indeed, similarly to urban planning, the interfaces in many classical pattern formation models are directly penalised with cost $\varepsilon$,
while in branched transport the preference for shorter and fewer pipes is rather implicitly encoded in the mass flux cost.
The urban planning model is particularly related to patterns observed in intermediate states in type-I superconductors \cite{ChCoKo08} or compliance minimisation \cite{KoWi14}
in which also a structure is sought that has to conduct a flux (in one case a magnetic, in the other a force flux).
Urban planning is special, though, in that the conducting structures are really one-dimensional, while in the other mentioned problems they exhibit a finite width.
Intuitively, this also implies that the other problems have more design freedom so that their analysis should be a little more complicated.
And indeed, in our lower bound proofs it was relatively straightforward to explicitly characterise cross-sections through the network,
which is much more difficult in problems without one-dimensional structures.

\section{Acknowledgements}
This work was supported by the Deutsche Forschungsgemeinschaft (DFG), Cells-in-Motion Cluster of Excellence (EXC 1003-CiM), University of M\"unster, Germany.
B.W.'s research was supported by the Alfried Krupp Prize for Young University Teachers awarded by the Alfried Krupp von Bohlen und Halbach-Stiftung.

\bibliographystyle{alpha}
\bibliography{BrWi14}

\end{document}